\documentclass{amsart}

\usepackage{graphicx} 
\usepackage{hyperref}

\usepackage{amssymb}
\usepackage{amsmath}
\usepackage{amsfonts}
\usepackage{amsthm}
\usepackage{bm}
\usepackage{wasysym}
\usepackage{mathtools}
\usepackage{xcolor}
\usepackage{tikz}
\usepackage{tikz-cd}
\usepackage{enumitem}
\usepackage[style=alphabetic,maxnames=99,maxalphanames=99]{biblatex} 
\usepackage{array}
\usepackage{mathdots}

\DeclareMathOperator{\GL}{GL}
\DeclareMathOperator{\SL}{SL}
\DeclareMathOperator{\Sp}{Sp}

\DeclareMathOperator{\PSL}{PSL}
\DeclareMathOperator{\Orth}{O}
\DeclareMathOperator{\SOrth}{SO}
\DeclareMathOperator{\PSOrth}{PSO}
\DeclareMathOperator{\tr}{tr}
\DeclareMathOperator{\ad}{ad}
\DeclareMathOperator{\Ad}{Ad}

\DeclareMathOperator{\id}{id}

\DeclareMathOperator{\Span}{Span}
\DeclareMathOperator{\del}
{\partial}
\DeclareMathOperator{\Stab}{Stab}
\DeclareMathOperator{\sign}{sign}

\DeclareMathOperator{\Hom}{Hom}

\DeclareMathOperator{\Flag}{Flag}
\DeclareMathOperator{\IsoFlag}{IsoFlag}
\DeclareMathOperator{\Lag}{Lag}
\DeclareMathOperator{\Fix}{Fix}
\DeclareMathOperator{\pos}{pos}
\DeclareMathOperator{\Hir}{Hir}
\DeclareMathOperator{\Gr}{Gr}

\DeclareMathOperator{\SU}{SU}
\DeclareMathOperator{\diag}{diag}
\newcommand{\bldmth}[1]{\bm{\mathsf{#1}}}
\let\amsamp=&

\newtheorem{thm}{Theorem}[section]
\newtheorem{prop}[thm]{Proposition}
\newtheorem{lem}[thm]{Lemma}

\newtheorem{corr}[thm]{Corollary}

\newtheorem*{thma}{Theorem A}

\theoremstyle{definition}
\newtheorem{rem}{Remark}[section]
\newtheorem{defn}{Definition}[section]
\newtheorem{xmpl}{Example}[section]
\newtheorem*{prblm}{Problem}
\newtheorem*{qstn}{Question}

\newcounter{notes}

\let\oldtocsection=\tocsection

\let\oldtocsubsection=\tocsubsection

\let\oldtocsubsubsection=\tocsubsubsection

\renewcommand{\tocsection}[2]{\hspace{0em}\oldtocsection{#1}{#2}}
\renewcommand{\tocsubsection}[2]{\hspace{1em}\oldtocsubsection{#1}{#2}}
\renewcommand{\tocsubsubsection}[2]{\hspace{2em}\oldtocsubsubsection{#1}{#2}}

\title[Topology of domains of discontinuity]{Topology of domains of discontinuity for Anosov representations via circle actions}
\author{Mason Hart}
\date{}

\begin{document}

\maketitle

\tableofcontents

\section{Introduction}
For a rank $1$ semi-simple Lie group $G$ (with a maximal compact subgroup $K$), let $X=G/K$ be the symmetric space of $G$.  Given a finitely generated group $\Gamma$, a representation $\rho:\Gamma\to G$ is convex-cocompact if the orbit map 
\begin{align*}
    \tau_{\rho}:&\Gamma\to X
\intertext{given by} 
&\gamma \mapsto \rho(\gamma)K
\end{align*}
is a quasi-isometric embedding where the group $\Gamma$ is endowed with a word metric and the symmetric space $X$ is endowed with the $G$--invariant Riemannian metric coming from the Killing form. Because $G/K$ has strictly negative curvature, $\Gamma$ is necessarily a hyperbolic group. By properties of $\delta$--hyperbolic spaces, there is an associated embedding 
\[\del\tau_\rho:\del\Gamma\to\del X\]
where $\del \Gamma$ is the Gromov boundary of $\Gamma$ and $\del X$ is the visual boundary of $X$. The image of the embedding $\del\tau_\rho$ is known as the limit set $\Lambda_\rho$ of $\rho$ and is exactly the set of non-wandering points in $\del X$. A remarkable property of the limit set of a convex cocompact representation is that the complement $$\Omega_\rho=\del X\setminus \Lambda_\rho$$
is a cocompact domain of discontinuity for the action of $\Gamma$ through $\rho$.  \par 
For example, when $G=\PSL(2,\mathbb{C})$ and $\Gamma=\pi_1(S)$ is the fundamental group of a higher genus surface $S$, a convex-cocompact representation $\rho:\Gamma\to G$ is known as a quasi-Fuchsian representation. The limit set $\Lambda_\rho$ of a quasi-Fuchsian representation is a Jordan curve in $\del X=\del\mathbb{H}^3=\mathbb{CP}^1$. Thus, in this case, the domain of discontinuity $\Omega_\rho$ splits into two disjoint topological discs. When we quotient $\Omega_{\rho}$ by the action of $\Gamma$, we get a surface
\[\mathcal{W}_\rho=\rho(\Gamma)\setminus \Omega_\rho\]
which carries a complex projective structure (i.e.\ its covering space $\Omega_\rho$ is an open subset of $\mathbb{CP}^1$ with deck transformations in $\PSL(2,\mathbb{C})$). Topologically, $\mathcal{W}_{\rho}$ is diffeomorphic to $S\sqcup\bar{S}=\del(S\times[0,1])$ and can be thought of as the boundary of the compactification
\[\rho(\Gamma)\setminus (\mathbb{H}^3\cup\Omega_\rho)\cong S\times [0,1]\]
of the hyperbolic $3$--manifold $\rho(\Gamma)\setminus \mathbb{H}^3$. \par
In the past two decades, much progress has been made generalizing this picture and other aspects of Teichm\"uller theory to higher rank Lie groups. For a higher rank semi-simple Lie group $G$ and a hyperbolic group $\Gamma$, Guichard and Wienhard \cite{Guichard_2012} (based on the notion for $G=\SL(n,\mathbb{R})$, $\Gamma=\pi_1(S)$ defined by Labourie \cite{labourie2005anosovflowssurfacegroups}) defined the class of Anosov representations of $\Gamma$ into $G$. Anosov representations have become a focal point of the so-called higher Teichm\"uller theory, serving as the proper higher rank analog of convex cocompact representations. The particular levels of Anosov-ness are indexed by non-empty subsets $\theta$ of simple roots. Then, the set of $\theta$--Anosov representations $\rho:\Gamma\to G$ is an open subset of $\Hom(\Gamma,G)$. For any $\theta$, a $\theta$--Anosov representation always has discrete image and finite kernel. 

One important feature of convex cocompact representations shared by Anosov representations is their admittance of cocompact domains of discontinuity inside the visual boundary of the symmetric space $X$ of $G$.
One significant difficulty with the symmetric space $X$ of a higher rank Lie group $G$ is that the visual boundary is not homogeneous under the action of $G$. Instead, the visual boundary splits into $G$--orbits, each of which is isomorphic to a compact homogeneous space called a flag variety. The type of a flag variety $\mathcal{F}_\theta$ are also indexed by non-empty subsets of simple roots and an essential property of $\theta$--Anosov representation $\rho:\Gamma\to G$ is that it admits a $\rho$--equivariant embedding
\[\xi:\del\Gamma\to \mathcal{F}_\theta.\]
called the limit map of $\rho$. If $G$ is higher rank, the image of $\xi$ usually no longer contains every non-wandering point and the complement
\[\mathcal{F}_\theta\setminus \xi(\del\Gamma)\]
is not a domain of discontinuity \cite{kapovich2017dynamics}. However, \cite{Guichard_2012} showed that $\theta$--Anosov representations admit cocompact domains of discontinuity in a (possibly different) flag variety $\mathcal{F}_\eta$, using $G$--invariant bilinear forms on linear representations. Furthermore, Kapovich, Leeb and Porti \cite{kapovich2017dynamics} provided a systematic description of domains of discontinuity for $\theta$--Anosov representations in a flag variety $\mathcal{F}_\eta$ in terms of the combinatorics of the Weyl group $W$ of $G$. Both of these approaches ultimately realize the domain of discontinuity as the complement of a thickened version of the limit set. According to the method of \cite{kapovich2017dynamics}, the thickening procedure is characterized by a set of double cosets of $W$ known as a balanced ideal (denoted $I$). For an appropriate balanced ideal $I$, a  cocompact domain of discontinuity
\[\Omega_\rho^I\subseteq \mathcal{F}_\eta\]
is constructed. Then, if $\Gamma$ is torsion-free, we can consider the quotient manifold
\[\mathcal{W}_\rho^I=\rho(\Gamma)\setminus \Omega_\rho^I.\]
This quotient manifold $\mathcal{W}_\rho^I$ has natural $(G,\mathcal{F}_\eta)$--structure in the sense of Ehresmann and Thurston, generalizing the complex projective structure seen in the quasi-Fuchsian case. \par
\par
In general, the topology of $\mathcal{W}_\rho^I$ is harder to understand, compared to the quasi-Fuchian case. When $\Gamma$ is a surface group and the representation $\rho$ factors through a homomorphism 
\[\iota:\SL(2,\mathbb{R})\to G,\]
more can be said about the topology. Such a representation is called $\iota$--Fuchsian. For $\iota$-Fuchsian representations into complex Lie groups $G$, Dumas and Sanders calculated the homology of the domain $\Omega_\rho^I$ and $\mathcal{W}_\rho^I$. For $\iota$-Fuchsian representations,
Alessandrini, Maloni, Tholozan and Wienhard \cite{alessandrini2023fiberbundlesassociatedanosov} who proved that there is a smooth $\SL(2,\mathbb{R})$--equivariant fiber bundle projection
\[p:\Omega_\rho^I\to\mathbb{H}^2\]
where $\SL(2,\mathbb{R})$ acts on $\Omega_\rho^I$ through $\iota$.
Furthermore, they show that this implies that $\mathcal{W}_\rho^I$ is a smooth fiber bundle over $S$ with the fiber $M_\rho^I$ and the structure group of the bundle is $\SOrth(2)$. The fiber bundle projection $\mathcal{W}_\rho^I\to S$ is constructed non-explicitly and consequentially the fiber $M_\rho^I$ is in general difficult to determine. This theorem reduces the problem of  determining the topology of $\mathcal{W}_\rho^I$ when $\rho$ is $\iota$--Fuchsian to the following:
\begin{prblm}
    What is the topology of the fiber $M_\rho^I$? How does the structure group $\SOrth(2)$ act on $M_\rho^I$?
\end{prblm}
The main result of this paper is an answer for the case when the domain of discontinuity is inside a $3$--dimensional complex flag variety.
\begin{thma}
    Let $\rho:\Gamma\to G$ be an $\iota$--Fuchsian representation where $\iota:\SL(2,\mathbb{R})\to G$ is a homomorphism into a classical matrix group over $\mathbb{C}$. Assume that $\rho$ admits a cocompact domain of discontinuity $\Omega_\rho^I$, constructed from a balanced ideal $I$, inside a flag variety $\mathcal{F}_\eta$ with $\dim_{\mathbb{C}}\mathcal{F}_\eta=3$. Then, the fiber $M_\rho^I$ of $p:\Omega_\rho^I\to\mathbb{H}^2$ is as follows:
    \begin{enumerate}[label=\normalfont{(\roman*)}]
        \item If $\mathcal{F}_\eta=\Flag(\mathbb{C}^3)$, then $M_\rho^I$ is diffeomorphic to $(S^2\times S^2)\#(S^2\times S^2)$.
        \item If $\mathcal{F}_\eta=\mathbb{CP}^3$, then $M_\rho^I$ is diffeomorphic to $S^2\times S^2$.
        \item If $\mathcal{F}_\eta=\Lag(\mathbb{C}^4)$, then $M_\rho^I$ is diffeomorphic to $\mathbb{CP}{}^2\#\overline{\mathbb{CP}}{}^2$.
    \end{enumerate}
    The structure group action by $\SOrth(2)$ can be determined in each case, and each is equivalent to an algebraic group action on a Hirzebruch surface or a connected sum of such actions.
\end{thma}
It is shown that for each of the choices of $\mathcal{F}_\eta$ in 
Theorem A, there are two choices (up to conjugation) for the homomorphism $\iota:\SL(2,\mathbb{R})\to G$ so that $\rho=\iota\circ\phi$ admits a domain of discontinuity in $\mathcal{F}_\eta$. One option for $\iota$ is the irreducible (principal) representation $\iota_{\bldmth{irr}}$ into $G$ and thus this result described the topology of $\mathcal{W}_\rho^I$ for Hitchin representations in these cases. The other option for $\iota$ is a reducible representation $\iota_{\bldmth{red}}$, and its decomposition into irreducible $\SL(2,\mathbb{R})$--representations depends on $\mathcal{F}$. Specifically, if $\rho_n:\SL(2,\mathbb{R})\to\SL(n,\mathbb{R})$ is the $n$-dimensional unique irreducible representation, then the representations $\iota_{\bldmth{irr}}$ and $\iota_{\bldmth{red}}$ are as in the table below:
\[\begin{array}{c|c|c|c}
    G &\mathcal{F}_\eta & \iota_{\bldmth{irr}} & \iota_{\bldmth{red}}\\\hline
     \SL(3,\mathbb{C}) & \Flag(\mathbb{C}^3)&\rho_3 & \rho_2\oplus \rho_1 \\
     \SL(4,\mathbb{C}) &\mathbb{CP}^3&\rho_4 & \rho_2\oplus \rho_2 \\
     \Sp(4,\mathbb{C})&\Lag(\mathbb{C}^4)&\rho_4 & \rho_2\oplus\rho_1\oplus \rho_1. 
\end{array}\]
(See Proposition \ref{fullcases} below.) Let $M_{\bldmth{irr}}$ (resp. $M_{\bldmth{red}}$) be the fiber and $\mathcal{W}_{\bldmth{irr}}$ (resp. $\mathcal{W}_{\bldmth{red}}$) be the quotient manifold for the irreducible (resp. reducible) case. While for each flag variety $\mathcal{F}_\eta$ in Theorem A, the fibers $M_{\bldmth{irr}}$ and $M_{\bldmth{red}}$ are diffeomorphic, they have different structure group actions and thus, a priori the topology of the quotient manifolds $\mathcal{W}_{\bldmth{irr}}$ and $\mathcal{W}_{\bldmth{red}}$ could be different. This leads to the question:
\begin{qstn}
    For a flag variety $\mathcal{F}_\eta$ as in Theorem A, are $\mathcal{W}_{\bldmth{irr}}$ and $\mathcal{W}_{\bldmth{red}}$ diffeomorphic?
\end{qstn}
This question is connected to the broader problem of understanding the space of Anosov representations inside the character variety. Let $\operatorname{Anosov}_\theta(\Gamma,G)$ be the set of Anosov representations inside $\Hom(\Gamma,G)$. By \cite[Thm.\ 9.12]{Guichard_2012}, if $\rho,\rho':\Gamma\to G$ are representations that lie on the same connected component of $\operatorname{Anosov}_\theta(\Gamma,G)$, then for any balanced ideal $I$, the quotient manifolds $\mathcal{W}_\rho^I$ and $\mathcal{W}_{\rho'}^I$ are diffeomorphic. So the topology of the manifolds $\mathcal{W}_{\bldmth{irr}}$ and $\mathcal{W}_{\bldmth{red}}$ is related to the following question:
\begin{qstn} $\ $
    \begin{enumerate}
        \item Are $\rho_3$ and $\rho_2\oplus \rho_1$ in the same component of $\operatorname{Anosov}_{\{1,2\}}(\Gamma,\SL(3,\mathbb{C}))$? 
        \item Are $\rho_4$ and $\rho_2\oplus \rho_2$ in the same component of $\operatorname{Anosov}_{\{2\}}(\Gamma,\SL(4,\mathbb{C}))$? 
        \item Are $\rho_4$ and $\rho_2\oplus \rho_1\oplus \rho_1$ in the same component of $\operatorname{Anosov}_{\{1\}}(\Gamma,\Sp(4,\mathbb{C}))$?
    \end{enumerate}
\end{qstn}
\noindent
In the present work, we are not able to definitively answer either of the two questions posed above. \par
Prior to the result of \cite{alessandrini2023fiberbundlesassociatedanosov}, the fiber bundle structure of $\mathcal{W}_\rho^I$ had been noted and the topology of the fiber determined. For $G=\PSL(4,\mathbb{R})$, $\mathcal{F}_\eta=\mathbb{RP}^3$, \cite{guichard2007convexfoliatedprojectivestructures} identified the quotient manifold $\mathcal{W}_\rho^I$ with the unit tangent bundle of $S$ and showed that Hitchin representations are exactly the holonomies of convex foliated projective structures on the unit tangent bundle. For $G=\SOrth_0(2,n+1)$, $\mathcal{F}_{\eta}=\operatorname{Pho}(\mathbb{R}^{2,n+1})$, Collier, Tholozan, and Toulisse identified $\mathcal{W}_\rho^I$ as a fiber bundle over $S$ with fiber $\operatorname{Pho}(\mathbb{R}^{2,n})$ and characterized the maximal representations by the photon structure on $\mathcal{W}_\rho^I$. Finally for $G=\SL(2n,\mathbb{K})$, $\mathcal{F}_\eta=\mathbb{KP}^{2n-1}$, Alessandrini, Davalo and Li \cite{alessandrinidavaloli} characterized Hitchin representations identified (quasi-)Hithin representations as holonomies of projective structures on $\mathcal{W}_\rho^I$. \par
Some other results along these lines include the work of Nolte and Riestenberg ($G=\SL(3,\mathbb{R})$, $\mathcal{F}_\eta=\Flag(\mathbb{R}^3)$, see \cite{nolteriestenberg2024concavefoliatedflagstructures}), Davalo and Evans ($G=\SOrth_0(p,p+1),G_2'$, $\mathcal{F}_\eta=\operatorname{Ein}^{p-1,p},\operatorname{Pho}^\times$, see \cite{davalo2025geometricstructureseinsteinuniverse,davalo2025geometricstructuresg2surfacegroup}). Additionally, in a forthcoming work \cite{Hartexithcobordism}, the author studies the $h$-cobordism type of the fiber for any $(G,\mathcal{F}_\eta)$ by relating it to the thickening of a point inside $\mathcal{F}_\eta$. 

For $G=\Sp(4,\mathbb{C})$ and $\mathcal{F}_\eta=\Lag(\mathbb{C}^4)$, Alessandrini, Maloni, Tholozan and Wienhard \cite{alessandrini2023fiberbundlesassociatedanosov} use a non-smooth projection map and Freedman's classification of simply-connected $4$--manifolds to show the fiber $M_\rho^I$ is homeomorphic to $\mathbb{CP}{}^2\#\overline{\mathbb{CP}}{}^2$. Thus, part (iii) of Theorem A improves this identification to an equivariant diffeomorphism. Similarly, part (ii) also follows from the work of \cite{alessandrini2023fiberbundlesassociatedanosov}, but in this paper we determine the structure group explicitly. The final major distinction between the present paper and previous work is that it applies representations factoring through a reducible representation in addition to quasi-Hitchin representations. \par
The essential tool for the proof is Fintushel's classification of simply-connected closed $4$--manifolds with effective $S^1$--actions \cite{FintushelSimply4Circle}. Intuitively, the classification says that the circle action on a simply-connected $4$--manifold is determined by the action around points with non-trivial stabilizers. This information is encoded in a labeled graph which is embedded in the orbit space of the action. Fintushel's original paper works in the topological setting and thus only shows the classification up to $S^1$--equivariant homeomorphism, although folklorically the classification holds in the smooth setting as well. We confirm this folklore result by principally reworking Fintushel's original proof with explicitly smooth methods. In order to make the application of this classification to smooth actions more tractable, we use an alternative labeling of the graph due to Jang \cite{Jang_2018,jang2023circleactionsoriented4manifolds}. We call this classifying graph with the labeling due to Jang the tangential weight graph of the $S^1$--action.  \par
Once we have this classification theorem for smooth $S^1$--action on simply-connected $4$--manifolds, we develop a method to determine the tangential weight graph for the action of $\SOrth(2)$ on $M_\rho^I$ by relating it to the action on the ambient flag variety $\mathcal{F}$. We determine the tangential weight graphs for a family of $S^1$--actions on Hirzebruch surfaces. By showing that the tangential weight graph of each fiber $M_\rho^I$ agrees with that of a Hirzebruch surface (or connected sum of Hirzebruch surfaces), we conclude Theorem A.
\par
We would like to note that the methods of this paper can also be applied to $3$--dimensional complex flag varieties for product groups. Specifically, the fiber for $G=(\SL(2,\mathbb{C}))^3$ with domain of discontinuity in $\mathcal{F}=(\mathbb{CP}^1)^3$ is diffeomorphic $\#^3(S^2\times S^2)$, and the fiber for $G=\SL(2,\mathbb{C})\times\SL(3,\mathbb{C})$ and $\mathcal{F}=\mathbb{CP}^1\times\mathbb{CP}^2$ is diffeomorphic to $(S^2\times S^2)\#\mathbb{CP}{}^2\#\overline{\mathbb{CP}}{}^2$. For brevity, we do not discuss the details for these cases and instead focus on simple matrix groups. \\ \\
\textbf{Outline of the paper.\ }In Section \ref{sect:liethry}, the prerequisite Lie theory is explained. In Section \ref{sec:anosov}, we define Anosov representations and their domains of discontinuity, and we conclude with Proposition \ref{fullcases} which lists the cases for which our method can be applied. In Section \ref{sec:circleclass}, the classification of smooth $S^1$--actions on simply-connected closed $4$--manifolds in terms of tangential weight graphs is shown (Theorem \ref{SmoothClassification}). In Section \ref{sec:TWD4flags}, we describe how to calculate the tangential weight graph of a flag variety and relate it to the tangential weight graph of the fiber (Proposition \ref{SameTWG}). In Section \ref{sec:calculations4fibers}, we apply these results to determine the fiber in each case:
\begin{itemize}
    \item For $\mathcal{F}=\Flag(\mathbb{C}^3)$, see Corollary \ref{fiberflgirr} and Corollary \ref{fiberflgred}.
    \item For $\mathcal{F}=\mathbb{CP}^3$, see Corollary \ref{fiberprjirr} and Corollary \ref{fiberprjred}.
    \item For $\mathcal{F}=\Lag(\mathbb{C}^4)$, see Corollary \ref{fiberlagirr} and Corollary \ref{fiberlagred}.
\end{itemize}
\vspace{5pt}
\textbf{Ackowledgements.\ }
First, the author would like to express his gratitude for the endless help and guidance from his advisor, Sara Maloni. Her support has been indispensable to the author's progress as a graduate student and the completion of this paper as a part of his thesis. The author also wants to thank Filippo Mazzoli for his useful discussions in the different stages of this research project, as well as Nicolas Tholozan and Max Riestenberg for their insightful conversations. \par 
 The author was partially supported by the U.S. National Science Foundation grant DMS-1848346 (NSF CAREER). Finally, the author would also like to acknowlegde the support of the Institut Henri Poincar\'e (UAR 839 CNRS-Sorbonne Universit\'e) and LabEx CARMIN (ANR-10-LABX-59-01).

\section{Lie theory and flag varieties}\label{sect:liethry}
We will now explicate the necessary Lie theory that will be needed to properly define and describe Anosov representations and their domain of discontinuity.
In this paper, we will only consider representations into complex Lie groups. However, the relevant structure theory for real and complex Lie groups will be both be used. Which coefficients are being used when will always be explicitly stated. \par
A Lie algebra $\mathfrak{g}$ over $\mathbb{K}=\mathbb{R}$, $\mathbb{C}$ is simple if it has no proper ideals over $\mathbb{K}$. A connected Lie group $G$ over $\mathbb{K}$ is simple if its associated Lie algebra $\mathfrak{g}=T_eG$ is simple and it has finite center. A Lie algebra $\mathfrak{g}$ over $\mathbb{K}$ is semi-simple if it has no proper abelian ideals over $\mathbb{K}$. The semi-simplicity of $\mathfrak{g}$ is equivalent to the non-degeneracy of the Killing form $$\kappa(X,Y)=\tr(\ad_X\circ\ad_Y)$$ where $$\ad_X(Y)=[X,Y]$$
and trace is computed using coefficents in $\mathbb{K}$.
A Lie group over $\mathbb{K}$ will be called semi-simple if its Lie algebra $\mathfrak{g}$ is semi-simple. \par
Important examples of semi-simple complex Lie groups are the classical groups
\begin{align*}
    \SL(n,\mathbb{C})&=\left\{A\in M_{n}(\mathbb{C}):\det A=1\right\}\\
    \SOrth(n,\mathbb{C})&=\{A\in\SL(n,\mathbb{C}):A^tJA=J\}\\
    \Sp(2n,\mathbb{C})&=\{A\in \SL(n,\mathbb{C}):A^t J'A=J'\}
\end{align*}
where $$J=\begin{bmatrix}
    & & 1\\ & \iddots & \\ 1 & &
\end{bmatrix} \ \ \ \ \ \ \ \ J'=\begin{bmatrix}
    & J \\ -J &
\end{bmatrix}$$
are a choice of invertible symmetric and skew-symmetric matrices. Choosing a different symmetric (or skew-symmetric matrix) for $J$ ( or $J'$) will produce a conjugate subgroup of $\SL(n,\mathbb{C})$.\par
The corresponding classical complex Lie algebras are
\begin{align*}
    \mathfrak{sl}_n&=\left\{A\in M_{n}(\mathbb{C}):\tr A=0\right\}\\
    \mathfrak{so}_n&=\{A\in\mathfrak{sl}_n:A^tJ+JA=0\}\\
    \mathfrak{sp}_n&=\{A\in \mathfrak{sl}_{2n}:A^t J'+J'A=0\}
\end{align*}
each of which of simple. Additionally, there are also the exceptional simple Lie algebras $\mathfrak{g}_2$, $\mathfrak{f}_4$, $\mathfrak{e}_6$, $\mathfrak{e}_7$, $\mathfrak{e}_8$ which are the Lie algebras of the exceptional groups $G_2$, $F_4$, $E_6$, $E_7$, and $E_8$. For specific models of these exceptional groups, see \cite[Chapter 5]{Onishchik1993LieGA}. \par
Any finite dimensional simple Lie algebra over $\mathbb{C}$ is the direct sum of classical or exceptional simple Lie algebras. This means, for our purposes, it will be sufficient to consider products of classical groups and exceptional groups rather than the whole class of semi-simple Lie groups over $\mathbb{C}$. \par
For the remainder of the section, let $G$ be a product of classical and exceptional groups over $\mathbb{C}$ with Lie algebra $\mathfrak{g}$. A subalgebra $\mathfrak{h}$ of $\mathfrak{g}$, which is maximal among subalgebras whose $\ad$--action on $\mathfrak{g}$ is simultaneously diagonalizable, is called a Cartan subalgebra. Any two choices of Cartan subalgebra are conjugate under the adjoint action of $G$. \par
We can choose a maximal compact subgroup $K$ of $G$ with Lie alegbra $\mathfrak{k}$ so that $\mathfrak{t}=\mathfrak{k}\cap\mathfrak{h}$ is a maximal abelian subspace of $\mathfrak{k}$ and $\mathfrak{a}=\mathfrak{p}\cap\mathfrak{h}$ is a maximal abelian subspace of $\mathfrak{p}=\mathfrak{k}^{\perp_{\kappa_{\mathbb{R}}}}$, the orthogonal complement of $\mathfrak{k}$ with respect to the real Killing form of $\mathfrak{g}$. In this case, the dimension $\dim_{\mathbb{R}}\mathfrak{a}=\dim_{\mathbb{C}}\mathfrak{h}$ is called the rank of $G$. Then, $\mathfrak{g}$ has a restricted root space decomposition. A restricted root is a non-zero real linear functional $\alpha\in\mathfrak{a}^\ast$ such that its root space
\[\mathfrak{g}_{\alpha}=\{X\in\mathfrak{g}:\ad_{H}(X)=\lambda(H)X\text{ for }H\in\mathfrak{a}\}\]
is non-zero. Then, the set of restricted roots
\[\Sigma=\{\alpha\in\mathfrak{a}^\ast:\mathfrak{g}_{\alpha}\neq 0,\alpha\neq 0\}\]
is an abstract root system in $\mathfrak{a}^\ast$.
 We can choose a subset $\Delta\subseteq \Sigma$ which is a basis of $\mathfrak{a}^\ast$. We call elements of $\Delta$ simple roots. A root $\alpha\in \Sigma$ is called positive (resp. negative) if $\alpha$ is written with non-negative (resp. non-positive) coefficients using the basis $\Delta$. Let $\Sigma^+$ (resp. $\Sigma^-$) denote the set of positive (resp. negative) roots. Then, the properties of root systems imply that $\Sigma=\Sigma^+\sqcup\Sigma^-$. 
 We define the standard Borel subgroup $B$ (with respect to our choice of $K$, $\mathfrak{a}$, and $\Delta$) to be the analytic subgroup (over $\mathbb{R}$) with Lie algebra 
\begin{align*}
    \mathfrak{b}=\mathfrak{h}\oplus\bigoplus_{\alpha\in\Sigma^+}\mathfrak{g}_\alpha
\end{align*}
Note that our choices of $K$ and $\mathfrak{a}$ ensure that each root space $\mathfrak{g}_{\alpha}$ is a complex linear subspace of $\mathfrak{g}$ so $\mathfrak{g}$ is, in fact, a subalgebra of $\mathfrak{g}$ over $\mathbb{C}$ and $B$ is a complex Lie subgroup of $G$. \par 
A proper subgroup $P$ of $G$ containing $B$ is called a standard parabolic subgroup and is always of the form $P_{\theta}$ where $P_\theta$ is the analytic subgroup of $G$ with Lie algebra
\[\mathfrak{p}_{\theta}=\mathfrak{b}\oplus\bigoplus_{\alpha\in\Sigma^-\cap\operatorname{Span}(\Delta\setminus\theta)}\mathfrak{g}_{\alpha}\]
for some non-empty subset $\theta\subseteq\Delta$. The non-empty subset $\theta$ of $\Delta$ is called the type of $P$. In particular, $P_{\Delta}=B$ is the Borel subgroup.\par
For a standard parabolic subgroup $P_{\theta}$, we are interested in the associated $G$--homogeneous space 
\[\mathcal{F}_{\theta}=G/P_{\theta}\]
which we call the (generalized) flag variety for $G$ of type $\theta$. The homogeneous space $\mathcal{F}_{\Delta}=G/B$ is called the full flag variety of $G$. A stabilizer of a point in $\mathcal{F}_\theta$ (or equivalently a conjugate subgroup of $P_\theta$) is called a parabolic subgroup. \par
Note that the Killing form restricted to $\mathfrak{a}$ gives a positive definite inner product.
The Weyl group $W$ of $\Sigma$ is the subgroup of $\Orth(\mathfrak{a},\kappa_{\mathbb{R}})$ generated by the reflections $s_\alpha$ across the hyperplanes $\ker\alpha$ for $\alpha\in\Sigma$. In fact, the simple reflections $S=\{s_\alpha:\alpha\in\Delta\}$ is a generating set for $W$ which form a Coxeter system. We have an induced action of $W$ on $\mathfrak{a}^+$, preserving the root system $\Sigma$ given by
\[w\cdot \alpha=\alpha\circ w^{-1}\].
The an element $w_0$ is the longest element of $W$ with respect to the generating set $S$ if and $w_0\Delta=-\Delta$. We define the opposition involution $\nu:\Delta\to\Delta$ by
\[\nu(\alpha):=-w_0\alpha.\]
We note that $W$ is isomorphic to $N_{K}(\mathfrak{a})/Z_K(\mathfrak{a})$
where
\begin{align*}
    N_K(\mathfrak{a})&=\{k\in K:\Ad_k(\mathfrak{a)=\mathfrak{a}}\}\\
    Z_K(\mathfrak{a})&=\{k\in K:\Ad_k=\id_{\mathfrak{a}}\}
\end{align*}
which has a natural action on $\mathfrak{a}$ by isometries as follows:
\[\left(kZ_K(\mathfrak{a})\right)\cdot a=\Ad_k(a).\]
Note the Lie algebra of $Z_K(\mathfrak{a})$ is actually $\mathfrak{t}$ and
\[Z_K(\mathfrak{a})=B\cap K.\]
\par
A pair of parabolic subgroups $(P,P')$ are called opposite if they are $G$--conjugate, for some $\theta\subseteq \Delta$, a pair $(P_{\theta},P^{-}_{\theta})$ where
$P^-_{\theta}$ is the analytic subgroup of $G$ with Lie algebra
\[\mathfrak{p}_\theta^-=\mathfrak{h}\oplus\bigoplus_{\alpha\in\Sigma^-}\mathfrak{g}_{\alpha}\ \oplus\bigoplus_{\alpha\in\Sigma^+\cap\Span(\Delta\setminus \theta)}.\]
It should be noted that $P^-_\theta$ is a parabolic subgroup conjugate to $P_{\nu(\theta)}$. Thus, we consider $\mathcal{F}_{\theta}$ and $\mathcal{F}_{\nu(\theta)}$ opposite flag varieties. \par
Finally, we would like to recall the Cartan projection of $G$. Let $\overline{\mathfrak{a}}^+=\{X\in\mathfrak{a}:\alpha(X)\geq 0\text{ for }\alpha\in\Sigma^+\}$ denote the Weyl chamber. Then, for any $g\in G$, there are $k,\ell\in K$ and a unique $X\in\overline{\mathfrak{a}}^+$ such that
\[g=k\exp(X)\ell.\]
Then, the Cartan projection $\mu:G\to\overline{a}^+$ is the function which assigns $X$ to $g$. We will see that the Cartan projection plays a central role in the definition of an Anosov representation.
\subsection{Classical groups and flag varieties}\label{classgrpflg}
Now, we will discuss what the flag varieties for classical groups look like. We will only go into detail for $G=\SL(n,\mathbb{C})$ and $G=\Sp(2n,\mathbb{C})$ as we will see later these groups are the only cases where our method will be relevant.
\subsubsection{Flag varieties for special linear group}
We now restrict to $G=\SL(n,\mathbb{C})$ with $\mathfrak{g}=\mathfrak{sl}_n$. We will consider the Cartan subalgebra of (traceless) diagonal matrices with complex entries:
\[\mathfrak{h}=\{\operatorname{diag(\lambda_1,\hdots,\lambda_n)}\in\mathfrak{sl}_n\}.\]
The real Killing form of $\mathfrak{sl}_n$ is
\[\kappa_{\mathbb{R}}(X,Y)=4n\operatorname{Re}(\tr XY)\]
where $\tr XY$ is the trace of the matrix product of $X,Y\in\mathfrak{sl}_n$. We can then choose maximal compact subgroup
\[K=\SU(n)=\{A\in\SL(n,\mathbb{C}):A^\ast A=I_n\}\]
where $A^\ast$ is the complex transpose of $A$. Then, 
\begin{align*}
    \mathfrak{k}&=\mathfrak{su}_n=\{A\in\mathfrak{sl}_n:A^\ast+A=0\\
    \mathfrak{p}&=\mathfrak{k}^{\perp_{\mathbb{R}}}=\{A\in\mathfrak{sl}_n:A=A^\ast\}.
\end{align*}
Then, we have 
\begin{align*}
    \mathfrak{t}&=\mathfrak{p}\cap\mathfrak{h}=\{\operatorname{diag}(\lambda_1,\hdots,\lambda_n)\in\mathfrak{sl}_n:\lambda_1,\hdots,\lambda_n\in i\mathbb{R}\}\\
    \mathfrak{a}&=\mathfrak{p}\cap\mathfrak{h}=\{\operatorname{diag}(\lambda_1,\hdots,\lambda_n)\in\mathfrak{sl}_n:\lambda_1,\hdots,\lambda_n\in\mathbb{R}\}
\end{align*}
Then, the restricted roots of $G$ are 
\[\Sigma=\{\epsilon_j-\epsilon_k:j,k\text{ distinct integers in }[1,n]\}\]
where $\epsilon_j:\mathfrak{a}\to\mathbb{R}$ is the linear functional defined by
\[\epsilon_j(\operatorname{diag}(\lambda_1,\hdots,\lambda_n))=\lambda_j.\]
Then, each root space $\mathfrak{g}_{\epsilon_j-\epsilon_k}$ is $2$--dimensional with root vectors $E_{jk}$, $iE_{jk}$ where $E_{jk}$ is the $n\times n$ matrix with all entries zero except for a one in the entry in the $j$--th row and $k$--th column. The rank of $\SL(n,\mathbb{C})$ is $n-1$ and we will use the basis of simple roots
\[\Delta=\{\epsilon_j-\epsilon_{j+1}:j\in\{1,\hdots,n-1\}\}.\]
Then,
\[\Sigma^+=\{\epsilon_j-\epsilon_k:1\leq j<k\leq n\}\ \ \ \ \Sigma^-=\{\epsilon_j-\epsilon_k:1\leq k<j\leq n\}\]
are the sets of positive and negative (restricted) roots. \par
This means the Borel subalgebra
\[\mathfrak{b}=\mathfrak{h}\oplus\bigoplus_{1\leq j<k\leq n}\mathbb{C}E_{jk}\]
is exactly the set of traceless upper triangular matrices. Thus, the Borel subgroup $B$ for $G=\SL(n,\mathbb{C})$ is the subgroup of upper triangular matrices. \par 
Any type $\theta\subseteq\Delta$ is of the form $\{\epsilon_{d_1}-\epsilon_{d_1+1},\hdots,\epsilon_{d_\ell}-\epsilon_{d_{\ell}+1}\}$ for a unique subsequence $d_1,\hdots, d_{\ell}$ of $1,\hdots,n-1$. Let $d_0=0$ and $d_{\ell+1}=n$. Then, $P_{\theta}$ is the subgroup of block upper triangular matrices with blocks of sizes $d_1-d_{0},d_2-d_1,\hdots,d_{\ell+1}-d_{\ell}$ on the diagonal, or more formulaically:
\[P_{\theta}=\left\{\begin{bmatrix}
    A_{1,1} & \cdots & A_{1,\ell+1} \\
    & \ddots & \vdots\\ & & A_{\ell+1,\ell+1}
\end{bmatrix}\in\SL(n,\mathbb{C}):A_{j,k}\in M_{d_j-d_{j-1},d_{k}-d_{k-1}}(\mathbb{C})\right\}\]
\\ For $G=\SL(n,\mathbb{C})$, we will often identify the type of a parabolic subgroup by the sequence $d_1,\hdots,d_{\ell}$ as above rather than a subset of simple roots, e.g.\ $P_{1,2}=P_{\{\epsilon_1-\epsilon_2,\epsilon_2-\epsilon_3\}}$. If $n\leq 10$, then we can omit the delimiters and refer to $P_{1,2}$ as $P_{12}$ without ambiguity. \par
Let $e_1,\hdots,e_n$ be the standard basis for $\mathbb{C}^n$. Note that for $k\in\{1,\hdots, n-1\}$, $P_{k}$ is exactly the stabilizer of the $k$--plane $\langle e_1,\hdots,e_k\rangle$ inside $G$. Thus, we see that the flag variety $\mathcal{F}_k$ can be identified with the Grassmanian $\Gr_k(\mathbb{C}^n)$ of $k$--planes inside $\mathbb{C}^n$. \par
Likewise, the subgroup
\[P_{d_1,\hdots,d_\ell}=P_{d_1}\cap\cdots P_{d_\ell}\]
stabilizes the tuple of subspaces
\[E^{\bullet}_{d_1,\hdots,d_{\ell}}=(\langle e_1,\hdots,e_{d_1}\rangle,\hdots,\langle e_1,\hdots, e_{d_{\ell}}\rangle).\]
The $G$--orbit of $E^\bullet_{d_1,\hdots,d_{\ell}}$ inside $\Gr_{d_1}(\mathbb{C}^n)\times\cdots\times\Gr_{d_{\ell}}(\mathbb{C}^n)$ is the set of flags in $\mathbb{C}^n$ of signature $d_1,\hdots,d_n$. A flag $F^\bullet$ in an $n$--dimensional complex vector space $V$ of signature $d_1,\hdots d_{\ell}$ is a tuple
\[F^\bullet=(F^{d_1},\hdots, F^{d_{\ell}})\]
where $F^{d_j}$ is a $d_j$--dimensional subspace of $V$ and
\[F^{d_1}\subsetneq F^{d_2}\subsetneq\cdots \subsetneq F^{d_{\ell}}.\]
Thus, $E^\bullet_{d_1,\hdots,d_{\ell}}$ is a flag inside $\mathbb{C}^n$ with $E^{d_j}_{d_1,\hdots,d_{\ell}}=\langle e_1,\hdots,e_{d_{j}}\rangle$, which we call the standard flag of signature $d_1,\hdots,d_{\ell}$. More generally, we say a flag $F^{\bullet}$ of signature $d_1,\hdots, d_\ell$ has a flag basis $v_1,\hdots,v_n\in\mathbb{C}^{n}$ if $v_1,\hdots,v_n$ is a basis such that
\[F^{d_j}=\langle v_1,\hdots,v_{d_j}\rangle\text{ for }j\in\{1,\hdots,\ell\}.\]
In particular, the standard basis is a flag basis for the standard flags of various signatures. A flag of signature of $1,\hdots,n$ is called a full flag.
\par
Then, the flag variety $\mathcal{F}_{d_1,\hdots,d_{\ell}}$ can be identified with the set $\Flag_{d_1,\hdots,d_\ell}(\mathbb{C}^n)$ of flags in $\mathbb{C}^n$ of signature $d_1,\hdots,d_\ell$. These are the prototypical examples of flag varieties and explain the origin of the descriptor ``flag''.\par
Next, we consider the Weyl group of $G=\SL(n,\mathbb{C})$. Note that for each simple root $\alpha=\epsilon_j-\epsilon_{j+1}$,
\[\ker\alpha=\{\diag(\lambda_1,\hdots,\lambda_n)\in\mathfrak{sl}_n:\lambda_j=\lambda_{j+1}\}\]
and the simple reflection $s_{\alpha}$ across $\ker\alpha$ is the transposition
\[\diag(\lambda_1,\hdots,\lambda_{j-1},\lambda_j,\lambda_{j+1},\lambda_{j+2},\hdots,\lambda_n)\mapsto\diag(\lambda_1,\hdots,\lambda_{j-1}\lambda_{j+1},\lambda_j,\lambda_{j+2},\hdots,\lambda_n)\]
and thus the Weyl group $W$ which is generated by $\{s_\alpha:\alpha\in\Delta\}$ is isomorphic to the symmetric group $S_n$ acting on $\mathfrak{a}$ by permutation of diagonal entries. Furthermore, the longest element of $W=S_n$ (which should send $\Sigma^+$ to $\Sigma^-$) is the order-reversing permutation $j\mapsto n-j+1$. Then,
\[\nu(\epsilon_{j}-\epsilon_{j+1})=-(w_0\cdot(\epsilon_j-\epsilon_{j+1}))=-(\epsilon_{n-j+1}-\epsilon_{n-j})=\epsilon_{n-j}-\epsilon_{n-j+1}\]
is the opposition involution. This means that $\Gr_{k}(\mathbb{C}^n)$ and $\Gr_{n-k}(\mathbb{C}^n)$ are opposite flag varieties. More generally, $\Flag_{d_1,\hdots,d_\ell}(\mathbb{C}^n)$ and $\Flag_{d_1',\hdots,d'_\ell}(\mathbb{C}^n)$ are opposite flag varieties if and only if
\[d_j+d_{\ell+1-j}'=n\text{ for }j\in\{1,\hdots,\ell\}.\]
\par
Finally, we want to discuss the Cartan projection for $\SL(n,\mathbb{C})$ which derives from the singular value decomposition.
\begin{xmpl}\label{CartanHyperbolic}
    Any $g\in G$ has a singular value decomposition
    \[g=ka\ell\]
    where $k,\ell\in \SU(n)$ and $a$ is a diagonal matrix $\diag(\sigma_1,\hdots,\sigma_n)$ with real entries $\lambda_1\geq\lambda_2\geq\cdots\geq\lambda_n>0$ in descending order. The diagonal entry $\sigma_j$ is known as the $j$--th singular value of $g$. Then, we have
    \[g=k\exp(X)\ell\]
    where $X\in\overline{\mathfrak{a}^+}=\{\diag(\lambda_1,\hdots,\lambda_n)\in\mathfrak{sl}_n:\lambda_1\geq\lambda_2\geq\cdots\geq\lambda_n\}$ is the diagonal matrix
    \[X=\diag(\log\sigma_1,\hdots,\log\sigma_n).\]
    Thus, the Cartan projection $\mu:G\to\overline{\mathfrak{a}^+}$ outputs the logarithms of the singular values, i.e.
    \[g\mapsto\diag(\log\sigma_1(g),\hdots,\log\sigma_n(g)).\]
    \par
    We now specialize to $\SL(2,\mathbb{C})$.  For $g\in \SL(2,\mathbb{C})$, assume that
    \[g=k\begin{bmatrix}
        e^{\frac{\delta}{2}} & \\ & e^{-\frac{\delta}{2}}
    \end{bmatrix}\ell\]
    for some $\delta\in\mathbb{R}_{>0}$ and $k,l\in\SU(2)$. Let $O\in\mathbb{H}^3$ be the point whose stabilizer in $\SL(2,\mathbb{C})$ is $\SU(2)$. Then,
    \[d_{\mathbb{H}^3}(O,g\cdot O)=d_{\mathbb{H}^3}\left(O,\begin{bmatrix}
        e^{\frac{\delta}{2}} & \\ & e^{-\frac{\delta}{2}}
    \end{bmatrix}\cdot O\right)=\delta=\mu(g)\]
    Thus, we see that for $G=\SL(2,\mathbb{C})$, the Cartan projection of $g$ agrees with the displacement of the origin in $\mathbb{H}^3$ by $g$.
\end{xmpl}
Now that we have covered $G=\SL(n,\mathbb{C})$, we will discuss $G=\Sp(n,\mathbb{C})$. 
\subsubsection{Isotropic flag varieties for symplectic groups}
Consider $\mathbb{C}^{2n}$ as vectors written in a basis $e_{1},\hdots,e_n,e_{-n},\hdots,e_{-1}$. We treat
\[\Sp(2n,\mathbb{C})=\{A\in \SL(n,\mathbb{C}):A^t J'A=J'\},\]
where
$$J=\begin{bmatrix}
    & & 1\\ & \iddots & \\ 1 & &
\end{bmatrix} \ \ \ \ \ \ \ \ J'=\begin{bmatrix}
    & J \\ -J &
\end{bmatrix},$$
as the linear transformations of $\mathbb{C}^{2n}$ that preserve the non-degenerate bilinear form $\omega:\mathbb{C}^{2n}\to\mathbb{C}$ where
$$\omega(e_{j},e_{-k})=\delta_{jk}, \ \ \omega(e_{-j},e_{k})=-\delta_{jk}\ \ \text{ for }j\in\{1,\hdots,n\}, k\in\{-n,-n+1,\hdots,n\}.$$
Recall that the Lie algebra $\mathfrak{g}$ of $G$ is
\[\mathfrak{sp}_{2n}=\left\{A\in\mathfrak{sl}_{2n}:A^tJ'+J'A=0\right\}.\]
More explicitly, this is the set of matrices
\[\mathfrak{sp}_{2n}=\left\{\begin{bmatrix}
    a_{1,1} & \cdots & a_{1,n} & a_{1,-n} & \cdots & a_{1,-1}\\ \vdots & \ddots & \vdots & \vdots & \iddots & \vdots \\
    a_{n,1} & \cdots  & a_{n,n} & a_{n,-n} & \cdots  & a_{1,-n} \\ a_{-n,1} & \cdots & a_{-n,n}  & -a_{n,n} & \cdots & -a_{1,n} \\ \vdots & \iddots & \vdots & \vdots & \ddots & \vdots \\ a_{-1,1} & \cdots & a_{-n,1} & -a_{n,1} & \cdots & -a_{1,1}
\end{bmatrix}\in \mathfrak{sl}_{2n}\right\}.\]
In other words, elements of $\mathfrak{sp}_{2n}$ has the block form $\left[\begin{smallmatrix}
    A & B\\ C & D
\end{smallmatrix}\right]$ where $B$, $C$ are symmetric along the antidiagonal and $A$ and $--D$ are reflections across the antidiagonal, that is
\begin{align*}
    B&=JB^tJ ,& C&=JC^tJ, & D=-JA^tJ.
\end{align*}
We consider the the Cartan subalgebra $\mathfrak{h}$ of anti-symmetric (with respect to the anti-diagonal) diagonal matrices, that is
\[\mathfrak{h}=\{\diag(\lambda_1,\hdots,\lambda_n,-\lambda_n,\hdots,-\lambda_1):\lambda_1,\hdots,\lambda_n\in\mathbb{C}\}\]
We choose the maximal compact subgroup $K=\Sp(n)=\Sp(2n,\mathbb{C})\cap U(2n)$. The Lie algebra of this subgroup is
\[\mathfrak{k}=\left\{\begin{bmatrix}
    A & B \\ -B^\ast & -JA^tJ
\end{bmatrix}:A,B\in M_{n}(\mathbb{C}),\ A^\ast=-A,\ B=JB^tJ\right\}\]
which has the orthogonal complement
\[\mathfrak{p}=\mathfrak{k}^{\perp_\mathbb{R}}=\left\{\begin{bmatrix}
    A & B \\ B^\ast & -JA^tJ
\end{bmatrix}:A,B\in M_{n}(\mathbb{C}),\ A^\ast=A,\ B=JB^tJ\right\}.\]
We have that
\begin{align*}
    \mathfrak{t}&=\mathfrak{k}\cap\mathfrak{h}=\{\diag(\lambda_1,\hdots,\lambda_n,-\lambda_n,\hdots,-\lambda_{1}):\lambda_1,\hdots,\lambda_n\in i\mathbb{R}\}\\
    \mathfrak{a}&=\mathfrak{p}\cap\mathfrak{h}=\{\diag(\lambda_1,\hdots,\lambda_n,-\lambda_n,\hdots,-\lambda_{1}):\lambda_1,\hdots,\lambda_n\in \mathbb{R}\}.
\end{align*}
Then, the restricted roots of $G$ are
\[\{\pm(\epsilon_{j}+\epsilon_k),\pm(\epsilon_j-\epsilon_k):1\leq j<k\leq n\}\cup\{\pm2\epsilon_j:1\leq j\leq n\}\]
where $\epsilon_j:\mathfrak{a}\to\mathbb{R}$ is the linear functional
\[\epsilon_j(\diag(\lambda_1,\hdots,\lambda_n,-\lambda_n,\hdots,-\lambda_1))=\lambda_j\]
for $j\in\{1,\hdots,n\}$.\par 
Let $E_{j,k}$ be the $2n\times 2n$ matrix such that $E_{j,k}e_{\ell}=\delta_{\ell k}e_{j}$ for $j,k,\ell\in\{1,\hdots,n,-n,\hdots,-1\}$. For distinct $j,k\in\{1,\hdots,n\}$, the root space $\mathfrak{g}_{\epsilon_j+\epsilon_k}$ has basis $E_{-j,k}+E_{-k,j}$.
Each root space is $2$--dimensional over $\mathbb{R}$. For any $1\leq j<k\leq n$,
\begin{align*}
    &\mathfrak{g}_{\epsilon_{j}+\epsilon_k} & &\text{ has the basis } &  &E_{-j,k}+E_{-k,j},\ iE_{-j,k}+iE_{-k,j},\\
    &\mathfrak{g}_{-\epsilon_{j}-\epsilon_k} & &\text{ has the basis } & &E_{j,-k}+E_{k,-j},\ iE_{j,-k}+iE_{k,-j},\\
    &\mathfrak{g}_{\epsilon_{j}-\epsilon_k} & & \text{ has the basis } & &E_{j,k}-E_{-k,-j},\ iE_{j,k}-iE_{-k,-j},\\
    &\mathfrak{g}_{\epsilon_{k}-\epsilon_j} & & \text{ has the basis } & &E_{k,j}-E_{-j,-k},\ iE_{k,j}-iE_{-j,-k}.
\end{align*}
Furthermore, for $j\in\{1,\hdots,n\}$
\begin{align*}
    &\mathfrak{g}_{2\epsilon_j} & &\text{ has the basis } &  &E_{-j,j},\ iE_{-j,j},\\
    &\mathfrak{g}_{-2\epsilon_j} & &\text{ has the basis } &  &E_{j,-j},\ iE_{j,-j}.
\end{align*}
The rank of $\Sp(2n,\mathbb{C})$ is $n$ and we will choose the basis of simple roots to be
\[\Delta=\{\epsilon_{1}-\epsilon_2,\hdots,\epsilon_{n-1}-\epsilon_{n}\}\cup\{2\epsilon_n\}.\]
Then, 
\begin{align*}
    \Sigma^+&=\{\epsilon_j\pm\epsilon_k:1\leq j< k\leq n\}\cup \{2\epsilon_j:1\leq j\leq n\} \\
    \Sigma^-&=\{\epsilon_j\pm\epsilon_k:1\leq k< j\leq n\}\cup \{-2\epsilon_j:1\leq j\leq n\}
\end{align*}
are the sets of positive and negative roots. \par
We observe that the positive roots are the exactly the roots whose root spaces consist of strictly upper triangular matrices. It follows that the standard Borel subgroup of $\Sp(2n,\mathbb{C})$ is the intersection of $\Sp(2n,\mathbb{C})$ with the standard Borel subgroup of $\SL(2n,\mathbb{C})$. In other words,
\[B=\{A\in\Sp(2n,\mathbb{C}):A\text{ is upper triangular}\}.\]
\par
One consequence of this relationship between the Borel subgroups is that the full flag variety $\mathcal{F}_{\Delta}$ for $\Sp(2n,\mathbb{C})$ is isomorphic to the $\Sp(2n,\mathbb{C})$--orbit of the standard full flag $E^{\bullet}_{full}$ inside $\Flag_{full}(\mathbb{C}^{2n})$. Using our alternative indexing by $1,\hdots,n,-n,\hdots,-1$, the standard full flag looks like
\[E^\bullet_{full}=(E^{1}_{full},\hdots,E_{full}^n,E^{-n}_{full},\hdots,E^{-1}_{full})\]
so that, for $r\in\{1,\hdots,n\}$,
\begin{align*}
    E^r_{full}&=\langle e_1,\hdots,e_r\rangle\\
    E^{-r}_{full}&=\langle e_1,\hdots, e_n,e_{-n},\hdots,e_{-r}\rangle
\end{align*}
It should be noted that the bilinear form $\omega$ is constructed so that
$$E^{-r}=(E^{r-1})^{\perp_{\omega}}$$
for each $r\in\{1,\hdots,n\}$. This relationship would be preserved by the action of $\Sp(2n,\mathbb{C})$. In fact, this characterizes the $\Sp(2n,\mathbb{C})$--orbit of $E^\bullet_{full}$.\par
More explicitly, an element $F^\bullet$ of $\Flag_{full}(\mathbb{C}^{2n})$ is said to be an extended $\omega$--isotropic full flag if it satisfies
\[F^{-r}=(F^{r-1})^{\perp_{\omega}}.\]
Then, the full flag variety $\mathcal{F}_{\Delta}$ can be identified with the subset of extended $\omega$--isotropic flags inside $\Flag_{full}(\mathbb{C}^{2n})$. Note that this relation between $F^{r-1}$ and $F^{-r}$ for a $\omega$--isotropic full flag $F^{\bullet}$ means that the subspaces $F^{-n},\hdots, F^{-1}$ are redundant information. Therefore, we can more simply identify the $\omega$--isotropic full flag $F^\bullet$, by the half-full flag
\[(F^1,F^{2},\hdots, F^n)\in\Flag_{1,\hdots, n}(\mathbb{C}^{2n}).\]
For each $r\in\{1,\hdots,n\}$, $F^r$ is isotropic, i.e.\ $F^r\subseteq (F^{r})^{\perp_{\omega}}$.
Through this projection from $\Flag_{full}(\mathbb{C}^{2n})$ to $\Flag_{1,\hdots,n}(\mathbb{C}^{2n})$, we get the subset of flags of signature $1,\hdots, n$ consisting of isotropic subspaces. 
\begin{defn}
    A flag $F^\bullet$ in $\mathbb{C}^{2n}$ of signature $d_1,\hdots,d_\ell$ such that 
    $$F^{d_\ell}\subseteq (F^{d_\ell})^{\perp_\omega}$$ 
    is called an $\omega$--isotropic flag in $\mathbb{C}^{2n}$.  \par
    
    The full flag variety of $\Sp(2n,\mathbb{C})$ can be realized as the subset of $\omega$--isotropic flags of signature $1,\hdots,n$, which we denote by 
\[\IsoFlag_{full}(\mathbb{C}^{2n},\omega)=\left\{(F^1,\hdots,F^{n})\in\Flag_{1,\hdots,n}(\mathbb{C}^{2n}):F^n= (F^n)^{\perp_{\omega}}\right\}.\]
\end{defn}
\par
Then, as was the case for the partial flag variety of $\SL(n,\mathbb{C})$, the other (partial) flag varieties of $\Sp(2n,\mathbb{C})$ come from projecting an $\omega$--isotropic full flag to its constituent isotropic subspaces. Let $d_1,\hdots,d_{\ell}$ be a subsequence of $1,\hdots,n$. Then, we define the $\omega$--isotropic flag variety of signature $d_1,\hdots,d_\ell$ to be the subset
\[\IsoFlag_{d_1,\hdots,d_\ell}(\mathbb{C}^{2n},\omega)=\{(F^{d_1},\hdots,F^{d_{\ell}})\in\Flag_{d_1,\hdots,d_\ell}(\mathbb{C}^{2n}):F^{d_{\ell}}\subseteq (F^{d_\ell})^{\perp_{\omega}}\}.\]
If we label the simple roots in $\Delta$ by
\[\alpha_j=\begin{cases}
    \epsilon_j-\epsilon_{j+1} & 1\leq j\leq n-1\\
    2\epsilon_n & j=n
\end{cases}\]
then $\IsoFlag_{d_1,\hdots,d_{\ell}}(\mathbb{C}^{2n},\omega)$ is the (partial) flag variety for $\Sp(2n,\mathbb{C})$ of type $\{\alpha_{d_1},\hdots,\alpha_{d_\ell}\}$. Let $P$ be the parabolic subgroup of $\SL(2n,\mathbb{C})$ consisting of matrices of block triangular form with blocks of sizes
$$d_1,d_2-d_1,\hdots,d_\ell-d_{\ell-1},n-d_{\ell},n-d_{\ell},d_\ell-d_{\ell-1},\hdots,d_{2}-d_{1},d_1.$$
Let $Q$ be the parabolic subgroup of $\Sp(2n,\mathbb{C})$) which is the stabilizer of the standard ($\omega$--isotropic) flag $E_{d_1,\hdots,d_\ell}^\bullet$ of signature $d_1,\hdots,d_\ell$. Then, we have
\[Q=P\cap\Sp(2n,\mathbb{C}).\]
 \par
Noteworthy isotropic partial flag varieties include
\begin{align*}
    \mathbb{CP}^{2n-1}&=\IsoFlag_{1}(\mathbb{C}^{2n},\omega)\\
    \Lag(\mathbb{C}^{2n})&=\IsoFlag_n(\mathbb{C}^{2n},\omega).
\end{align*}
The Lagrangian Grassmanian $\Lag(\mathbb{C}^{2n})$ is the set of maximal isotropic subspaces in $\mathbb{C}^{2n}$. Later, we will consider the topology of domains of discontinuity inside $\Lag(\mathbb{C}^{4})$. \par
Finally, we discuss the Weyl group of $\Sp(2n,\mathbb{C})$. We introduce the shorthand, for $\lambda_1,\hdots,\lambda_n\in\mathbb{C}$,
\[\operatorname{asdiag}(\lambda_1,\hdots,\lambda_n)=\diag(\lambda_1,\hdots,\lambda_n,-\lambda_n,\hdots,-\lambda_1).\]
Then, for $j\in\{1,\hdots,n-1\}$,
\[\ker\alpha_j=\ker(\epsilon_{j}-\epsilon_{j+1})=\{\operatorname{asdiag}(\lambda_1,\hdots,\lambda_n)\in\mathfrak{a}:\lambda_j=\lambda_{j+1}\}\]
and so the simple reflection $s_{\alpha_j}$ acts on $\mathfrak{a}$ by the transposition
\[\operatorname{asdiag}(\lambda_1,\hdots,\lambda_n)\mapsto\operatorname{asdiag}(\lambda_1,\hdots,\lambda_{j-1},\lambda_{j+1},\lambda_j),\lambda_{j+2},\hdots,\lambda_n).\]
On the other hand, we have that
\[\ker\alpha_n=\ker\epsilon_n=\{\operatorname{asdiag}(\lambda_1,\hdots,\lambda_n)\in\mathfrak{a}:\lambda_n=0\}\]
and so the simple reflection act on $\mathfrak{a}$ by the coordinate negation
\[\operatorname{asdiag}(\lambda_1,\hdots,\lambda_n)\mapsto\operatorname{asdiag}(\lambda_1,\hdots,\lambda_{n-1},-\lambda_n).\]
\par
Then, the Weyl group $W$ generated by the simple reflections $s_{\alpha_1},\hdots,s_{\alpha_n}$ is isomorphic to the group of signed permutation matrices:
\[W\cong\left\{DP\in \Orth(n):D\text{ diagonal with entries }\pm 1,P\text{ permutation matrix} \right\}.\]
Given a signed permutation matrix $DP$, there are $c_1,\hdots,c_n\in\{\pm1\}$ and $\sigma\in S_n$ such that $D=\diag(c_1,\hdots, c_n)$ and $P$ is the permutation matrix $\sigma$. 
We denote the set of signed permutations as
\[S^{\pm}_n=\left\{\sigma^\pm: \sigma^\pm\text{ is a permutation of }\{1,\hdots,n,-n,\hdots,-1\}\text{ with }\sigma^\pm(-j)=-\sigma^\pm(j)\right\}.\]
Then, the sign permutation matrix $DP$ corresponds to the signed permutation $$\sigma^\pm(j)=\sign(j)c_{\sigma(j)}\sigma(|j|)\text{ for }j\in\{1,\hdots,n,-n,\hdots,-1\}$$.
We will signify the signed permutation $\sigma^\pm$ by
\[\begin{pmatrix} c_{\sigma(1)}\sigma(1) & \cdots & c_{\sigma(n)}\sigma(n)\end{pmatrix}.\]
For example, the matrix $\left[\begin{smallmatrix}
     &1 &  \\ -1 & & \\ & & 1
\end{smallmatrix}\right]$ corresponds to the signed permutation
\[\begin{pmatrix}
    -2 & 1 & 3
\end{pmatrix}.\]
Note that the group operation on $W=S_n^\pm$ is not simply composition of permutations. Specifically, for $c_1,\hdots,c_n,f_1,\hdots,f_n\in\{\pm1\}$ and $\sigma,\tau\in S_n$, we have if
\begin{align*}
    \sigma^\pm(j)&=\sign(j)c_{\sigma(|j|)}\sigma(|j|) & \tau^\pm(j)&=\sign(j)f_{\sigma(|j|)}\tau(|j|)
\end{align*}
then,
\[(\sigma^\pm\cdot \tau^\pm)(j)=\sign(j)c_{\sigma(\tau(|j|)}f_{\tau(|j|)}\sigma(\tau(|j|))\]
for $j\in\{1,\hdots,n,-n,\hdots,-1\}$.
\par
The longest element of $W=S_{n}^\pm$ is the signed permutation corresponding to negation. That is,
\[w_0=\begin{pmatrix}
    -1 & -2 & \cdots & -n+1 & -n
\end{pmatrix}.\]
This means also that the opposition involution $\nu(\alpha)=-\alpha\circ w_0$ acts trivially on $\Delta$. Thus, every flag variety is self-opposite. 
\subsubsection{Isotropic flag varieties for special orthogonal groups}
As stated before, the detailed Lie theory for special orthogonal groups will not be important for the rest of the paper, but we now give a brief outline.
For $G=\SOrth(n,\mathbb{C})$, the analysis would be mostly similar to $G=\Sp(2n,\mathbb{C})$. Give $\mathbb{C}^{n}$ the standard basis $e_1,\hdots,e_n$. The group $\SOrth(n,\mathbb{C})$ preserves the bilinear form $\omega:\mathbb{C}^n\times\mathbb{C}^n\to\mathbb{C}$ given by
\[\omega(e_{j},e_{k})=\delta_{j+k,n+1}.\]\par
We separate into the cases $n=2p$ and $n=2p+1$ for $p\geq 1$.
The ranks of $\SOrth(2p,\mathbb{C})$ and $\SOrth(2p+1,\mathbb{C})$ are both $p$. The flag varieties for $\SOrth(2p+1,\mathbb{C})$ can be identified with the sets of $\omega$--isotropic flags of various signatures; the full flag variety being 
\[\IsoFlag_{full}(\mathbb{C}^{2p+1},\omega)=\{F^\bullet\in\Flag_{1,\hdots,p}(\mathbb{C}^{2p+1}):F^p\subseteq (F^p)^{\perp_{\omega}}\}.\]
The flag variety of smallest dimension is the smooth quadric hypersurface in $\mathbb{CP}^{2p}$, that is
\[\operatorname{Quad}_{2p}=\{\ell\in\mathbb{CP}^{2p}:\ell\subseteq \ell^{\perp_{\omega}}\}\]
which has complex dimension $2p-1$.
For $G=\SOrth(2p,\mathbb{C})$, there is peculiar subtlety. Note that the $p$--planes
\[E^{p^+}=\langle e_1,\hdots,e_{p-1},e_{p}\rangle,\ \ \ E^{p^-}=\langle 1,\hdots,e_{p-1},e_{p+1}\rangle\]
are both maximal isotropic subspaces. However, the permutation matrix corresponding to the transposition of $p$ and $p+1$, which would send $E^{p+}$ to $E^{p-}$, has determinant $-1$. In fact, if one analyzes the block structure of the stabilizer of $E^{p^+}$ inside $\Orth(2p,\mathbb{C})$, you find that any $A\in\Orth(2p,\mathbb{C})$ that preserves $E^{p^+}$ is orientation-preserving, and thus any $A\in \Orth(2p,\mathbb{C})$ which sends $E^{p^+}$ to $E^{p^-}$ does not reside in $\SOrth(2p,\mathbb{C})$. The conclusion is that the set of $\omega$--isotropic $p$--planes in $\mathbb{C}^{2p}$ is not a $\SOrth(2p,\mathbb{C})$--homogeneous space and instead has two orbits corresponding to $E^{p^+}$ and $E^{p^-}$. We call elements in the orbit of $E^{p^\pm}$ isotropic $p^\pm$--planes. The full flag variety $\Flag_{full}(\mathbb{C}^{2p},\omega)$ for $\SOrth(2p,\mathbb{C})$ is then the set of tuples $(F^1,\hdots,F^{p-1},F^{p^+},F^{p-})$ where $(F^1,\hdots, F^{p-1})$ is a flag of signature $1,\hdots,p-1$, $F^{p^+}$ and $F^{p^-}$ are $\omega$--isotropic $p^+$--planes and $p^-$--planes respectively, and $F^{p-1}=F^{p^+}\cap F^{p^-}$. The partial flag varieties come from projecting onto the constituent subspaces, including the isotropic $p^+$--planes and $p^-$--planes. \par
Similar to $\SOrth(2p+1,\mathbb{C})$, the flag variety of smallest dimension for $\SOrth(2p,\mathbb{C})$ is the smooth quadric hypersurface in $\mathbb{CP}^{2p-1}$, that is
\[\operatorname{Quad}_{2p-1}=\{\ell\in\mathbb{CP}^{2p-1}:\ell\subseteq \ell^{\perp_{\omega}}\}\]
which has complex dimension $2p-2$. \par
This completes our discussion of flag varieties for classical groups.
\subsection{Bruhat decomposition and relative positions}\label{BruhatRelative}
Now, we will describe the Bruhat decomposition and how it provides a way to measure the relative positions of flags. Specifically, we will see that the relationship between elements of flag varieties for $G$ is a combinatorial one, captured by the Bruhat order on the Weyl group. Fix $G$ to be some product of classical and exceptional groups over $\mathbb{C}$. For more discussions of relative position functions for flag varieties, see \cite[Section 3]{dumas-sanders} and \cite[Section 3]{CarvajalesSteckerHomogeneous}. We will mostly follow the notation of \cite{Stecker2018BalancedIA}.
\begin{defn}
    Let $\theta,\eta\subseteq\Delta$ be flag types for $G$. Let $(\Pi,\leq)$ be a poset. A relative position function for $\mathcal{F}_\theta$  and $\mathcal{F}_\eta$ with values in $\Pi$ is a function $\pos:\mathcal{F}_\theta\times \mathcal{F}_{\eta}\to\Pi$ such that $$\pos(x_1,y_1)\leq \pos(x_2,y_2)$$
    if and only if there is the $G$--orbit of $(x_1,y_1)$ is contained in the closure of the $G$--orbit of $(x_2,y_2)$, that is
    \[G\cdot(x_1,y_1)\subseteq\overline{G\cdot(x_2,y_2)}.\]
\end{defn}
An immediate consequence of this definition is that $\pos(x_1,y_1)=\pos(x_2,y_2)$ if and only if $(x_1,y_1)$ and $(x_2,y_2)$ are in the same $G$--orbit inside $\mathcal{F}_{\theta}\times\mathcal{F}_{\eta}$. Thus, $\Pi$ parametrizes the $G$--congruence classes of pairs in $\mathcal{F}_\theta\times\mathcal{F}_{\eta}$, with the order $\leq$ encoding the degeneration of congruence classes into less generic congruence classes.\par
Let $B$ be the standard parabolic subgroup. Recall that for each $w\in W$, we can choose a $\tilde{w}\in N_K(\mathfrak{a})$ such that $w=\Ad_{\tilde{w}}|_{\mathfrak{a}}$. Then, the Bruhat decomposition of $G$ is
\[G=\bigsqcup_{w\in W}B\tilde{w}B\]
where $B\tilde{w}B$ denotes the double coset of $\tilde{w}$, which depends only on $w$. \par
\begin{defn}
    Let $f\in\mathcal{F}_{\Delta}$ be the standard full flag so that the stabilizer of $f$ in $G$ is $B$. Then, for any pairs of flags $x,y\in\mathcal{F}_{\Delta}$, there are $g_1,g_2$ such that 
\[x=g_1f, \ \ \ y=g_2f.\] 
Then, we define the function $\pos_{\Delta,\Delta}:\mathcal{F}_\Delta\times\mathcal{F}_\Delta\to W$ so that $\pos_{\Delta,\Delta}(x,y)$ is the unique $w\in W$ such that
\[Bg_1^{-1}g_2B=B\tilde{w}B.\]
\end{defn}
In order to make, $\pos_{\Delta,\Delta}:\mathcal{F}_{\Delta}\times\mathcal{F}_{\Delta}\to W$ a relative position function, we endow the Weyl group $W$ with the Bruhat order where $w_1\leq w_2$ if there are some reduced expression $s_1\cdots s_\ell$ in the generating set of simple reflections and subsequence $j_1,\hdots,j_k$ such that 
\begin{align*}
    w_1&=s_{j_1}\cdots s_{j_k} & &\text{and} & w_2&=s_1\cdots s_\ell.
\end{align*}
Note that the longest element $w_0\in W$ is necessarily the maximal element with respect to the Bruhat order. \par
For $G=\SL(n,\mathbb{C})$, where $\mathcal{F}_{\Delta}=\Flag_{full}(\mathbb{C}^{n})$, the relative position function takes values in $S_n$. The following proposition elucidates this further.
\begin{prop}\label{RelPosSLn}
    Given any two full flags $F^{\bullet},H^\bullet$ in $\mathbb{C}^{n}$, the following hold:
    \begin{enumerate}[label=\normalfont{(\roman*)}]
        \item There is a flag basis $v_1,\hdots,v_n$ for $F^\bullet$ and a permutation $\sigma\in S_{n}$ such that $v_{\sigma(1)},\hdots,v_{\sigma(n)}$ is a flag basis for $H^\bullet$.
        \item The relative position $\pos_{\Delta,\Delta}(F^{\bullet},H^{\bullet})$ is represented by the permutation $\sigma$.
    \end{enumerate}
\end{prop}
\begin{proof}
    For part (i), we first note that for each fixed $j\in\{1,\hdots,n\}$, the function $D_j:\{0,1,\hdots,n\}\to\{0,1,\hdots,j\}$ given by
    \[D_j(k)=\dim(F^k\cap H^{j})\]
    is non-decreasing and more strongly
    \[0\leq D_j(k)-D_{j}(k-1)\leq 1\text{ for all }k\in\{1,\hdots,n\}.\]
    Along with the fact that
    \begin{align*}
        D_j(n)&=\dim(\mathbb{C}^n\cap H^j)=j & &\text{and} & D_j(0)=\dim(\{0\}\cap H^j)=0,
    \end{align*}
    we conclude that for each $j$, the set
    \[K_{j}=\{k:1\leq k\leq n,D_j(k)-D_j(k-1)=1\}\]
    has exactly $j$ elements. For any $k\in K_j$, the fact that
    \[D_j(k)-D_j(k-1)=\dim(F^k\cap H^j)-\dim(F^{k-1}\cap H^{j})=1\]
    means we can choose some $$v_k\in (F^k\setminus F^{k-1})\cap H^j.$$ In particular, this would imply that $$v_k\in (F^{k}\setminus F^{k-1})\cap H^{j+1}$$ and thus 
    \[D_{j+1}(k)-D_{j+1}(k-1)=\dim(F^{k}\cap H^{j+1})-\dim(F^{k-1}\cap H^{j+1})=1.\]
    In other words, $k\in K_{j+1}$ and generalizing, we deduce that there is a chain of containment
    \[\varnothing\neq K_1\subsetneq K_{2}\subsetneq \cdots\subsetneq K_{n-1}\subsetneq K_n=\{1,\hdots,n\}.\]
    Then, it must be that there is a permutation $\sigma\in S_n$ such that
    \[K_j=\{\sigma(1),\hdots,\sigma(j)\}\] for $j\in\{1,\hdots,n\}$. Then, we have that
    \begin{align*}
        v_1\in F^1\setminus F^0,\ v_2\in F^2\setminus F^1, \ \hdots, \ v_n\in F^n\setminus F^{n-1}
    \end{align*}
    and thus $v_1,\hdots,v_n$ is a flag basis for $F^\bullet$. \par
    If $k=\sigma(j)$, then $k\in K_j$ and thus 
    \[v_k\in (F^{k}\setminus F^{k-1})\cap H^j.\]
    On the other hand, $k\notin K_{j-1}$ and thus
    \[F^k\cap H^{j-1}=F^{k-1}\cap H^{j-1}\]
    which implies $v_k\notin H^{j-1}$. Thus, for $j\in\{1,\hdots,n\}$
    \[v_k=v_{\sigma(j)}\in H^j\setminus H^{j-1}.\]
    In other words,
    \begin{align*}
        v_{\sigma(1)}\in H^1\setminus H^0,\ v_{\sigma(2)}\in H^2\setminus H^1, \ \hdots, \ v_{\sigma(n)}\in H^n\setminus H^{n-1}
    \end{align*}
    and thus $v_{\sigma(1)},\hdots,v_{\sigma(n)}$ is a flag basis for $H^\bullet$. This completes the proof of (i). \par
    For (ii), we normalize $v_1,\hdots,v_n$ so that the matrix
    \[A=\begin{bmatrix}
        v_1 & \cdots & v_n
    \end{bmatrix}\]
    forms a matrix with determinant $1$. Then, if $P_\sigma$ is the signed permutation matrix so that
    \[P_\sigma e_j=\begin{cases}
        \sign(\sigma)e_{\sigma(1)} & j=1\\
        e_{\sigma(j)}
    \end{cases},\]
    Then, we have that
    \[F^{\bullet}=AE^\bullet_{full}, \ \ \ H^\bullet=AP_\sigma E^{\bullet}_{full}.\]
    This means the double coset corresponding to $\pos_{\Delta,\Delta}(F^\bullet,H^{\bullet})$ is
    \[BA^{-1}AP_{\sigma}B=BP_{\sigma}B.\]
    The signed permutation matrix $P_{\sigma}$ is constructed to be an element of $N_{K}(\mathfrak{a})$ representing $\sigma$. Therefore, $$\pos_{\Delta,\Delta}(F^{\bullet},H^{\bullet})=\sigma.$$
\end{proof}
Recall that we consider as $G=\Sp(2n,\mathbb{C})$ as the linear transformations of 
\[\mathbb{C}^{2n}=\langle e_1,\hdots,e_{n},e_{-n},\hdots,e_{-1}\rangle\]
preserving the anti-symmetric bilinear form
$$\omega(e_{j},e_{-k})=\delta_{jk}, \ \ \omega(e_{-j},e_{k})=-\delta_{jk}\ \ \text{ for }j\in\{1,\hdots,n\}, k\in\{-n,-n+1,\hdots,n\}.$$
The full flag variety for $\Sp(2n,\mathbb{C})$ is $\Flag_{full}(\mathbb{C}^{2n},\omega)$ and the Weyl group is the signed permutation group $S^{\pm}_{n}$. Then, we have an analogous version of Proposition \ref{RelPosSLn} for $\Sp(2n,\mathbb{C})$.
\begin{prop}\label{RelPosSp2n}
    Given two full $\omega$--isostropic flags $F^\bullet,H^\bullet$, extend them to full flags in $\mathbb{C}^{2n}$ so that
    \[F^{-r}=(F^{r-1})^{\perp_{\omega}},\ \ \ \ H^{-r}=(H^{r-1})^\perp\]
    for $r\in\{1,\hdots,n\}$. The following holds:
    \begin{enumerate}[label=\normalfont{(\roman*)}]
        \item There is a flag basis $v_1,\hdots,v_{n},v_{-n},\hdots,v_{-1}$ for $F^\bullet$ and a signed permutation $\sigma^\pm\in S^\pm_n$ such that $v_{\sigma^\pm(1)},\hdots,v_{\sigma^\pm(n)},v_{\sigma^\pm(-n)},\hdots,v_{\sigma^\pm(-1)}$ is a flag basis for $H^\bullet$.
        \item The relative position $\pos_{\Delta,\Delta}(F^\bullet,H^\bullet)$ is represented by $\sigma^\pm$.
    \end{enumerate}
\end{prop}
\begin{proof}
    For part (i), we know from Proposition \ref{RelPosSLn}, there is a flag basis $v_1,\hdots,v_n,v_{-n},\hdots,v_{-1}$ for $F^\bullet$ and some permutation 
    \[\sigma^\pm:\{1,\hdots,n,-n,\hdots,-1\}\to\{1,\hdots,n,-n,\hdots,-1\}\]
    such that $v_{\sigma^\pm(1)},\hdots,v_{\sigma^\pm(n)},v_{\sigma^\pm(-1)},\hdots,v_{\sigma^\pm(-n)}$ is a flag basis for $H^\bullet$. Abusing notation, we will temporarily omit the superscript ($\pm$). We need to show that
    \[\sigma(-j)=-\sigma(j)\text{ for }j\in\{1,\hdots,n,-n,\hdots,-1\}.\]
    To do so, assume for the sake of contradiction that there is some $j_1\in\{1,\hdots,n,-n,\hdots,-1\}$ such that $\sigma(j_1)\neq-\sigma(-j_1)$. Because we can exchange $\sigma^{-1}(-\sigma(j_1))$ for $j_1$, we can assume without loss of generality that $\sigma(j_1)$ is positive. Then, we have
    \[v_{\sigma(-j_1)}\in H^{-j_1}\setminus H^{-j_1-1}=\langle v_{\sigma(1)},\hdots,v_{\sigma(j+1-1)}\rangle^{\perp_{\omega}}\setminus \langle v_{\sigma(1)},\hdots,v_{\sigma(j_1)}\rangle^{\perp_{\omega}}\]
    so that $\omega(v_{\sigma(j_1)},v_{\sigma(-j_1)})\neq 0$. From the $\omega$--isotropic condition on $H^\bullet$, it must be that
    $$\sigma(j_1)+\sigma(-j_1)\leq 0$$.
    But we also know that $$\sigma(j)\neq -\sigma(-j)$$ and thus
    \[\sigma(j_1)<-\sigma(-j_1).\]
    Then, then $j_2=\sigma^{-1}(-\sigma(-j_1))$ also satisfied $\sigma(j_2)>0$ and $\sigma(j_2)\neq -\sigma(-j_2)$ and thus the same argument would produce an infinite sequence $j_1,j_2,\hdots$ such that
    \[0<\sigma(j_1)<\sigma(j_2)<\cdots\]
    which contradicts finiteness of $\{1,\hdots,n\}$. Thus,
    \[\sigma(j)=-\sigma(-j)\text{ for all }j\in\{1,\hdots,n,-n,-\hdots,-1\},\]
    that is, the permutation
    \[\sigma^\pm:\{1,\hdots,n,-n,\hdots,-1\}\to\{1,\hdots,n,-n,\hdots,-1\}\]
    is a signed permutation. \par
    For part (ii), let $P$ be the $2n\times 2n$ matrix so that
    \[Pe_j=\begin{cases}
        \sign(\sigma(j))e_{\sigma(j)} & j\in\{1,\hdots,n\}\\ e_{\sigma(j)} & j\in\{-n,\hdots,-1\}
    \end{cases}\]
    Then, a direct computation shows that $P$ is in $\Sp(2n,\mathbb{C})$, specifically $N_K(\mathfrak{a})$. Furthermore, $PZ_{K}(\mathfrak{a})$ corresponds to the signed permutation $\sigma^\pm$.
\end{proof}
\par We have seen how Bruhat decomposition provides a relative position function for full flag varieties. For partial flag varieties, we we have
\[G=\bigsqcup_{w\in W_{\theta,\eta}}P_{\theta}\tilde{w}P_\eta\]
where $W_{\theta,\eta}$ is the double coset space
\[W_{\theta,\eta}=(P_\theta\cap N_K(\mathfrak{a}))\setminus N_K(\mathfrak{a})/(P_\eta\cap N_K(\mathfrak{a})).\]
Through the identification, $W=N_K(\mathfrak{a})/Z_K(\mathfrak{a})=\langle s_{\alpha}:\alpha\in\Delta\rangle$, we actually have
\begin{align*}
    (N_K(\mathfrak{a})\cap P_\theta)/Z_K(\mathfrak{a})&=\langle s_{\alpha}:\alpha\in\Delta\setminus \theta\rangle & (N_K(\mathfrak{a})\cap P_\eta)/Z_K(\mathfrak{a})&=\langle s_{\alpha}:\alpha\in\Delta\setminus \eta\rangle
\end{align*}
so that $W_{\theta,\eta}$ can also be identified with the double coset space
\[W_{\theta,\eta}=W_\theta\setminus W/W_\eta \]
where $W_\theta=\langle\{s_{\alpha}:\alpha\in\Delta\setminus\theta\}\rangle$ (and $W_\eta$ defined analogously).
This gives us a combinatorial description of the relative positions of partial flags as we had for full flags. To complete this description, we see in the next proposition how the Bruhat order on $W$ descends to $W_{\theta,\eta}$. \par
 Let $(x_1,y_1)$ and $(x_2,y_2)$ be pairs of flags in $\mathcal{F}_{\theta}\times\mathcal{F}_{\eta}$. Let $\pi_{\theta}:\mathcal{F}_{\Delta}\to\mathcal{F}_{\theta}$ and $\pi_\eta:\mathcal{F}_{\Delta}\to\mathcal{F}_{\eta}$ be the projection maps and choose lifts $\tilde{x}_1,\tilde{y}_1,\tilde{x}_2$ and $\tilde{y}_2$. Then, the relative positions
    \[w_1=\pos_{\Delta,\Delta}(\tilde{x}_1,\tilde{y}_1), \ \ \ w_2=\pos_{\Delta,\Delta}(\tilde{x}_2,\tilde{y}_2)\]
    determine the double cosets $[w_1],[w_2]\in W_{\theta,\eta}$, which are independent of the choice of lifts. For each double coset $[w]\in W_{\theta,\eta}$, there is a unique element $\min[w]\in[w]$ which is minimal with respect to the Bruhat order (See \cite{bjorner2006combinatorics}).
\begin{prop}
     Given $(x_1,y_1)$, $(x_2,y_2)$ and $[w_1],[w_2]$ as above, the following are equivalent
    \begin{enumerate}
        \item[\normalfont{(a)}] The $G$--orbits of $(x_1,y_1)$ and $(x_2,y_2)$ satisfy
        \[G\cdot (x_1,y_1)\subseteq \overline{G\cdot(x_2,y_2)}.\]
        \item[\normalfont{(b)}] The double cosets $[w_1],[w_2]$ satisfy
            \[\min[w_1]\leq \min[w_2]\]
            where $\leq$ is the Bruhat order for $W$.
    \end{enumerate}
\end{prop}
\begin{proof}
    Assuming (a), there is a sequence $\{g_n\}$ in $G$ such that $g_n\cdot (x_2,y_2)\to (x_1,y_1)$. We can choose our lift $(\tilde{x}_2,\tilde{y}_2)\in\mathcal{F}_\Delta\times\mathcal{F}_\Delta$ of $(x_2,y_2))$ so that
    \[\pos_{\Delta,\Delta}(\tilde{x}_2,\tilde{y}_2)=\min[w_2].\]
    Then, using the compactness of $\mathcal{F}_\Delta\times\mathcal{F}_\Delta$ and passing to a subsequence, we can assume that $g_n\cdot(\tilde{x}_2,\tilde{y}_2)$ converges to a flag $(\tilde{x}_1,\tilde{y}_1)$ in $\mathcal{F}_\Delta\times\mathcal{F}_\Delta$. The $G$--equivariance and continuity of $\pi_\theta\times\pi_\eta$ implies that $(\tilde{x}_1,\tilde{y}_1)$ is, in fact, a lift of $(x_1,y_1)$. Then, we can conclude that
    \[G\cdot (\tilde{x}_1,\tilde{y}_1)\subseteq\overline{G\cdot(\tilde{x}_2,\tilde{y}_2)}.\]
    Then, we conclude that
    \[\min[w_1]\leq w_1\leq w_2=\min[w_2].\]
    Thus, we have shown (a)$\Rightarrow$(b). \par
    Assuming (b), choose lifts $(\tilde{x}_1,\tilde{y}_1)$ and $(\tilde{x}_2,\tilde{y}_2)$ so that
    \[w_1=\pos_{\Delta,\Delta}(\tilde{x}_1,\tilde{y}_1)=\min[w_1]\leq\min[w_2]=\pos_{\Delta}(\tilde{x}_2,\tilde{y}_2)=w_2.\]
    Then,
    \begin{align*}
        G\cdot(\tilde{x}_1,\tilde{y}_1)\subseteq \overline{G\cdot(\tilde{x}_2,\tilde{y}_2)}.
    \end{align*}
    Applying $\pi_\theta\times\pi_\eta$ to both sides and using the continuity of $\pi_\theta\times\pi_\eta$, we conclude
    \begin{align*}
        G\cdot(x_1,y_1)\subseteq \overline{G\cdot(x_2,y_2)}.
    \end{align*}
    This completes (b)$\Rightarrow$(a).
\end{proof}
The relation between double cosets $[w_1],[w_2]\in W_{\theta,\eta}$ is considered the Bruhat order on $W_{\theta,\eta}$. Putting these together, we can define a relative position function for partial flags.
\begin{defn}
    The relative position function $\pos_{\theta,\eta}:\mathcal{F}_\theta\times\mathcal{F}_\eta\to W_{\theta,\eta}$ is defined so that, for $x\in\mathcal{F}_\theta,y\in\mathcal{F}_\eta$, $\pos_{\theta,\eta}(x,y)$ is the double coset $[w]\in W_{\theta,\eta}$ where $w=\pos_{\Delta,\Delta}(\tilde{x},\tilde{y})$ for any lifts $\tilde{x},\tilde{y}\in\mathcal{F}_\Delta$.
\end{defn}
\par
We will now consider some specific cases of relative position functions for partial flag varieties.
\begin{xmpl}\label{ProjPositions}
    Let $G=\SL(4,\mathbb{C})$, $\theta=\{1,2,3\}$ and $\eta=\{1\}$ so that
    \begin{align*}
        \mathcal{F}_{\theta}&=\Flag_{full}(\mathbb{C}^4) & \mathcal{F}_{\eta}&=\mathbb{CP}^3 & W_{\theta,\eta}&=S_4/S_3
    \end{align*}
    where $S_3$ is identified with the subgroup of $S_4$ fixing $1$. \par
    This means that $[\sigma]\in S_4/S_3$ is determined by $\sigma(1)$. If $F^\bullet\in\Flag_{full}(\mathbb{C}^4)$ has flag basis $v_1,\hdots, v_4$ and $H^1=\langle v_{\sigma(1)}\rangle$, then
    \[\sigma(1)=\min\{j:H^1\subseteq F^j\}.\]
    
    Thus, the relative position function for $\Flag_{full}(\mathbb{C}^4)$ and $\mathbb{CP}^3$ can be more simply described with values in $\{1,2,3,4\}$, viz.
    \begin{align*}
        \pos(F^\bullet,H^1)&=\min\{j:H^1\subseteq F^j\}.
    \end{align*}
\end{xmpl}
\begin{xmpl}\label{LineLagRelPos}
    Let $G=\Sp(4,\mathbb{C})$, $\theta=\{1\}$ and $\eta=\{2\}$ so that
    \[\mathcal{F}_\theta=\mathbb{CP}^3 \ \ \ \ \mathcal{F}_{\eta}=\Lag(\mathbb{C}^4).\]
    The Weyl group $W$ is the groups of signed permutations of $\{1,2,-2,-1\}$. The left action of $W_{\{1\}}$ on a signed permutation $\sigma^\pm\in W$ flips the sign of $(\sigma^\pm)^{-1}(2)$, while the right action of $W_{\{2\}}$ permutes the outputs while preserving their signs. This means the only remaining information is the sign $\sign((\sigma^\pm)^{-1}(1))$. In other words, if $F^1\in \mathbb{CP}^3$ and $H^2\in\Lag(\mathbb{C}^4)$ are extended to extended full $\omega$--isotropic flags $F^\bullet$, $H^\bullet$ so that $F^\bullet$ has flag basis $v_1,v_2,v_{-2},v_{-1}$ and $H^\bullet$ has flag basis $v_{\sigma^\pm(1)},v_{\sigma^\pm(2)},v_{\sigma^\pm(-2)},v_{\sigma^\pm(-1)}$, then the relative position
    \[\pos_{\{1\},\{2\}}(F^1,H^2)=[\sigma^\pm]\] is determined by $\sign((\sigma^\pm)^{-1}(1))$. \\
\end{xmpl}
\section{Anosov representations and domains of discontinuity}\label{sec:anosov}
In this section, we will define Anosov representations and the construction of their domains of discontinuity in flag varieties. We will then focus on the topology of domains of discontinuity for $\iota$-Fuchsian representations. We finish section by producing a list of cases where our method will apply.
\subsection{Anosov representations}
 Anosov representations were first defined for $G=\SL(n,\mathbb{R})$ and $\Gamma$ a uniform lattice, by Labourie \cite{labourie2005anosovflowssurfacegroups}. Then, Guichard and Wienhard \cite{Guichard_2012} generalized Anosov representations to arbitrary semi-simple Lie groups $G$ and word-hyperbolic groups $\Gamma$. Since then, many equivalent characterizations of Anosov representations have been found. The definition we will use here is among the simplest to state and follows the work of \cite{kapovich2017dynamics} and \cite{Bochi_2019}. \\ 
\indent Let $\Gamma$ be a word-hyperbolic group with word metric $|\cdot|$ and $G$ a non-compact connected Lie group with semi-simple Lie algebra $\mathfrak{g}$ and finite center. For our purposes, $G$ will be a complex Lie group, but the same definition holds for real groups. Let $(K,\mathfrak{a},\Delta)$ be as described in Section \ref{sect:liethry}. Recall that the Cartan projection
\[\mu:G\to\overline{\mathfrak{a}^+}.\]
is a map that generalizes the displacement function of hyperbolic space. 
\begin{defn}\label{Anosov}
    For a type $\theta\subseteq\Delta$, a representation $\rho:\Gamma\to G$ is called $P_\theta$--Anosov (or simply $\theta$--Anosov) if there are constants $C,c>0$ such that 
    \[\alpha(\mu(\rho(\gamma)))\geq C|\gamma|-c\text{ for all $\alpha\in\theta$, $\gamma\in\Gamma$}.\]
\end{defn}
Recall the opposition involution $\nu:\Delta\to\Delta$ given by $\nu(\alpha)=-(w_0\cdot \alpha)$. Note that being $\theta$--Anosov is equivalent to being $(\theta\cup\nu(\theta))$--Anosov by \cite[Lemma 3.18]{Guichard_2012}. Going forward, we will always assume $\theta=\nu(\eta)$, i.e.\ $P_\theta$ is a self-opposite parabolic subgroup. \par
Let $\del\Gamma$ be the Gromov boundary of $\Gamma$.
One of the most important features of $\theta$--Anosov representation is that they admit a special embedding of $\del\Gamma$ into $\mathcal{F}_\theta$.
\begin{prop}[\text{\cite[Theorem 1.3]{Gueritaud_2017}}]
    For a $\theta$--Anosov representation $$\rho:\Gamma\to G,$$ 
    there is a unique continuous map $\xi:\del\Gamma\to\mathcal{F}_\theta$ such that
    \begin{enumerate}[label=\normalfont{(\roman*)}]
        \item The map $\xi$ is $\rho$--equivariant: $\rho(\gamma)\cdot \xi(x)=\xi(\gamma\cdot x)$ for $x\in\del\Gamma$,  $\gamma\in \Gamma$.
        \item The map $\xi$ is transverse: $\pos_{\theta,\theta}(\xi(x),\xi(y))=[w_0]$ for any distinct $x,y\in\del \Gamma$.
        \item The map $\xi$ is dynamics preserving: for any non-torsion element $\gamma$ in $\Gamma$ with attracting fixed point $\gamma^+\in\del\Gamma$, $\xi(\gamma^+)$ is an attracting fixed point for the action of $\rho(\gamma)$ on $\mathcal{F}_\theta$.
    \end{enumerate}
    Such a map is called the limit map of $\rho$.
\end{prop}
\begin{xmpl}
    Let $\Gamma=\pi_1(S_g)$ for some $g\geq 2$ and $G=\PSL(2,\mathbb{C})$. Because $\PSL(2,\mathbb{C})$ has rank $1$, Borel subgroups are the only parabolic subgroups and $\mathbb{CP}^1$ is the only flag variety. Then, $\rho:\Gamma\to\PSL(2,\mathbb{C})$ is Anosov if and only if the orbit map
    \[\tau_{\rho,O}:\Gamma\to\mathbb{H}^3\ \ \ \ \ \ \ \ \gamma\mapsto\rho(\gamma)\cdot O\]
    is a quasi-isometric embedding for some (or any) $O\in\mathbb{H}^3$ (See \cite[Theorem 5.15]{Guichard_2012}). These are known as quasi-Fuchsian representations. In this case, the limit map $\xi$ is exactly the boundary map $\del\tau_{\rho,O}:\del\Gamma\to\del\mathbb{H}^3=\mathbb{CP}^1$ induced by the quasi-isometric embedding $\tau_{\rho,O}$. \par
    The image of $\xi$ is the limit set of $\rho(\Gamma)$. The complement of the limit set
    \[\Omega=\mathbb{CP}^1\setminus \xi(\del\Gamma)\]
    has a properly discontinuous action by $\rho(\Gamma)$. The quotient of $\Omega$ by $\rho(\Gamma)$ is a closed manifold diffeomorphic to $S_g\sqcup \overline{S_g}$.
\end{xmpl}
We will see how to construct domains of discontinuity for Anosov representation into higher rank Lie groups, using a thickened version of $\xi(\del\Gamma)$. Before that, we must discuss the extra piece of data that determined ``how'' to thicken the limit curve.
\subsubsection{Balanced ideals and domains of discontinuity}\label{DomDisc}
Now, we will describe the method of Kapovich-Leeb-Porti to construct domains of discontinuity for Anosov representations. For more discussion of these domains, see \cite{kapovich2017dynamics}, \cite{Stecker2018BalancedIA}. \par
We continue with the assumption that $\theta$ is self-opposite, i.e.\ $\theta=\nu(\theta)$.
Recall from \ref{BruhatRelative} that for types $\theta,\eta$, the double coset space
\[W_{\theta,\eta}=W_\theta\setminus W/W_\eta\]
with the Bruhat order is the set of values for the relative position function $\pos_{\theta,\eta}:\mathcal{F}_\theta\times\mathcal{F}_{\eta}\to W_{\theta,\eta}$. \par
A consequence of the assumption that $\theta=\nu(\theta)$ and $w_0s_{\alpha}w_0=s_{\nu(\alpha)}$ is $w_0$ normalizes $W_\theta$. This means for $[w]\in W_{\theta,\eta}$, $w_0[w]=[w_0w]$ is a well-defined double coset in $W_{\theta,\eta}$. This defines a order reversing action on $W_{\theta,\eta}$ by $w_0$.
\begin{defn}
    An ideal of $W_{\theta,\eta}$ is a subset $I\subseteq W_{\theta,\eta}$ such that
    \[[w_1]\leq[w_2],\ [w_2]\in I\ \Rightarrow [w_1]\in I\]
    There are three properties of an ideal we can consider:
    \begin{itemize}
        \item An ideal $I$ is \textit{fat} if $W_{\theta,\eta}\setminus I\subseteq w_0I$.
        \item An ideal $I$ is \textit{slim} if $w_0I\subseteq W_{\theta,\eta}\setminus I$.
        \item An ideal $I$ is \textit{balanced} if $w_0I=W_{\theta,\eta}\setminus I$.
    \end{itemize}
\end{defn}
Given an ideal $I\subseteq W_{\theta,\eta}$ and a flag $x\in\mathcal{F}_\theta$, the thickening of $x$ according to $I$ is the subset
\[\Phi^I_x=\{y\in\mathcal{F}_\eta:\pos_{\theta,\eta}(x,y)\in I\}\]
which is a closed subset of $\mathcal{F}_\eta$.
Because the relative position function $\pos_{\theta,\eta}$ is $G$--invariant, thickening according to an ideal $I$ satisfies the equivariance relation
\[g\cdot \Phi^I_x=\Phi^I_{g\cdot x}\text{ for }x\in\mathcal{F}_\theta,\  g\in G.\]
\begin{prop}[\text{\cite[Theorem 7.13]{kapovich2017anosovsubgroupsdynamicalgeometric}}]\label{KLPDomainThm}
    Let $\rho:\Gamma\to G$ be a $\theta$--Anosov representation with limit map $\xi:\del\Gamma\to\mathcal{F}_\theta$. If $I\subseteq W_{\theta,\eta}$ is a fat ideal and $K_\rho^I$ is the thickening of $\xi(\del\Gamma)$ according to $I$, i.e.\
    \[K_\rho^I=\bigcup_{t\in\del\Gamma}\Phi^I_{\xi(t)},\]
    then the complement
    \[\Omega_\rho^I=\mathcal{F}_\eta\setminus K_\rho^I\]
    is a domain of discontinuity for $\rho$. \par
    If $I$ is balanced, then the quotient
    \[\mathcal{W}_\rho^I=\rho(\Gamma)\setminus \Omega_\rho^I\]
    is a compact.
\end{prop}
Thus, a $\theta$--Anosov representation admits a cocompact domain of discontinuity in $\mathcal{F}_\eta$ if $W_{\theta,\eta}$ contains a balanced ideal. In fact, Stecker in \cite{Stecker2018BalancedIA} shows that this is essentially the only way to construct domains of discontinuity in a flag variety for $\Delta$--Anosov representations. \par
We note that if $\theta'\subseteq\theta$, then there is the natural projection $\pi_{W}:W_{\theta,\eta}\to W_{\theta',\eta}$ defined by
\[W_\theta wW_{\eta}\mapsto W_{\theta'}wW_{\eta}.\]
We also have the projection $\pi_{\mathcal{F}}:\mathcal{F}_{\theta}\to\mathcal{F}_{\theta'}$ defined by
\[gP_\theta\mapsto g P_{\theta'}.\]
Importantly, these projections are compatible with relative positions in the sense that
\[\pos_{\theta',\eta}(\pi_{\mathcal{F}}(x),y)=\pi_W(\pos_{\theta,\eta}(x,y))\]
for $x\in\mathcal{F}_\theta$ and $y\in\mathcal{F}_\eta$ (which follows from the fact that $W_\theta=(N_K(\mathfrak{a})\cap P_\theta)/Z_K(\mathfrak{a})$). A consequence from this is that if $I\subseteq W_{\theta,\eta}$ is the balanced ideal in Proposition \ref{KLPDomainThm} and there is a balanced ideal $I'\subseteq W_{\theta',\eta}$  such that
\[I=\pi_W^{-1}(I'),\]
then the domains inside $\mathcal{F}_\eta$ defined by these ideals are the same, i.e.\ 
\[\Omega_\rho^I=\Omega_\rho^{I'}.\]
The conclusion of this is that domain $\Omega_\rho^I$ does not actually require $\rho$ to be $\theta$--Anosov, but the weaker condition of $\theta'$--Anosov. \par
\begin{rem}\label{minAnosovDom}
    The observation above gives a procedure for finding the correct minimal Anosov condition for a representation to admit a domain of discontinuity in a flag variety $\mathcal{F}_\eta$. Let $\theta=\Delta$. First, we look for balanced ideals $I$ inside $W_{\Delta,\eta}=W/W_{\eta}$. Then, there exists a balanced ideal $I'\subseteq W_{\theta',\eta}$ such that $I=\pi_W^{-1}(I')$ if and only if $I$ is left $W_{\theta'}$--invariant. Thus, 
    \[\theta'=\{\alpha\in\Delta:s_\alpha\cdot I=I\}\]
    is the minimal $\theta'\subseteq \Delta$ so that $W_{\theta',\eta}$ admits a balanced ideal.
\end{rem}
Many of the domains of discontinuity constructed using the Kapovich-Leeb-Porti method were also described by Guichard \& Wienhard in \cite{Guichard_2012}. We will now describe some noteworthy domains that can be contructed with either method. \par

\begin{xmpl}\label{FlagC3RelPos}
    Let $G=\SL(3,\mathbb{C})$. Consider the full types $\theta=\eta=\Delta=\{1,2\}$. The set of relative positions between flags in $\mathcal{F}_\theta=\Flag(\mathbb{C}^3)$ and $\mathcal{F}_\eta=\Flag(\mathbb{C}^3)$ is then indexed by $W=S_3$. If we denote a permutation $\sigma\in S_3$ by the string $\sigma(1)\sigma(2)\sigma(3)$, then the Hasse diagram for $W$ looks as in Figure \ref{fig:BruhS3} and thus we can determine that the only ideal in $W=S_3$ which is balanced is the subset $I=\{123,213,132\}$. 
    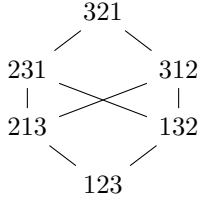
\begin{figure}
        \centering
        \begin{tikzcd}[column sep = 5pt, row sep = 8pt]
	& 321 \\
	231 && 312 \\
	213 && 132 \\
	& 123
	\arrow[no head, from=1-2, to=2-1]
	\arrow[no head, from=1-2, to=2-3]
	\arrow[no head, from=2-1, to=3-1]
	\arrow[no head, from=2-1, to=3-3]
	\arrow[no head, from=2-3, to=3-1]
	\arrow[no head, from=2-3, to=3-3]
	\arrow[no head, from=3-1, to=4-2]
	\arrow[no head, from=3-3, to=4-2]
\end{tikzcd}
        \caption{Hasse diagram of Bruhat order on $W=S_3$}
        \label{fig:BruhS3}
    \end{figure}
    Note that this ideal can be described more formally as 
$$I=\{\sigma\in S_3:\sigma(1)=1\text{ or }\{\sigma(1),\sigma(2)\}=\{1,2\}\}$$. If $F^\bullet,H^\bullet$ are flags in $\Flag(\mathbb{C}^3)$ with $$\pos_{\Delta,\Delta}(F^\bullet,H^\bullet)=\sigma,$$
then there is a flag basis $v_1,v_2,v_3$ for $F^\bullet$ such that $v_{\sigma(1)},v_{\sigma(2)},v_{\sigma(3)}$ is a flag basis for $H^\bullet$. The condition that either $\sigma(1)=1$ or $\{\sigma(1),\sigma(2)\}=\{1,2\}$ means that either $F^1=H^1=\langle v_1\rangle$ or $F^2=H^2=\langle v_1,v_2\rangle$. This gives a description of the balanced ideal without referring to the Weyl group $S_3$, viz.
\[\pos_{\Delta,\Delta}(F^\bullet, H^\bullet)\in I\ \text{ if and only if}\ \text{ either }F^1=H^1\text{ or }F^2=H^2.\]
Then, for any Borel Anosov representation $\rho:\Gamma\to\SL(3,\mathbb{C})$ with limit curve $\xi:\del\Gamma\to\Flag(\mathbb{C}^3)$, then the open subset
\[\Omega_\rho^I=\{H^\bullet\in\Flag(\mathbb{C}^3):H^1\neq \xi^1(t), H^2\neq\xi^2(t)\text{ for all }t\in\del\Gamma\}\]
is a cocompact domain of discontinuity for the action of $\rho(\Gamma)$.
\end{xmpl}

\begin{xmpl}\label{CP3RelPos}
    Now let $G=\SL(4,\mathbb{C})$ and consider the types $\theta=\Delta$, $\eta=\{1\}$ so that $\mathcal{F}_\theta=\Flag(\mathbb{C}^4)$ and $\mathcal{F}_\eta=\mathbb{CP}^3$. Then, as in Example \ref{ProjPositions}, a relative position $[\sigma]\in W_{\theta,\eta}$ is determined by $\sigma(1)$ and we can consider the relative position function
    \[\pos(F^\bullet,H^1)=\min\{j:H^1\subseteq F^j\}.\]
    This takes values in $\{1,2,3,4\}$ with the usual order giving the Bruhat order. Furthermore, the action by $w_0$ on $\{1,2,3,4\}$ is given by the usual reversal of order $j\mapsto 5-j$. \par
    Then, it can be determined that the only balanced ideal in $\{1,2,3,4\}$ is the set $\tilde{I}=\{1,2\}$. Note that this ideal is invariant under the left action by $W_{\{2\}}=\{\sigma\in S_4:\sigma(\{1,2\})=\{1,2\}\}$ and thus only depends on $F^2$ in particular
    \[\pos(F^\bullet,H^1)\in \tilde{{I}}\iff H^1\subseteq F^2.\]
    This means that $\tilde{I}$ descends to a balanced ideal $I$ in $W_{\{2\},\{1\}}$ with
    \[\pos_{\{2\},\{1\}}(F^2,H^1)\in I\iff H^1\subseteq F^2.\]
    This means by Proposition \ref{KLPDomainThm}, for any $\{2\}$--Anosov representation $\rho:\Gamma\to \SL(4,\mathbb{C})$ with limit map $\xi^2:\del\Gamma\to\Gr_2(\mathbb{C}^4)$, then
    \[\Omega_\rho^I=\mathbb{CP}^3\setminus \bigcup_{t\in\del\Gamma}\mathbb{P}(\xi^2(t))\]
    is a cocompact domain of discontinuity.
\end{xmpl}
Finally, we will discuss how to construct a domain of discontinuity inside the Lagrangian Grassmanian for $1$--Anosov symplectic representations. This is one instance of the Guichard-Weinhard method for constructed domains of discontinuity. See Section 7 of \cite{Guichard_2012} for more details. For the sake of uniformity, we will instead describe the balanced ideal corresponding to this domain.
\begin{xmpl}
    Let $G=\Sp(4,\mathbb{C})$ and consider the types $\theta=\Delta$ and $\eta=\{2\}$ so that
    \[\mathcal{F}_\theta=\IsoFlag_{full}(\mathbb{C}^4,\omega), \ \ \ \ \mathcal{F}_\eta=\Lag(\mathbb{C}^4)\]
Recall from the Weyl Group $W$ of $G$ is the group of signed permutations of $\{1,2,-2,-1\}$. The subgroup $W_\eta$ is generated by the transposition $s_1=(2\ 1)$. When $s_1$ acts on the right, it permutes the absolute values of the outputs without changing their sign, e.g.\ 
\[(1\ -2)\cdot s_1=  (2\ -1) \ \ \ (-1\ 2)\cdot s_1=(-2\ 1).\]
This means the set of relative positions $W_{\theta,\eta}=W/W_{\eta}$ can be thought of as the set of pairs of signs
\[W_{\theta,\eta}=\{(+,+),(-,+),(+,-),(-,-)\}\]
so that $(-,+)$ represents the coset $\{(-1\  2),(-2\ 1)\}$. The Bruhat order on $W_{\theta,\eta}$ looks like Figure \ref{fig:BruhSp4} where $w_0=(-2\ -1)$ acts by flipping both signs in each ordered pair. 
\begin{figure}
    \centering
    \begin{tikzcd}[row sep = 10pt]
	{(-, -)} \\
	{(-, +)} \\
	{(+, -)} \\
	{(+, +)}
	\arrow[no head, from=1-1, to=2-1]
	\arrow[no head, from=2-1, to=3-1]
	\arrow[no head, from=3-1, to=4-1]
\end{tikzcd}
    \caption{Bruhat order on $W_{\Delta,\{2\}}$ for $G=\Sp(4,\mathbb{C})$}
    \label{fig:BruhSp4}
\end{figure}
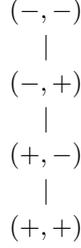
We conclude that $I=\{(+,+),(+,-)\}$ is the only balanced ideal in $W_{\theta,\eta}$. \par
The only simple reflection, with respect to $I$ is left invariant, is $s_{2}=(1\ -2)$. So, by Remark \ref{minAnosovDom}, a representation $\Gamma\to \Sp(4,\mathbb{C})$ merely needs to be $\{1\}$--Anosov in order to admit a domain of discontinuity in $\Lag(\mathbb{C}^4)$. \par

More explicitly, whether an order pair of signs lies in $I$ depends only on the first sign. Then, we can consider $W_{\{1\},\eta}$ as the set of signs $\{+,-\}$ with the projection $W_{\theta,\eta}\to W_{\{1\},\eta}$ represented by projection onto the first factor. Thus, $I$ projects onto the balanced ideal $I'=\{+\}$. Then, for any $\{1\}$--Anosov representation $\Gamma\to\Sp(4,\mathbb{C})$ with limit map $\xi^1:\del \Gamma\to \mathbb{CP}^3$,
\[\Omega_\rho^{I'}=\{L\in\Lag(\mathbb{C}^4):\pos_{\{1\},\{2\}}(\xi^{1}(t),L)\notin I'\text{ for all }t\in\del\Gamma\}\]
is a domain of discontinuity for $\rho$ in $\Lag(\mathbb{C}^4)$.\par
We recall from Example \ref{LineLagRelPos} that an element $[\sigma^\pm]$ of $W_{\{1\},\{2\}}$ is determined by the sign of $(\sigma^{\pm})^{-1}(1)$. \par
Let $F^\bullet,H^\bullet$ be extended $\omega$--isotropic flags. Let $v_1,v_{2},v_{-2},v_{-1}$ be a flag basis for $F^\bullet$ and $\sigma^\pm\in W$ be a signed permutation so that $v_{\sigma^\pm(1)},v_{\sigma^\pm(2)},v_{\sigma^{\pm}(-2)},v_{\sigma^\pm(-1)}$ is a flag basis for $H^\bullet$. Then by Proposition \ref{RelPosSp2n}, $\pos_{\{1\},\{2\}}(F^1,H^2)=[\sigma^\pm]$.
   Note that $(\sigma^\pm)^{-1}(1)$ is positive if and only if $(\sigma^{\pm})^{-1}(1)=1,2$ which happens exactly when 
    \[F^1=\langle v_{1}\rangle\subseteq \langle v_{\sigma^\pm(1)},v_{\sigma^{\pm}(2)}\rangle =H^2. \]
    This means that the domain $\Omega_\rho^{I'}$ can be more simply presented as
    \[\Omega_\rho^{I'}=\{L\in\Lag(\mathbb{C}^4):\xi^1(t)\nsubseteq L\text{ for all }t\in \del\Gamma\}.\]
    It is worth noting that the balanced ideal $I'=\{+\}$ generalizes to any $W_{\{1\},\{n\}}$ for  $G=\Sp(2n,\mathbb{C})$ so that for any $\{1\}$--Anosov representation $\rho:\Gamma\to\Sp(2n,\mathbb{C})$, the subset
    \[\Omega_\rho^{I'}=\{L\in\Lag(\mathbb{C}^{2n}):\xi^1(t)\nsubseteq L \text{ for all }t\in\del\Gamma\}.\]
    is a cocompact domain of discontinuity for $\rho$ in $\Lag(\mathbb{C}^{2n})$. \par
    Furthermore, $I'$ has an inverse ideal $I''$ in $W_{\{n\},\{1\}}$ so that
    \[[w]\in I''\iff [w^{-1}]\in I'\text{ for }w\in W_{\{n\},\{1\}}\]
    This has the effect of swapping the roles of $\mathcal{F}_\theta$ and $\mathcal{F}_\eta$. To be specific, if $\rho:\Gamma\to \Sp(2n,\mathbb{C})$ is an $\{n\}$--Anosov representation, then the subset
    \[\Omega_\rho^{I''}=\{\ell\in\mathbb{CP}^{2n-1}:\ell\nsubseteq \xi^n(t)\text{ for all }t\in\del\Gamma\}.\]
    These two domains (and their orthogonal group counterparts) play a key role in Guichard and Wienhard's construction of domains of discontinuity, see \cite[Sec.\ 2.7]{Guichard_2012}.
\end{xmpl}

\subsection{Fuchsian representations into higher rank Lie groups}
We will now discuss the main subclass of Anosov representations which will be the focus of this paper, viz. $\iota$--Fuchsian representations. Our goal is to understand the topology of the domain of discontinuity described in Section \ref{DomDisc} for these representations.
\subsubsection{$\iota$--Fuchsian representations}
 \par
Recall from Example \ref{CartanHyperbolic}, for $\SL(2,\mathbb{C})$ with the choice of $(K,\mathfrak{a},\Delta)$ as
\[K=\SU(2), \ \ \ \mathfrak{a}=\left\{\begin{bmatrix}
    \lambda & \\ & -\lambda
\end{bmatrix}:\lambda\in \mathbb{R}\right\}, \ \ \ \Delta=\{\alpha_1\}=\{\epsilon_1-\epsilon_2\}\]
then, the Cartan projection composed with the simple root $\alpha_1$ gives the displacement of the origin $O$ in $\mathbb{H}^3$:
\[\alpha_1(\mu(\gamma))=d_{\mathbb{H}^3}(O,\gamma\cdot O)\text{ for }\gamma\in \SL(2,\mathbb{C)}.\]
The same arguments holds for $\SL(2,\mathbb{R})$ with $(K,\mathfrak{a},\Delta)$ as
\[K=\SOrth(2), \ \ \ \mathfrak{a}=\left\{\begin{bmatrix}
    \lambda & \\ & -\lambda
\end{bmatrix}:\lambda\in \mathbb{R}\right\}, \ \ \ \Delta=\{\alpha_1\}=\{\epsilon_1-\epsilon_2\}\]
to get
\[\alpha_1(\mu(\gamma))=d_{\mathbb{H}^2}(O,\gamma\cdot O)\text{ for all }\gamma\in\SL(2,\mathbb{R)}.\]
This means that a representation $\rho:\Gamma\to \SL(2,\mathbb{R})$ is $\alpha_1$--Anosov if and only if there are constants $C,c>0$ such that
\[d_{\mathbb{H}^2}(O,\rho(\gamma)\cdot O)\geq C|\gamma|-c\]
which is equivalent to saying that the map $\tau_{\rho,O}:\Gamma\to \mathbb{H}^2$ defined by $\gamma\mapsto\rho(\gamma)\cdot O$ is a quasi-isometric embedding. This is the definition for a convex cocompact representation in $\SL(2,\mathbb{R})$. The behavior of convex cocompact representation in rank one is one of the main motivations for defining Anosov representations as a higher rank generalization. \par
Fix $g\geq 2$ and let $S$ be a closed orientable surface of genus $g$. Let $\Gamma=\pi_1(S)$. If $S$ is equipped with a Riemannian metric with constant sectional curvature $-1$, then the universal cover $\tilde{S}$ can be identified with $\mathbb{H}^2$ so that the deck transformations of $\tilde{S}$ correspond to isometries of $\mathbb{H}^2$. This gives a representation
\[\bar{\phi}:\Gamma\to\PSL(2,\mathbb{R})=\operatorname{Isom}^+(\mathbb{H}^2).\]
From \cite{CULLER198664}, we know there is a lift $\phi:\Gamma\to \SL(2,\mathbb{R})$ so that the diagram in Figure \ref{fig:fuchscommdiag} commutes.
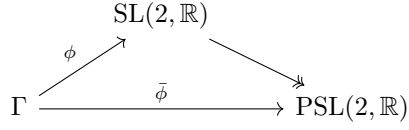
\begin{figure}
    \centering
    \begin{tikzcd}
	& {\SL(2,\mathbb{R})} \\
	\Gamma && {\PSL(2,\mathbb{R})}
	\arrow[two heads, from=1-2, to=2-3]
	\arrow["\phi", from=2-1, to=1-2]
	\arrow["\bar{\phi}", from=2-1, to=2-3]
\end{tikzcd}
    \caption{Lift of Fuchsian representation to $\SL(2,\mathbb{R})$}
    \label{fig:fuchscommdiag}
\end{figure}
Such a representation $\phi:\Gamma\to \SL(2,\mathbb{R})$ is called a Fuchsian representation.
Then, because $\Gamma$ acts on $\mathbb{H}^2$ through $\phi$ as the deck transformations of $\tilde{S}$, we conclude by the \v Svarc-Milnor Lemma that the orbit map $\tau_{\eta,O}:\Gamma\to \mathbb{H}^2$ is a quasi-isometric embedding. Thus, $\theta:\Gamma\to\SL(2,\mathbb{R})$ is $\alpha_1$--Anosov. This argument works to show that the inclusion map of any uniform lattice in a rank one Lie group is Anosov. \par
Continuing the notation $\Gamma=\pi_1(S)$ with $\theta:\Gamma\to\SL(2,\mathbb{R})$ being a Fuchsian representation. Let $G$ be a (possibly higher rank) Lie group with a choice of $(K,\mathfrak{a},\Delta)$. To produce an Anosov representation $\rho:\Gamma\to G$, we need a representation $\iota:\SL(2,\mathbb{R})\to G$. 
\begin{defn}
    For a Lie group $G$ and a non-trivial homomorphism $\iota:\SL(2,\mathbb{R})\to G$, a representation $\rho:\Gamma\to G$ is called \textit{$\iota$--Fuchsian} if there is some Fuchsian representation $\phi:\Gamma\to\SL(2,\mathbb{R})$ such that $\rho=\iota\circ\phi$.
\end{defn}
Then, the next proposition will tell us for what type $\theta\subseteq\Delta$ is $\rho=\iota\circ \phi$ a $\theta$--Anosov representation. 
\begin{prop}[\text{\cite[Prop.\ 4.7]{Guichard_2012}}]\label{Anosovness4Fuchs}
    Let $\iota:\SL(2,\mathbb{R})\to G$ be an representation with induced Lie algebra homomorphism $d\iota:\mathfrak{sl}_2(\mathbb{R})\to \mathfrak{g}$. Assume that $(K,\mathfrak{a},\Delta)$ are chosen to that
    \[a_{\iota}=d\iota\left(\begin{bmatrix}
        1 & \\ & -1
    \end{bmatrix}\right)\in\overline{\mathfrak{a}^+}.\]
Then, $\rho=\iota\circ \phi$ is $\theta$--Anosov where
\[\theta=\{\alpha\in \Delta:\alpha(a_\iota)\neq 0\}.\]
\end{prop}
Proposition \ref{Anosovness4Fuchs} is best understood with the interpretation of Anosov representations as representations with growing singular value gaps as will be seen in the next example.
\begin{xmpl}
    Let $G=\SL(n,\mathbb{C})$ with the standard choices for $(K,\mathfrak{a},\Delta)$ as
    \begin{align*}
        K&=\SU(n) \\ 
        \mathfrak{a}&=\{\diag(\lambda_1,\hdots,\lambda_n)\in\mathfrak{sl}_n:\lambda_1,\hdots,\lambda_n\in\mathbb{R}\} \\ 
        \Delta&=\{\epsilon_{j}-\epsilon_{j+1}:1\leq j\leq n-1\}.
    \end{align*} \par
    For $d\geq 1$, let $\rho_d:\SL(2,\mathbb{C})\to\SL(d,\mathbb{C})$ be the unique irreducible $\SL(2,\mathbb{C})$--representation of dimension $d$. We can identify $\mathbb{C}^d$ with the space $\mathbb{C}^{(d-1)}[X,Y]$ of degree $d-1$ homogeneous polynomials in $X,Y$ using the basis $X^{d-1},X^{d-2}Y,\hdots, Y^{d-1}$. Under this identification, we can describe $\rho_d$ by its linear action on $\mathbb{C}^{(d-1)}[X,Y]$ as follows
    \[\rho_d\left(\begin{bmatrix}
        A & B \\ C & D
    \end{bmatrix}\right)\cdot f(X,Y)=f(DX-BY,AY-CX)\]
    for $f(X,Y)\in\mathbb{C}^{(d-1)}$. \par
    
    A \textit{partition} of $n$ is a tuple $(d_1,\hdots,d_\ell)$ so that
    \begin{enumerate}
        \item The entries $d_1,\hdots,d_\ell$ are positive integers.
        \item The entries are in descending order, i.e.\ $d_1\geq d_2\geq\cdot \geq d_\ell$.
        \item The entries add up to $n$, i.e.\ $d_1+\hdots+d_\ell=n$.
    \end{enumerate}
    The entries $d_1,\hdots,d_\ell$ are called the parts of the partition and the number of times a given part appears in the partition is called the multiplicity of the part. \par
    For any partition $(d_1,\hdots,d_\ell)$, we have the representation
    \[\rho_{d_1}\oplus\cdots\oplus \rho_{d_\ell}:\SL(2,\mathbb{C})\to\SL(n,\mathbb{C}).\]
     A consequence of the Jacobson-Morozov Theorem is that any representation $\iota:\SL(2,\mathbb{C})\to\SL(n,\mathbb{C})$ is  conjugate to $\rho_{d_1}\oplus\cdots\oplus\rho_{d_\ell}$ for a unique partition $(d_1,\hdots,d_\ell)$ of $n$. For more discussion of this classification, see \cite[Sect. 3.2]{CollingwoodMcGovern}. \par
    So let us continue with the assumption that $\iota=\rho_{d_1}\oplus\cdots \oplus\rho_{d_\ell}$ for some partition $(d_1,\hdots,d_\ell)$ of $n$. Then, because 
    \[\rho_d\left(\begin{bmatrix}
        \lambda & \\ & \lambda^{-1}
    \end{bmatrix}\right)=\begin{bmatrix}
        \lambda^{d-1} & & & & \\ & \lambda^{d-3} & & & \\ & & \ddots & & \\ & & & \lambda^{3-d} & \\ & & & & \lambda^{1-d}
    \end{bmatrix}\]
    we see that 
    \[a_{\rho_d}=d\rho_d\left(\begin{bmatrix}
        1 & \\  & -1
    \end{bmatrix}\right)=\begin{bmatrix}
        {d-1} & & & & \\ & {d-3} & & & \\ & & \ddots & & \\ & & & {3-d} & \\ & & & & {1-d}
    \end{bmatrix}.\]
    In particular, $a_{\rho_d}$ is diagonal with real entries. This means that 
    \[a_{\iota}=a_{\rho_{d_1}}\oplus\cdot\oplus a_{\rho_{d_\ell}}\]
    is an element of $\mathfrak{a}$. Thus, for a suitable element $k$ of $N_K(\mathfrak{a})$, $\Ad_k(a_{\iota})$ is in the closed Weyl chamber $\overline{\mathfrak{a}^+}$. More explicitly $k$ is a signed permutation matrix that permutes the diagonal entries of $a_{\iota}$ so that $\Ad_k(a_\iota)$ has diagonal entries in descending order. Then, the representation $\iota':\SL(2,\mathbb{C})\to\SL(n,\mathbb{C})$ given by $\iota'(g)=k\iota(g)k^{-1}$ satisfies the assumptions of Proposition \ref{Anosovness4Fuchs} and thus $\iota'$ (and therefore also $\iota$) is $\theta$--Anosov for
    $\theta=\{\alpha\in\Delta:\alpha(a_{\iota'})\neq0\}$. \par
    To give an explicit example, let $n=6$ and
    \[\iota=\rho_3\oplus\rho_2\oplus\rho_1.\]
    Then,
    \[a_{\iota}=\begin{bmatrix}
        2 & & & & & \\
        & 0 & & & & \\
        & & -2  & & & \\
         & & &1 & &\\ & & & & -1 &\\ & & & & & 0
    \end{bmatrix}.\]
    Then, conjugating by
    \[k=\begin{bmatrix}
        1 & 0 & 0 & 0 & 0 & 0 \\ 0 & 0 & 0& 1 &0 &0 \\ 0 &0 &0 &0 &0 & 1\\0 & 1 &0 &0 &0 &0 \\ 0& 0& 0& 0& 1 &0 \\  0& 0& 1 & 0 & 0 &0
    \end{bmatrix}\]
    will produce the representation $\iota'$ so that
    \[a_{\iota'}=\begin{bmatrix}
        2 & & & & & \\
        & 1 & & & & \\
        & & 0  & & & \\
         & & &0 & &\\ & & & & -1 &\\ & & & & & 2
    \end{bmatrix}.\]
    Then, $\theta=\{\alpha_1,\alpha_2,\alpha_4,\alpha_5\}$ where $\alpha_j=\epsilon_j-\epsilon_{j+1}$. \par
\end{xmpl}
This example allows us to refer to the broad class of $\iota$--Fuchsian representations into $\SL(n,\mathbb{C})$ by partitions of $n$. Now, we will see how this works for $G=\Sp(2n,\mathbb{C})$.
\begin{prop}[\text{\cite[Thm.\ 5.1.3]{CollingwoodMcGovern}}]\label{invariantsymp}
    For a partition $(d_1,\hdots,d_\ell)$ of $2n$, consider the representation $$\iota=\rho_{d_1}\oplus\cdots\oplus\rho_{d_\ell}:\SL(2,\mathbb{C})\to \SL(2n,\mathbb{C}).$$ The following are equivalent:
    \begin{enumerate}
        \item The vector space $\mathbb{C}^{2n}$ admits a $\iota$--invariant symplectic form $\omega$ so that we can consider $\iota$ as a representation
        \[\iota:\SL(2,\mathbb{C})\to\Sp(\mathbb{C}^{2n},\omega).\]
        \item Each odd part of $(d_1,\hdots,d_\ell)$ has even multiplicity.
    \end{enumerate}
\end{prop}
\subsubsection{Smooth fibration of domains of discontinuity}
Now, we are ready to discuss the domains of discontinuity for $\iota$--Fuchsian representations. \par
First, note that for $\Gamma=\pi_1(S)$, the cocompact action of $\Gamma$ on $\mathbb{H}^2$ through $\theta$ describes a $\theta$--equivariant homeomorphism
\[\psi:\mathbb{RP}^1\to \del\Gamma\] 
where $\mathbb{RP}^1$ is identified with the boundary of $\mathbb{H}^2$ in the upper half-plane model. Let $\rho$ be an $\iota$--Fuchsian representation which is $\theta$--Anosov with limit map $\xi:\del\Gamma\to\mathcal{F}_\theta$. Then, the composition $\xi\circ \psi:\mathbb{RP}^1\to \mathcal{F}_\theta$ is a $\rho$--equivariant map which we will refer to as the \textit{limit curve} of $\rho$. In fact from now on, we will suppress $\psi$ and write the limit curve simply as $\xi:\mathbb{RP}^1\to\mathcal{F}_\theta$. The important quality about $\mathbb{RP}^1$ over $\del\Gamma$ is that it admits a natural action by $\SL(2,\mathbb{R})$. In particular, one can show that $\xi$ is $\iota$--equivariant with respect to this action, i.e.\
\[\xi(g\cdot t)=\iota(g)\cdot\xi(t)\ \text{ for all $g\in\SL(2,\mathbb{R})$ and $t\in\mathbb{RP}^1$.}\] 
This means $\xi(\mathbb{RP}^1)$ is invariant under the action of $\iota(\SL(2,\mathbb{R}))$. Thus, by the equivariance of thickening, for any balanced ideal $I\subseteq W_{\theta,\eta}$, the thickening of the limit curve
\[K_\rho^I=\bigcup_{t\in\mathbb{RP}^1}\Phi_{\xi(t)}^I\]
and the associated domain of discontinuity
\[\Omega_\rho^I=\mathcal{F}_\eta\setminus K_\rho^I\]
are $\iota(\SL(2,\mathbb{R}))$--invariant subsets of the flag variety $\mathcal{F}_{\eta}$. \par
In this case and in greater generality, Alessandrini, Maloni, Tholozan and Wienhard use this $\SL(2,\mathbb{R})$--action \cite{alessandrini2023fiberbundlesassociatedanosov} realize the domain of discontinuity $\Omega_\rho^I$ as a fiber bundle.
\begin{prop}[\text{\cite[Thm.\ 3.1]{alessandrini2023fiberbundlesassociatedanosov}}]\label{AMTWEquiv}
    For an $\iota$--Fuchsian representation $\rho:\Gamma\to G$ and a domain of discontinuity $\Omega_\rho^I$ as above, there is a smooth $\SL(2,\mathbb{R})$--equivariant fiber bundle projection
    \[p:\Omega_\rho^I\to \mathbb{H}^2\]
    where $\SL(2,\mathbb{R})$ acts on $\Omega_\rho^I$ through $\iota$. Furthermore, the fiber of $p$ is a closed manifold.
\end{prop}
The result gives us significant topological information about the domains of discontinuity for $\iota$--Fuchsian representations. In particular, $\Omega_\rho^I$ is diffeomorphic to $\mathbb{H}^2\times M_\rho^I$ where $M_\rho^I=p^{-1}(O)$ is a closed manifold of dimension $\dim\mathcal{F}_\theta-2$. 
\begin{rem}\label{compactproject}
    Note that we have an $\SL(2,\mathbb{R})$--equivariant projection $p:\Omega_\rho^I\to \mathbb{H}^2$ and an $\SL(2,\mathbb{R})$--equivariant function $K_\rho^I\to \mathbb{RP}^1$ that sends the thickening $\Phi_{\xi(t)}^I$ to $t$. Putting these together, we get an $\SL(2,\mathbb{R})$--equivariant function $\bar{p}:\mathcal{F}_\eta\to\overline{\mathbb{H}^2}$
    \[\bar{p}(x)=\begin{cases}
        p(x) & x\in\Omega_\rho^I \\
        t & x\in\Phi_{\xi(t)}^I
    \end{cases}\]
    for $x\in\mathcal{F}_\eta$.
\end{rem}
\par
In order to describe further implications of Proposition \ref{AMTWEquiv}, let us briefly recall the construction of the Euler class of a principal $S^1$--bundle as an obstruction class. \par
Consider the unit circle $S^1$ as a closed subgroup of $\mathbb{C}^\times$. Let $H$ be either $S^1$ or $\mathbb{C}^\times$.
Let $B$ be a smooth manifold. Recall that a smooth principal $H$--bundle over $B$ is a smooth fiber bundle $\pi:P\to B$ with a smooth (right) $H$--action on $P$ so that $H$ acts freely and transitively on each fiber of $P$. \par 

Consider a smooth principal bundle $\pi:P\to B$. Giving $B$ a smooth triangulation, let $B^{(n)}$ denote the $n$--skeleton of this triangulation. The $0$--skeleton $B^{(0)}$ is  a discrete collection of points and thus we can choose a section $$s^{(0)}:B^{(0)}\to \pi^{-1}(B^{(0)}).$$ 
Any section over endpoints of an edge extend continuous to the edge, we can extend $s^{(0)}$ to a continuous section $$s^{(1)}:B^{(1)}\to \pi^{-1}(B^{(1)}).$$ For any oriented $2$--simplex $\sigma$ of $B$, the bundle $\pi|_{\sigma}:\pi^{-1}(\sigma)\to\sigma$ is trivial and thus, there is a continuous section $$t:\sigma\to\pi^{-1}(\sigma).$$ Thus, $s|_{\del\sigma}$ and $t|_{\del \sigma}$ are both continuous sections over $\del\sigma$. This means there is a continuous function $$\phi:\del\sigma\to H$$ such that
\[\phi(x)\cdot t(x)=s(x)\text{ for }x\in\del\sigma.\]
Then, because $\sigma$ is oriented, there is an induced orientation on $\del\sigma\cong S^1$, thus the degree of $\phi:\del \sigma\to H$ is well-defined. Then, the $2$--cochain $c$ defined by the association
\[\sigma\mapsto \deg(\phi)\]
is a cocycle and cohomology class $[c]$ is independent of the choice of sections. This cohomology class $[c]\in H^2(B;\mathbb{Z})$ is called the Euler class of $P$. The Euler class is the only obstruction to $P$ admitting a section and it is a classic result that this association
\begin{align*}
    P&\mapsto [c]
\intertext{defines a correspondence}
\{\text{smooth principal $H$--bundles}\}
&\leftrightarrow H^2(B;\mathbb{Z}).
\end{align*}
See \cite[Sec.\ 6.2]{moritadiffforms}. \par
Let $\pi:P\to B$ be a smooth principal $S^1$--bundle over a smooth manifold $B$ and let $F$ be a smooth manifold with a smooth (left) $S^1$--action. Then, the right $S^1$--action on $P\times F$ defined by
\[\tilde{E}=(p,f)\cdot g=(p\cdot g,g^{-1}\cdot f)\]
is smooth, free, and proper. Thus, the quotient $P\underset{S^1}{\times}F=\tilde{E}/S^1$ is a smooth manifold. One can show that the map $\bar{\pi}:E\to B$ given by
\[[p,f]\mapsto\pi(p)\]
defines a smooth fiber bundle projection. The fiber bundle $\bar{\pi}:E\to B$ is called the associated bundle for $P$ and $F$. \par
We note that if $P$ is a principal $S^1$--bundle and we consider $\mathbb{C}^\times$ with the subgroup action of $S^1$, then the associated bundle $P\underset{S^1}{\times}\mathbb{C}^\times$ is naturally a principal $\mathbb{C}^\times$ bundle via the action
\[[[p,f]\cdot h=[p,fh]\text{ for }h\in \mathbb{C}^\times.\]
In this case, the Euler class of $P$ and the class of $P\underset{S^1}{\times}\mathbb{C}^\times$ agree.
Note that, a priori, the topology of the bundle depends on both the Euler class of $P$ and the smooth $S^1$--action on $F$. If $P$ has Euler class $e\in H^2(B;\mathbb{Z})$, $F=S^1$ with the action $$g\cdot f=g^kf \ \ \  \ \ g\in S^1, f\in S^1$$ for some $k\in\mathbb{Z}$. Then, the associated bundle $P\underset{S^1}{\times} F$ is isomorphic to a principal $S^1$--bundle with Euler class $ke$. \par
Finally, we note that if $B$ is a closed orientable surface, then $H^2(B;\mathbb{Z})$ is a free $\mathbb{Z}$--module with basis given by the Poincar\'e dual of a point. Thus, principal $S^1$--bundles over a closed orientable surface can be specified by an integer although the sign of this integer depends on the choice of orientation on the surface.
\par
Let us now state the second part of the result of \cite{alessandrini2023fiberbundlesassociatedanosov}. Let $p:\Omega_\rho^I\to\mathbb{H}^2$ be the $\SL(2,\mathbb{R})$--equivariant fiber bundle from Proposition \ref{AMTWEquiv}. Quotienting by the action of $\Gamma$, we get a map from $\mathcal{W}_\rho^I=\Gamma\setminus \Omega_\rho^I$ to $S\cong \Gamma\setminus\mathbb{H}^2$, which we denote by
\[\bar{p}:\mathcal{W}^I_\rho\to S.\]
Note that the equivariance of $p$ means that $M^I_\rho=p^{-1}(O)$ is invariant under the action of $\iota(\SOrth(2))$ (because the origin $O\in\mathbb{H}^2$ is invariant under the action of $\SOrth(2)$).
\begin{prop}[\cite{alessandrini2023fiberbundlesassociatedanosov}]
    Let $g$ be the genus the surface $S$. \par
    For an $\iota$--Fuchsian representation $\rho:\Gamma\to G$ with quotient manifold $\mathcal{W}_\rho^I$, the map
    $$\bar{p}:\mathcal{W}_\rho^I\to S$$
    is a smooth fiber bundle isomorphic to the associated bundle
    \[P\underset{S^1}{\times}M_\rho^I\to S\]
    where $P$ is a principal $S^1$--bundle with Euler class $g-1$ over $S$ and $S^1$ acts on $M_\rho^I$ via the homomorphism
    \[S^1\cong \SOrth(2)\xrightarrow[]{\iota}\iota(\SOrth(2)).\]
\end{prop}
A consequence of this theorem is that if you want to determine the topology of the quotient manifold $\mathcal{W}_\rho^I$, one can do so by determining the fiber $M_\rho^I$ up to $S^1$--equivariant diffeomorphism. Naturally, one would start by restricting to fibers in a class of manifolds where the topology of $S^1$--action is well-understood. 
\subsection{Three-dimensional complex flag varieties}\label{handledomains}
The principal aim of this paper is to restrict to the cases where the fiber $M_\rho^I$ is a simply-connected $4$--manifold and apply Fintushel's classification of $S^1$--actions on such manifolds to determine both the topology of $M_\rho^I$ and its $S^1$--action. To this end, we will end this section by determining for which choices of the group $G$ and the flag variety $\mathcal{F}_\eta$ would $M_\rho^I$ be a simply-connected $4$--manifold. For each case, we will determine the choices for the representation $\iota$ so that $\rho=\iota\circ\phi$ admits a domain of discontinuity in $\mathcal{F}_\eta$. \par
First, the Schubert decomposition of a flag variety means that if $G$ is a complex Lie group, any flag variety $\mathcal{F}_\eta$ of $G$ is simply-connected. Then, if the real codimension of $K_\rho^I$ is at least $3$, then $\Omega_\rho^I$ (and thus $M_\rho^I$) must be simply-connected. Thus, in order to find fibers $M_\rho^I$ that are simply-connected $4$--manifolds, we need to find domains of discontinuity in complex flag varieties of complex dimension $3$. \par
For each $G=\SL(n,\mathbb{C}),\SOrth(n,\mathbb{C})$, $\Sp(2n,\mathbb{C})$, we will determine for which value of $n$ are the flag varieties of $G$ too large, i.e.
\[\dim_{\mathbb{C}}\mathcal{F}_\eta>3\text{ for all non-empty }\eta\subseteq \Delta.\]
It suffices to show this inequality for minimal flag varieties $\mathcal{F}_{\eta}$ where $\eta$ is a singleton $\{\alpha\}$ for some simple root $\alpha\in\Delta$.\par
For $G=\SL(n,\mathbb{C})$ with $n\geq 3$, the minimal flag varieties are the Grassmannians
\[\Gr_{k}(\mathbb{C}^n)=\mathcal{F}_{\{k\}}\]
for $k\in\Delta=\{1,\hdots,n-1\}$. The block form of parabolic subgroups $P_k$ means it has complex dimension $n^2-1-k(n-k)$ which means
\[\dim_{\mathbb{C}}\Gr_k(\mathbb{C}^n)=k(n-k)\geq n-1.\]
Thus, for $n\geq 5$, $\SL(n,\mathbb{C})$ does not admit any flag varieties of complex dimension $3$. \par
For $G=\SOrth(n,\mathbb{C})$ with $n\geq 4$, the minimal flag varieties are the isotropic Grassmannians
\[\IsoFlag_{k}(\mathbb{C}^n)=\{F^k\in\Gr_k(\mathbb{C}^n):F^k\subseteq (F^{k})^{\perp_\omega}\}\]
for $k\in\{1,\hdots,p\}$, where $p=\left\lfloor\frac{n}{2}\right\rfloor$ (with the caveat that for $n=2p$, $\IsoFlag_p(\mathbb{C}^n)$ splits into two minimal flag varieties and $\IsoFlag_{p-1}(\mathbb{C}^{n})$ is not a minimal flag variety).
By \cite[Sec.\ 4.6]{Eisenbud_Harris_2016} has dimension
\[\dim_{\mathbb{C}}\IsoFlag_k(\mathbb{C}^n)=k(n-k)-\binom{k+1}{2}.\]
Fixing $n$, this dimension is a quadratic function in $k$ with maximum value at $k=\frac{1}{3}n-\frac{1}{6}$, thus the minimum dimension for $k\in\{1,\hdots,p\}$ is the one furthest from this maximizing value. For $n=4,6$, this minimum dimension occurs at $k=p$ with
\begin{align*}
    \dim_{\mathbb{C}}\IsoFlag_2(\mathbb{C}^4)&=1\\
    \dim_{\mathbb{C}}\IsoFlag_3(\mathbb{C}^6)&=3.
\end{align*}
For $n=5,8$, the minimum dimension occurs at both $k=1$ and $k=p$ with
\begin{align*}
    \dim_{\mathbb{C}}\IsoFlag_1(\mathbb{C}^5)&=\dim_{\mathbb{C}}\IsoFlag_2(\mathbb{C}^5)=3\\
    \dim_{\mathbb{C}}\IsoFlag_1(\mathbb{C}^8)&=\dim_{\mathbb{C}}\IsoFlag_4(\mathbb{C}^8)=6.
\end{align*}
For $n=7$, $n\geq 9$, the minimum dimension occurs at $k=1$ with
\[\dim_{\mathbb{C}}\IsoFlag_1(\mathbb{C}^n)=n-2.\]
Thus, for $n\geq 7$, $\SOrth(n,\mathbb{C})$ does not admit any flag varieties of complex dimension $3$.\par
For $G=\Sp(2n,\mathbb{C})$ with $n\geq 2$, the minimal flag varieties are the isotropic Grassmannians
\[\IsoFlag_k(\mathbb{C}^{2n},\omega)=\{F^k\in\Gr_k(\mathbb{C}^n):F^k\subseteq (F^{k})^{\perp_\omega}\}\]
where $\omega$ is a non-degenerate symplectic form on $\mathbb{C}^{2n}$ and $k\in\{1,\hdots,n\}$. Recall from Section \ref{classgrpflg} that the stabilizer $Q$ of the standard flag in $\IsoFlag_k(\mathbb{C}^{2n},\omega)$ is the intersection of $\Sp(2n,\mathbb{C})$ with the parabolic subgroup $P_{k,2n-k}$ of $\SL(2n,\mathbb{C})$. This means if $\mathfrak{q}$ is the Lie algebra of $Q$, a matrix
\[A=\begin{bmatrix}
    a_{1,1} & \cdots & a_{1,n} & a_{1,-n} & \cdots & a_{1,-1}\\ \vdots & \ddots & \vdots & \vdots & \iddots & \vdots \\
    a_{n,1} & \cdots  & a_{n,n} & a_{n,-n} & \cdots  & a_{1,-n} \\ a_{-n,1} & \cdots & a_{-n,n}  & -a_{n,n} & \cdots & -a_{1,n} \\ \vdots & \iddots & \vdots & \vdots & \ddots & \vdots \\ a_{-1,1} & \cdots & a_{-n,1} & -a_{n,1} & \cdots & -a_{1,1}
\end{bmatrix}\in \mathfrak{sp}_{2n}\]
is in $\mathfrak{q}$ if and only if 
\begin{align*}
    a_{i,j}&=0\text{ for } i\in [k,n],j\in[1,k]\\
    a_{-i,j}&=0\text{ for }i\in[1,n], j\in[1,\min\{i,k\}].
\end{align*}
These linear relations mean $\mathfrak{q}$ has codimension $k(2n-k)-\binom{k}{2}$ in $\mathfrak{sp}_{2n}$. Therefore,
\[\dim_{\mathbb{C}}\IsoFlag_k(\mathbb{C}^{2n},\omega)=k(2n-k)-\binom{k}{2}.\]
Fixing $n$, this dimension is once again a quadratic function in $k$. Applying a similar analysis to before, we conclude the following. For $n=2$, the minimal dimension for a flag variety is
\[\dim_{\mathbb{C}}\IsoFlag_{1}(\mathbb{C}^{4},\omega)=\dim_{\mathbb{C}}\IsoFlag_2(\mathbb{C}^4,\omega)=3.\]
For $n\geq 3$, the flag variety of minimal dimension is always $\mathbb{CP}^{2n-1}=\IsoFlag_{1}(\mathbb{C}^{2n},\omega)$ with
\[\dim_{\mathbb{C}}\mathbb{CP}^{2n-1}=2n-1\geq 2(3)-1=5.\]
Thus, $\Sp(4,\mathbb{C})$ is the only symplectic group which admits a flag variety of complex dimension $3$.\par
By inspecting these dimension formulas, we arrive at the following proposition.
\begin{prop}\label{3dflagcases}
    Let $G$ be a complex classical group of rank at least $2$. Assume that $G$ admits a flag variety $\mathcal{F}$ with $\dim_{\mathbb{C}}\mathcal{F}=3$. Then, $(G,\mathcal{F})$ is one of the following options:
    \begin{enumerate}[label=\normalfont{(\roman*)}]
        \item The group $G$ is $\SL(3,\mathbb{C})$ and the flag variety $\mathcal{F}$ is $\Flag(\mathbb{C}^3)$.
        \item The group $G$ is $\SL(4,\mathbb{C})$ and the flag variety $\mathcal{F}$ is either $\mathbb{CP}^{3}$ or $\Gr_{3}(\mathbb{C}^4)$.
        \item The group $G$ is $\Sp(4,\mathbb{C})$ and the flag variety $\mathcal{F}$ is either $\mathbb{CP}^3$ or $\Lag(\mathbb{C}^4)$.
        \item The group $G$ is $\SOrth(5,\mathbb{C})$ and the flag variety $\mathcal{F}$ is either $\operatorname{Quad}_{3}$ or $\IsoFlag_{2}(\mathbb{C}^5,\omega)$.
        \item The group $G$ is $\SOrth(6,\mathbb{C})$ and the flag variety $\mathcal{F}$ is either $\IsoFlag_{3^+}(\mathbb{C}^{6},\omega')$ or $\IsoFlag_{3^-}(\mathbb{C}^6,\omega')$.
    \end{enumerate}
    where $\omega$ (resp. $\omega'$) is a non-degenerate symmetric bilinear form on $\mathbb{C}^5$ (resp. $\mathbb{C}^6$).
\end{prop}
\begin{rem}
    Note that cases (iv) and (v) in Proposition \ref{3dflagcases} are redundant in the following way: the exceptional isomorphisms $\mathfrak{so}_5\cong\mathfrak{sp}_4$ and $\mathfrak{so}_6\cong \mathfrak{sl}_4$ means that there are covering homomorphisms
    \[\Sp(4,\mathbb{C})\to\SOrth(5,\mathbb{C}) \ \ \ \ \SL(4,\mathbb{C})\to\SOrth(6,\mathbb{C}).\]
    These covering homomorphisms induce a bijection between parabolic subgroups and thus flag varieties. The end result of this is that it suffices to just consider $\Sp(4,\mathbb{C)}$ and $\SL(4,\mathbb{C})$ and not $\SOrth(5,\mathbb{C})$ or $\SOrth(6,\mathbb{C})$. \par
    Furthermore, the duality between $\mathbb{CP}^1$ and $\Gr_3(\mathbb{C}^4)$ means that any representation which admits a domain of discontinuity in $\mathbb{CP}^3$ admits an equivalent domain of discontinuity in $\Gr_3(\mathbb{C}^4)$. We will see the final redundancy by taking a closer look at the balanced ideals for each case.
\end{rem}
Fix an Fuchsian representation $\phi:\Gamma\to\SL(2,\mathbb{R})$. We will end this section by determining for which choices of $\iota:\SL(2,\mathbb{R})\to G$ does $\rho=\iota\circ \phi$ produce an $\iota$--Fuchsian representation which admits a domain of discontinuity in one of the flag varieties listed in cases (i-iv) of Proposition \ref{3dflagcases}. 

\vspace{.2in}
\noindent\textbf{SL(3,$\,  \mathbb{C}$):}
Let $G=\SL(3,\mathbb{C})$ and $\mathcal{F}_{\eta}=\Flag(\mathbb{C}^3)$. Domains of discontinuity in $\Flag(\mathbb{C}^3)$ are constructed from balanced ideals in $W_{\Delta,\Delta}=S_3$. As we saw in Example \ref{FlagC3RelPos}, there is only one balanced ideal, and it describes a domain of discontinuity for Borel Anosov representations. By \ref{Anosovness4Fuchs}, letting $\iota$ be either the irreducible representation $$\rho_3:\SL(2,\mathbb{R})\to\SL(3,\mathbb{C})$$ or the reducible representation $$\rho_2\oplus \rho_1:\SL(2,\mathbb{R})\to\SL(3,\mathbb{C})$$ will produce a Borel Anosov representation $\rho=\iota\circ\phi$. \par
For $\iota_{\bldmth{f},\bldmth{i}}=\rho_3$, let $$\xi_{\bldmth{f},\bldmth{i}}:\mathbb{RP}^1\to \Flag(\mathbb{C}^3)$$ be the limit curve for the $\iota_{\bldmth{f},\bldmth{i}}$--Fuchsian representation $\rho_{\bldmth{f,i}}=\iota_{\bldmth{f},\bldmth{i}}\circ\phi$. Then,
\[\Omega_{\bldmth{f},\bldmth{i}}=\left\{F^\bullet\in\Flag\left(\mathbb{C}^3\right):F^{1}\neq \xi^1_{\bldmth{f},\bldmth{i}}(t),\ F^2\neq \xi^2_{\bldmth{f},\bldmth{i}}(t)\ \text{ for all }t\in\mathbb{RP}^1\right\}\]
is a domain of discontinuity for $\rho_{\bldmth{f},\bldmth{i}}$.
By Proposition \ref{AMTWEquiv}, there is an $\SL(2,\mathbb{R})$--equivariant projection
\[p_{\bldmth{f},\bldmth{i}}:\Omega_{\bldmth{f},\bldmth{i}}\to\mathbb{H}^2.\]
Quotienting $\Omega_{\bldmth{f},\bldmth{i}}$ by the action of $\Gamma$ produces the quotient manifold
\[\mathcal{W}_{\bldmth{f},\bldmth{i}}=\Gamma\setminus\Omega_{\bldmth{f},\bldmth{i}}.\]
\par
For $\iota_{\bldmth{f},\bldmth{r}}=\rho_2\oplus\rho_1$, let $$\xi_{\bldmth{f},\bldmth{r}}:\mathbb{RP}^1\to \Flag(\mathbb{C}^3)$$ be the limit curve for the $\iota_{\bldmth{f},\bldmth{r}}$--Fuchsian representation $\rho_{\bldmth{f},\bldmth{r}}=\iota_{\bldmth{f},\bldmth{r}}\circ\phi$. Then,
\[\Omega_{\bldmth{f},\bldmth{r}}=\left\{F^\bullet\in\Flag\left(\mathbb{C}^3\right):F^{1}\neq \xi^1_{\bldmth{f},\bldmth{r}}(t),\ F^2\neq \xi^2_{\bldmth{f},\bldmth{r}}(t)\ \text{ for all }t\in\mathbb{RP}^1\right\}\]
is a domain of discontinuity for $\rho_{\bldmth{f},\bldmth{r}}$.
By Proposition \ref{AMTWEquiv}, there is an $\SL(2,\mathbb{R})$--equivariant projection
\[p_{\bldmth{f},\bldmth{r}}:\Omega_{\bldmth{f},\bldmth{r}}\to\mathbb{H}^2.\]
Quotienting $\Omega_{\bldmth{f},\bldmth{r}}$ by the action of $\Gamma$ produces the quotient manifold
\[\mathcal{W}_{\bldmth{f},\bldmth{r}}=\Gamma\setminus\Omega_{\bldmth{f},\bldmth{r}}.\]
\par
Let $G=\SL(4,\mathbb{C})$ and $\mathcal{F}_{\eta}=\mathbb{CP}^{3}$. From Example \ref{CP3RelPos}, we know that there is only one balanced ideal which describes a domain of discontinuity in $\mathbb{CP}^3$ and it requires the representation to be $2$--Anosov. By Proposition \ref{Anosovness4Fuchs}, an $\iota$--Fuchsian representation $\Gamma\to\SL(4,\mathbb{C})$ is $2$--Anosov if and only if it has a middle singular value gap. This is true if and only if $\iota=\rho_{d_1}\oplus\cdots\oplus\rho_{d_\ell}$ for a partition of $4$ with only even parts. Thus, the options are $\iota_{\bldmth{p},\bldmth{i}}=\rho_4$ and $\iota_{\bldmth{p},\bldmth{r}}=\rho_2\oplus\rho_2$. \par
For $\iota_{\bldmth{p},\bldmth{i}}=\rho_4$, let $$\xi_{\bldmth{p},\bldmth{i}}:\mathbb{RP}^1\to \Gr_2(\mathbb{C}^4)$$ be the limit curve for the $\iota_{\bldmth{p},\bldmth{i}}$--Fuchsian representation $\rho_{\bldmth{p},\bldmth{i}}=\iota_{\bldmth{p},\bldmth{i}}\circ\phi$. Then,
\[\Omega_{\bldmth{p},\bldmth{i}}=\left\{F^1\in\mathbb{CP}^3:\ F^1\nsubseteq \xi^2_{\bldmth{p},\bldmth{i}}(t)\ \text{ for all }t\in\mathbb{RP}^1\right\}\]
is a domain of discontinuity for $\rho_{\bldmth{p},\bldmth{i}}$.
By Proposition \ref{AMTWEquiv}, there is an $\SL(2,\mathbb{R})$--equivariant projection
\[p_{\bldmth{p},\bldmth{i}}:\Omega_{\bldmth{p},\bldmth{i}}\to\mathbb{H}^2.\]
Quotienting $\Omega_{\bldmth{p},\bldmth{i}}$ by the action of $\Gamma$ produces the quotient manifold
\[\mathcal{W}_{\bldmth{p},\bldmth{i}}=\Gamma\setminus\Omega_{\bldmth{p},\bldmth{i}}.\]
For $\iota_{\bldmth{p},\bldmth{r}}=\rho_2\oplus \rho_2$, let $$\xi_{\bldmth{p},\bldmth{r}}:\mathbb{RP}^1\to \Gr_2(\mathbb{C}^4)$$ be the limit curve for the $\iota_{\bldmth{p},\bldmth{r}}$--Fuchsian representation $\rho_{\bldmth{p},\bldmth{r}}=\iota_{\bldmth{p},\bldmth{r}}\circ\phi$. Then,
\[\Omega_{\bldmth{p},\bldmth{r}}=\left\{F^1\in\mathbb{CP}^3:\ F^1\nsubseteq \xi^2_{\bldmth{p},\bldmth{r}}(t)\ \text{ for all }t\in\mathbb{RP}^1\right\}\]
is a domain of discontinuity for $\rho_{\bldmth{p},\bldmth{r}}$.
By Proposition \ref{AMTWEquiv}, there is an $\SL(2,\mathbb{R})$--equivariant projection
\[p_{\bldmth{p},\bldmth{r}}:\Omega_{\bldmth{p},\bldmth{r}}\to\mathbb{H}^2.\]
Quotienting $\Omega_{\bldmth{p},\bldmth{r}}$ by the action of $\Gamma$ produces the quotient manifold
\[\mathcal{W}_{\bldmth{p},\bldmth{r}}=\Gamma\setminus\Omega_{\bldmth{p},\bldmth{r}}.\]
\par
If $G=\Sp(4,\mathbb{C})$ but the domain flag variety is still $\mathcal{F}_\eta=\mathbb{CP}^3$, then the only balanced ideal (which we saw in Example \ref{LineLagRelPos}) will still correspond to the domains $\Omega_{\bldmth{p},\bldmth{i}}$ and $\Omega_{\bldmth{p},\bldmth{r}}$. This is because both $\iota_{\bldmth{p},\bldmth{i}}$ and $\iota_{\bldmth{p},\bldmth{r}}$ admit invariant symplectic forms by Proposition \ref{invariantsymp} and thus $\rho_{\bldmth{p},\bldmth{i}}$ and $\rho_{\bldmth{p},\bldmth{i}}$ can be reinterpreted as $\{2\}$--Anosov representations into $\Sp(4,\mathbb{C})$ with limit curves
\[\xi_{\bldmth{p},\bldmth{i}},\xi_{\bldmth{p},\bldmth{r}}:\mathbb{RP}^1\to \Lag(\mathbb{C}^4).\]
To simplify our arguments, we will focus on the interpretation of $\rho_{\bldmth{p},\bldmth{i}}$ and $\rho_{\bldmth{p},\bldmth{r}}$ as representations into $\SL(4,\mathbb{C})$. \par
Let $G=\Sp(4,\mathbb{C})$ and $\mathcal{F}_{\eta}=\Lag(\mathbb{C}^4)$. By Example \ref{LineLagRelPos}, there is only one balanced ideal which describes a domain of discontinuity in $\Lag(\mathbb{C}^4)$ and it requires the representation to be $\{1\}$--Anosov. In order for $$\iota=\rho_{d_1}\oplus\cdots\oplus\rho_{d_\ell}:\SL(2,\mathbb{R})\to\SL(4,\mathbb{C})$$ 
to admit an invariant symplectic form (so that $\iota$ can be interpreted as representation into $\Sp(4,\mathbb{C})$), it must be that each odd part of the partition $(d_1,\cdots,d_\ell)$ has even multiplicty. The only such partition of $4$ are
\[(4), \ \ \ (2,2), \ \ \ (2,1,1), \ \ \ (1,1,1,1).\]
Amongst these partitions, only $(4)$ and $(2,1,1)$ will produce an $\{1\}$--Anosov $\iota$--Fuchsian representation. Thus, we get the representations
\begin{align*}
    \iota_{\bldmth{l},\bldmth{i}}=\rho_4&:\SL(2,\mathbb{R})\to\Sp(4,\mathbb{C})\\
    \iota_{\bldmth{l},\bldmth{r}}=\rho_2\oplus\rho_1\oplus\rho_1&:\SL(2,\mathbb{R})\to\Sp(4,\mathbb{C}).
\end{align*}
\par
For $\iota_{\bldmth{l},\bldmth{i}}=\rho_4$, let $$\xi_{\bldmth{l},\bldmth{i}}:\mathbb{RP}^1\to \mathbb{CP}^3$$ be the limit curve for the $\iota_{\bldmth{l},\bldmth{i}}$--Fuchsian representation $\rho_{\bldmth{l},\bldmth{i}}=\iota_{\bldmth{l},\bldmth{i}}\circ\phi$. Then,
\[\Omega_{\bldmth{l},\bldmth{i}}=\left\{L\in\Lag(\mathbb{C}^4):\ \xi^1_{\bldmth{l},\bldmth{i}}(t)\nsubseteq L\ \text{ for all }t\in\mathbb{RP}^1\right\}\]
is a domain of discontinuity for $\rho_{\bldmth{l},\bldmth{i}}$.
By Proposition \ref{AMTWEquiv}, there is an $\SL(2,\mathbb{R})$--equivariant projection
\[p_{\bldmth{l},\bldmth{i}}:\Omega_{\bldmth{l},\bldmth{i}}\to\mathbb{H}^2.\]
Quotienting $\Omega_{\bldmth{l},\bldmth{i}}$ by the action of $\Gamma$ produces the quotient manifold
\[\mathcal{W}_{\bldmth{l},\bldmth{i}}=\Gamma\setminus\Omega_{\bldmth{l},\bldmth{i}}.\]
For $\iota_{\bldmth{l},\bldmth{r}}=\rho_2\oplus\rho_1\oplus\rho_1$, let $$\xi_{\bldmth{l},\bldmth{r}}:\mathbb{RP}^1\to \mathbb{CP}^3$$ be the limit curve for the $\iota_{\bldmth{l},\bldmth{r}}$--Fuchsian representation $\rho_{\bldmth{l},\bldmth{r}}=\iota_{\bldmth{l},\bldmth{r}}\circ\phi$. Then,
\[\Omega_{\bldmth{l},\bldmth{r}}=\left\{L\in\Lag(\mathbb{C}^4):\ \xi^1_{\bldmth{l},\bldmth{r}}(t)\nsubseteq L\ \text{ for all }t\in\mathbb{RP}^1\right\}\]
is a domain of discontinuity for $\rho_{\bldmth{l},\bldmth{r}}$.
By Proposition \ref{AMTWEquiv}, there is an $\SL(2,\mathbb{R})$--equivariant projection
\[p_{\bldmth{l},\bldmth{r}}:\Omega_{\bldmth{l},\bldmth{r}}\to\mathbb{H}^2.\]
Quotienting $\Omega_{\bldmth{l},\bldmth{r}}$ by the action of $\Gamma$ produces the quotient manifold
\[\mathcal{W}_{\bldmth{l},\bldmth{r}}=\Gamma\setminus\Omega_{\bldmth{l},\bldmth{r}}.\]
\par
The preceding analysis produces a list of cases which will direct the rest of the paper. We summarize in the following proposition.
\begin{prop}\label{fullcases}
    Let $G$ be a simply-connected complex classical group with a $3$--dimensional flag variety $\mathcal{F}_{\eta}$. Let $\rho:\Gamma\to G$ be an $\iota$--Fuchsian representation. Assume that $\rho$ admits a cocompact domain of discontinuity in $\mathcal{F}_\eta$ constructed via a balanced ideal according to Kapovich-Leeb-Porti. Then, the options for $G$, $\mathcal{F}_\eta$, and $\iota$ are in a row of the following table:
    \[\begin{array}[row sep = 20pt]{c|c|c}
            G & \mathcal{F}_\eta  & \iota \\
             \hline \SL(3,\mathbb{C}) & \Flag\left(\mathbb{C}^3\right) & \rho_3 \text{ and } \rho_2\oplus\rho_1 \\
             \SL(4,\mathbb{C}), \Sp(4,\mathbb{C})& \mathbb{CP}^3 & \rho_4 \text{ and } \rho_2\oplus\rho_2\\
             \Sp(4,\mathbb{C}) & \Lag\left(\mathbb{C}^4\right) & \rho_4 \text{ and } \rho_2\oplus\rho_1\oplus \rho_1
        \end{array}\]
\end{prop}
\section{Classification of $S^1$--actions on $4$--manifolds}\label{sec:circleclass}
The main result of this section is a smooth version of Fintushel's classification of $S^1$--actions on closed simply-connected $4$--manifolds. This classification, originally proven in \cite{FintushelSimply4Circle}, states that such circle actions are determined by a weighted graph embedded in the orbit space of the $4$--manifold. Fintushel's classification is done in the topological category and the weight data attached to the graph encodes the $S^1$--action surrounding the exceptional orbits inside the $4$--manifold. For smooth actions, this weight data is more easily expressed in terms of the induced $S^1$--action on the tangent bundle. This is the approach used by Jang and Musin to analyze $S^1$--actions on (not necessarily simply-connected) $4$--manifolds in \cite{jang2023circleactionsoriented4manifolds}. We adapt their labeled graph to rework Fintushel's proof in order to classify smooth $S^1$--actions up to equivariant diffeomorphism. 

\subsection{Circle actions and orbit spaces}
\indent
For the remainder of Section \ref{sec:circleclass}, $M$ will be a simply-connected closed $4$--manifold with a smooth $S^1$--action. Because $S^1=\{g\in\mathbb{C}:|g|=1\}$ is a compact Lie group, we can choose an invariant Riemannian metric on $M$ so that $S^1$ acts by isometries. Any closed subgroup of $S^1$ looks like
\[H_m=\{z\in S^1:z^m=1\}\]
for some $m\geq 0$. \par
It will be useful to classify the (oriented) irreducible real representations of each $H_m$ (see \cite[Section 8]{Bröcker1985}). For this next proposition, we will specify an $H_m$--representation as an ordered pair $(V,\rho)$ where $V$ is a real vector space and $\rho:H_m\to \GL(V)$ is a homomorphism. The vector space $V$ will either be $\mathbb{R}$ or $\mathbb{C}$ with the standard orientations.

\begin{prop}\label{U(1)Class} For any integer $m\geq 0$,
    \begin{itemize}
        \item If $m=1$, then $H_1=\{1\}$ and the only irreducible $H_1$--representation is $(V^1_0,\rho^1_0)=(V=\mathbb{R}$, $\rho:H_1\to\mathbb{R}^\times)$ with $\rho(g)=1$.
    \item If $m\neq 0,1$, then $H_m\cong\mathbb{Z}_m$ and the irreducible $H_m$--representations are
    \begin{align*}
        (V_d^m,\rho_d^m)=\begin{cases}
        (V=\mathbb{R},\ \rho:H_m\to \mathbb{R}^\times)\text{ with }\rho(g)=g^d & d=0,\frac{m}{2}\\
        (V=\mathbb{C},\ \rho:H_m\to\mathbb{C}^\times)\text{ with }\rho(g)=g^d & \text{otherwise}
        \end{cases}
    \end{align*}
    for $d\in\{0,1,\hdots,m-1\}$.
    \item If $m=0$, then $H_m=S^1$ and the irreducible $S^1$--representations are
    \[(V^0_w,\rho_w^0)=\begin{cases}
        (V=\mathbb{R},\ \rho:H_m\to\mathbb{R}^\times)\text{ with }\rho(g)=1 & w=0\\
        (V=\mathbb{C},\ \rho:H_m\to\mathbb{C}^\times)\text{ with }\rho(g)=g^w & w\neq 0
    \end{cases}\]
    for $w\in\mathbb{Z}$. The integer $w$ is called the weight of the oriented representation $V_w^0$.
    \end{itemize}
\end{prop}
For the remainder of the paper, we will refer to the irreducible representations described in the above proposition by the vector space $V$ with the action implicit in our labeling, e.g.\ $V_1^2$ is shorthand for the representation $(V_1^2,\rho_1^2)$. Occasionally, when complex coefficients are explicitly used, $V^m_0$ will denote the trivial complex $H_m$--representation instead of the real representation. \par
Note that if we forget the orientation of $V_w^0$, then $V_w^0\cong V_{-w}^0$. In a direct sum of $S^1$--representations, negating an even number of weights will not change the orientation. This principle is summarized in the following proposition.
\begin{prop}\label{isoofU(1)reps}
    Let $w_1,\hdots,w_k$ and $w_1',\hdots,w'_k$ be non-zero integers. Then, there is an isomorphism of $S^1$--representations between $V^0_{w_1}\oplus\cdots\oplus V_{w_k}^0$ and $V^0_{w'_1}\oplus\cdots V_{w_k'}^0$ if and only if there is some permutation $\sigma$ of $\{1,\hdots, k\}$ such that $w_i=\pm w'_{\sigma(i)}$ for each $i\in \{1,\hdots, k\}$. This isomorphism is orientation-preserving if and only if $\sign(w_1\cdots w_k)=\sign(w_1'\cdots w_k')$.
\end{prop}
Now, we will examine the structure of the $S^1$--action $M$ based on the stabilizer of points. For $m\geq 0$, let $$\operatorname{Fix}(M,m)=\{x\in M:h\cdot x=x\text{ for }h\in H_m\}.$$
\begin{prop}\label{fixedappear}
     For each $m\geq 0$,
    \begin{enumerate}[label=\normalfont{(\roman*)}]
        \item The subset $\operatorname{Fix}(M,m)$ is $S^1$--invariant.
        \item Each connected component of $\operatorname{Fix}(M,m)$ is a totally geodesic submanifold of $M$.
        \item Each connected component of $\operatorname{Fix}(M,m)$ is orientable and of even codimension in $M$.
        \item The Euler characteristic of $\operatorname{Fix}(M,m)$ is the same as $M$, and  $\chi(M)=\chi(\operatorname{Fix}(M,m))$.
    \end{enumerate}
\end{prop}
\begin{proof}
    Statement (i) follows from the observation that points in the same orbit have the conjugate stabilizer subgroups and $S^1$ is abelian. Statement (ii) is a consequence of \cite[Theorem 5.1]{KobayashiTransform}. \par
    
    For (iii), let $C$ be a connected component of $\operatorname{Fix}(M,m)$ and choose a point $x\in C$. We consider the linear representation $H_m\to \Orth(T_xM)$. Because $H_m$ is a reductive group, we can decompose $T_xM$ into irreducible subrepresentations. By considering geodesics emanating from $x$, we conclude that $T_xC=\{v\in T_xM:h\cdot v=v\text{ for all }h\in H_m\}$. Thus, we just need to determine that the sum of the dimensions of the non-trivial irreducible subrepresentations of $T_xM$ is even. If $m\neq 2$, then $H_m\cong S^1$ or $\mathbb{Z}_m$, whose irreducible representations are either trivial or 2-dimensional. Thus, $T_xC=\{v\in T_xM:h\cdot v=v\text{ for all }h\in H_m\}$ must have even codimension in $T_xM$. If $m=2$, then $H_m=\mathbb{Z}_2$ has two irreducible representations, both $1$--dimensional. In this case, $H_m\cong\mathbb{Z}_2$ and so if $g\in H_m$ is the element of order $2$, then $T_xM$ decomposes into the eigenspaces for $g:T_xM\to T_xM$.
    \[T_xM=E_1\oplus E_{-1}\]
    where $E_{\pm1}=\{v\in T_xM:g\cdot v=\pm v\}$ and so $E_1=T_xC$. To conclude that $E_{-1}$ is even-dimensional, we note that $g$ is an element of the one-parameter family $S^1$ acting on $M$. In particular, $g:M\to M$ is isotopic to the identity and thus orientation-preserving. The linear transformation $g:T_xM\to T_xM$ preserving orientation forces $E_{-1}$ to be even-dimensional. The same considerations imply that normal bundle of $C$ has an $S^1$--invariant orientation. Because $M$ is oriented, this induces an orientation on $C$.  This completes the proof of (iii). \par
    For (iv), the case when $m=0$ is a direct consequence of \cite[Theorem 5.5]{KobayashiTransform}, i.e. $\chi(M)=\chi(F)$ where $F=\Fix(M,0)$.
    For $m>0$, we note that (i) \& (ii) means $\operatorname{Fix}(M,m)=\bigsqcup_{i=1}^{\ell}C_i$ where each $C_i$ is an $S^1$--invariant submanifold of $M$. Then, \cite[Theorem 5.5]{KobayashiTransform} implies $\chi(C_i)=\chi(C_i\cap F)$    \[\chi(\operatorname{Fix}(M,m))=\sum_{i=1}^{\ell}\chi(C_i)=\sum_{i=1}^{\ell}\chi(C_i\cap F)=\chi\left(\bigsqcup_{i=1}^{\ell}C_i\cap F\right)=\chi(F)=\chi(M).\]
\end{proof}
Note that the above proposition is true for any circle action on a connected orientable closed manifold $M$ (without assuming simply-connectedness or $\dim M=4$). Everything to follow will somehow use either simply-connectedness or $\dim M=4$. Most of the arguments are based on those used in \cite{FintushelSimply4Circle}.

\indent For the remainder of this section, we fix the notation (borrowed from \cite{FintushelSimply4Circle}) for fixed points and exceptional points:
\begin{align*}
    F&=\operatorname{Fix}(M,0) & E&=\bigcup_{m>0}\operatorname{Fix}(M,m)\setminus F.
\end{align*}
Let $M^\ast$ denote the orbit space of $M$ and $p:M\to M^\ast$ the orbit projection. For $A\subseteq M$, $A^\ast$ will be shorthand for $p(A)$. If, on the other hand, we first describe a subset $A^\ast\subseteq M^\ast$, then $A$ will be shorthand for $p^{-1}(A^\ast)$.
\begin{prop}\label{spherearc}
    For $m\geq 2$, if $C$ is a connected component of $\operatorname{Fix}(M,m)$ not contained in $\operatorname{Fix}(M,m')$ for any $m'\in m\mathbb{Z}_{\geq 0}$ other than $m$, then $C$ is equivariantly diffeomorphic to $\mathbb{CP}^1$ with the $S^1$--action $g\cdot [z_0:z_1]=[z_0:g^m\cdot z_1]$. In particular, $C^\ast$ is a closed arc, with endpoints $(C\cap F)^\ast$ and interior points $(C\setminus F)^\ast$. The interior points $C\setminus F$ have stabilizer $H_m$.
\end{prop}
\begin{proof}
    By Proposition \ref{fixedappear}(i,ii,iii), $C$ is an $S^1$--invariant submanifold of dimension $0$ or $2$. However, $C$ is not contained in $F$ and contains a proper $S^1$--orbit, hence is positive-dimensional. Thus, $C$ is a connected orientable surface. Furthermore, $C\cap F$ must have even codimension in $C$ and hence is a finite collection of points. Thus by Proposition \ref{fixedappear}(iv), we conclude $\chi(C)=\chi(C\cap F)\geq 0$. This means $C$ is either a sphere or a torus. By \cite[Lemma 2.3]{Montgomery1960GroupsO}, $C$ cannot be a torus. Therefore, $C$ is a sphere and $|C\cap F|=2$. The condition that $C$ is not contained in any smaller $\operatorname{Fix}(M,m')$ implies that the stabilizer of a point in $C\setminus F$ must be $H_m$. These properties characterize the action on $\mathbb{CP}^1\cong S^2$ described in the proposition.
\end{proof}
\begin{defn}
    A connected component of $\Fix(M,m)$ satisfying the assumptions of Proposition \ref{spherearc} will be called an exceptional sphere of weight $m$.
\end{defn}
 The following result is a first step to decomposing a manifold using the $S^1$--action on the tangent bundle.
\begin{prop}\label{BredonNeighbor}
    Let $N$ be a closed $S^1$--invariant submanifold of $M$ and $T^\perp N$ its normal bundle with respect to our chosen $S^1$--invariant metric on $M$. Let $\exp:TM\to M$ be the exponential map. Then, there is some $\epsilon>0$ such that $\exp$ restricted to
    \[B_\epsilon(T^\perp N)=\{(x,v)\in T^{\perp}N:||v||\leq\epsilon\}\]
    is an $S^1$--equivariant diffeomorphism onto a compact neighborhood of $N$. In particular, $\exp$ restricted to
    \[\mathring{B}_\epsilon(T^\perp N)=\{(x,v)\in T^{\perp}N:||v||<\epsilon\}\]
    is also an $S^1$--equivariant diffeomorphism onto an open neighborhood of $N$.
\end{prop}
\begin{proof}
    See Section 2 of \cite{BredonVI}.
\end{proof}
\begin{lem}\label{orbitneigh} For each $x\in M$, let $O=S^1\cdot x$ denote its orbit and $H_m$ be its stabilizer. Then, the map $\nu:H_m\setminus B_\epsilon(T^{\perp}_xO)\to M^{\ast}$ defined by
\[[(x,v)]\mapsto p(\exp(x,v))\]
is a homeomorphism onto a compact neighborhood of $x^\ast$.
\end{lem}
\begin{proof}
    Because the sets involved are compact and Hausdorff, it suffices to show that $\nu$ is a continuous and injective. First, we note the map $p\circ\exp:B_{\epsilon}(T_x^\perp O)\to M^\ast$ is a composition of continuous functions and thus continuous. By Proposition \ref{BredonNeighbor}, $\exp$ is $H_m$--equivariant. Therefore, $p\circ\exp$ sends $H_m$--orbits to the same point in $M^\ast$. Thus, $p\circ \exp$ descends to the \textit{continuous} map $\nu:H_m\setminus B_{\epsilon}(T^\perp_x O)\to M^\ast$. Now, we show injectivity. Take an arbitrary $(x,v_1)$, $(x,v_2)$ such that
\[\nu([(x,v_1])=p(\exp(x,v_1))=p(\exp(x,v_2))=\nu([(x,v_2)]).\]
Then, there must be some $g\in S^1$ such that $g\exp(x,v_1)=\exp(x,v_2)$. Because $\exp$ is $S^1$--equivariant and injective, this means $g\in H_m$ and $g\cdot v_1=v_2$. Of course, this implies $[(x,v_1)]=[(x,v_2)]$ in $H_m\setminus B_{\epsilon}(T^\perp_xO)$. Therefore, $\nu$ is a homeomorphism onto its image. To see that the image of $\nu$ is a neighborhood of $x^\ast$, we note that
\[\nu(H_m\setminus\mathring{B}_\epsilon(T_x^{\perp}O))=p(\exp(\mathring{B}_\epsilon(T^\perp O)))\]
and $\exp(\mathring{B}_\epsilon(T^\perp O))$ is an $S^1$--invariant open neighborhood of $x$.
\end{proof}
This gives us a rather direct way to understand the local topology of $M^\ast$ around $x^\ast$ in terms of the action on the $T_x^\perp O$. We now explicitly write this out.
\begin{prop}\label{localtopoM*}
    For each $x\in M$, let $O=S^1\cdot x$ denote its orbit and $H_m$ its stabilizer. For the different cases, we describe the class of $T_x^\perp O$ as a $H_m$--representation.
    \begin{enumerate}[label=\normalfont{(\alph*)}]
        \item If $x\notin E\cup F$, then $H_m=\{1\}$ and $T^\perp_x O$ is isomorphic to $V^1_0\oplus V^1_0\oplus V^1_0$. In this case, $M^\ast$ is locally homeomorphic to $\mathbb{R}^3$ around $x^\ast$.
        \item If $x\in E\setminus F$, then $H_m$ is a non-trivial cyclic group and $T^{\perp}_xO$ is isomorphic to $V^m_0\oplus V_d^m$ for some $d\in\{1,\hdots,m-1\}$. In this case, $M^\ast$ is locally homeomorphic to $\mathbb{R}^3$ around $x^\ast$.
        \item If $x\in F$ is an isolated fixed point, then $H_m=S^1$ and $T_x^\perp O\cong V_{w_1}^0\oplus V_{w_2}^0$ for some coprime integers $w_1,w_2\neq0$. In this case, $M^\ast$ is locally homeomorphic to $\mathbb{R}^3$ around $x^\ast$.
        \item If $x\in F$ and lies on a fixed surface, then $H_m=S^1$ and $T_x^{\perp}O\cong V_0^0\oplus V_0^0\oplus V_1^0$. In this case, $M^\ast$ is locally homeomorphic to $\mathbb{C}\times\mathbb{R}_{\geq0}$ around $x^\ast$.
    \end{enumerate}
\end{prop}
\begin{proof}
    \textbf{Case (a).} If $x\notin E\cup F$, then $x$ is not fixed by any trivial element and thus its stabilizer $H_m$ is trivial. Then, $O=S^1\cdot x$ is an embedded circle which $S^1$ acts freely on. Then, 
    $T_x^\perp O$ is a $3$--dimensional vector space which $H_m=\{1\}$ acts trivial on. By Lemma \ref{orbitneigh}, $x^\ast$ has a neighborhood homeomorphic to
    \[H_m\setminus B_\epsilon(T_x^\perp O)\cong B_\epsilon(\mathbb{R}^3)\]
    where $x^\ast$ corresponds to the origin. \par 
    \textbf{Case (b).} If $x\in E\setminus F$, then $H_m$ must be a non-trivial proper subgroup of $S^1$, i.e.\ $m\neq 0,1$. Therefore, the orbit $O=S^1\cdot x$ is an embedded circle. By Proposition \ref{spherearc}, $O$ is a latitudinal line of an exceptional sphere $C$ of weight $m$. Because $C$ is fixed by $H_m$, $T_xC$ is isomorphic to $V^m_0\oplus V_0^m$ as an $H_m$--representation. Then, the tangent space $T_xM$ has the $S^1$--invariant decomposition
    $$T_xM=T_xC\oplus T_x^\perp C\cong (V_0^m\oplus V_0^m)\oplus V_d^m$$
    for some $d\in\{0,1,\hdots,m-1\}$. If $d=0$, then $H_m$ acts trivially on $T_xM$ and this would mean the action on $C$ would not be the entire component of $\Fix(M,m)$. Therefore, $d\in\{1,\hdots,m-1\}$. Note that $T_xO\subseteq T_xC$ and thus
    \[T_x^\perp O=(T_x^\perp O\cap T_x C)\oplus T^\perp_x C\cong V_0^m\oplus V_d^m.\]
    By Lemma \ref{orbitneigh}, there is a neighborhood of $x^\ast$ homeomorphic to
    \[H_m\setminus B_\epsilon(V_0^m\oplus V_d^m).\]
    Note that the action of $H_m$ on $V_0^m\oplus V_d^m$ looks like $\mathbb{R}\times\mathbb{C}$ with the action
    \[g\cdot (t,z)=(t,g^dz)\text{ for }g\in H_m, \ (t,z)\in\mathbb{R}\times\mathbb{C}.\]
    Let $k=\frac{m}{\gcd(m,d)}$. Then,
    \[D=\left\{(t,r\cos\theta+ir\sin\theta):t\in\mathbb{R},r\geq 0, \theta\in\left[0,\frac{2\pi}{k}\right]\right\}\]
    is a wedge-shaped fundamental domain for the action of $H_m$. The quotient $H_m\setminus D$ is homeomorphic to $\mathbb{R}^3$. Thus, $M^\ast$ is locally homeomorphic to $\mathbb{R}^3$ at $x^\ast$. \par
    \textbf{Case (c).} If $x\in F$ is an isolated fixed point, then $H_m=S^1$ and $$O=S^1\cdot x=\{x\}.$$
    Therefore, $T_x^\perp O=T_xM$ which decomposes into non-trivial $S^1$--representations:
    $T_xM\cong V_{w_1}^0\oplus V_{w_2}^0$. Note that $H_{w_1}\cap H_{w_2}$ acts trivial on $T_xM$ and so because the action on $M$ is effective, it must be that $H_{w_1}\cap H_{w_2}=\{1\}$. In other words, $w_1$ and $w_2$ are coprime. Because $x$ is isolated fixed point, neither $w_1$ nor $w_2$ can be $0$. By Lemma \ref{orbitneigh}, $x^\ast$ has a neighborhood homeomorphic to
    \[S^1\setminus B_\epsilon(V_{w_1}^0\oplus V_{w_2}^0)\]
    Note that by linearity of the $S^1$--action, $B_\epsilon(V_{w_1}^0\oplus V_{w_2}^0)$ is the $S^1$--equivariant cone of 
    \[\del B_\epsilon(V_{w_1}^0\oplus V_{w_2}^0)\]
    which is a $3$--manifold whose orbits give a Seifert fibering. By \cite[Section 3]{SeifertFiberOriginal}, the orbit space 
    \begin{align*}
    S^1\setminus\del B_\epsilon(V_{w_1}^0\oplus V_{w_2}^0)\cong S^2
        \intertext{ and thus by the cone}
        S^1\setminus B_\epsilon(V_{w_1}^0\oplus V_{w_2}^0)\cong B^3.
    \end{align*}
    Thus, $x^\ast$ has a Euclidean neighborhood in this case. \par
    \textbf{Case (d).} If $x$ lies on a fixed surface, i.e. a $2$--dimensional component $S$ of $\Fix(M,0)$, then $H_m=S^1$. Note that $T_x^\perp O=T_xM$ has the $S^1$--invariant decomposition
    \[T_xM=T_xS\oplus T_x^\perp S=(V_0^0\oplus V_0^0)\oplus V_{1}^0.\]
    The $S^1$--action on $V_0^0\oplus V_0^0\oplus V_1^0$ looks like $\mathbb{C}^2$ with the action
    \[g\cdot (z_1,z_2)=(z_1,gz_2)\text{ for }g\in S^1,\ (z_1,z_2)\in\mathbb{C}^2.\]
    Then, the closed subset $\mathbb{C}\times\mathbb{R}_{\geq0}$ intersects each $S^1$--orbit exactly once. One can thus show that the orbit space
    \[S^1\setminus (V_0^0\oplus V_0^0\oplus V_1^0)\]
    is homemorphic to $\mathbb{C}\times\mathbb{R}_{\geq0}$ and $x^\ast$ corresponds to the origin $(0,0)$.
\end{proof}
The above proposition implies the following corollary which puts strong restrictions on the topology of the orbit space (Compare \cite[Prop.\ 3.1]{FintushelSimply4Circle}.
\begin{corr}\label{simplyM*}
    The orbit space $M^\ast$ is a compact simply-connected $3$--manifold with boundary points corresponding to points on fixed surfaces in $M$.
\end{corr}
\begin{proof}
    All that remains to prove is that $\pi_1(M^\ast)$ is trivial. Pick some basepoint $x_0\notin E\cup F$. Then, take an arbitrary loop $\gamma:[0,1]\to M^\ast$ based at $x_0$. First, we can homotopy $\gamma$ away from the boundary. By \ref{spherearc}, $(E\cup F)^\ast\setminus\del M^\ast$ is a union of a finite collection of points and arcs with at most two arcs adjacent to each point. By taking into account Proposition \ref{localtopoM*}, we see $(E\cup F)^\ast\setminus\del M^\ast$ is actually an embedded graph in $M^\ast$. If we homotopy $\gamma$ to be transverse to this graph, the low codimensions make their intersection empty. Thus after taking homotopy, we can assume $\gamma$ is a loop in $(M\setminus (E\cup F))^\ast$. However, note that this is the base space of the principal $S^1$--bundle.
    \[p:M\setminus (E\cup F)\to (M\setminus E\cup F)^\ast.\]
    Because the fiber $S^1\cdot x_0$ is path-connected, $p_\ast:\pi_1(M\setminus (E\cup F),x_0)\to\pi_1((M\setminus(E\cup F))^\ast,x_0^\ast)$ is surjective. Thus after taking homotopy, we can assume that $\gamma:[0,1]\to (M\setminus(E\cup F))^\ast$ lifts to a loop $\tilde{\gamma}:[0,1]\to M\setminus(E\cup F)$ based at $x_0$. We know $\tilde{\gamma}$ is null-homotopic in $M$ and thus $\gamma$ is null-homotopic in $M^\ast$.
\end{proof}
\subsection{Tangential weight graph}
The goal of this section is describe a new labeling scheme, based on the multigraphs used in \cite{Jang_2018}, for the vertices and edges of the orbit data graph of smooth actions in terms of the action of the tangent space of fixed points. We then use this new labeling scheme to prove a smooth analog of Fintushel's classification in a special case.

 \begin{defn}
    Let $x\in F$ be an isolated fixed point. By Proposition \ref{localtopoM*}, $T_xM$ is isomorphic to $V^1_{w_1}\oplus V^1_{w_2}$ for some coprime integers. By Proposition \ref{isoofU(1)reps}, $\sign(w_1 w_2)\in\{\pm1\}$ is determined by the orientation and circle action on $T_xM$. We call $\sign(w_1 w_2)$ the sign of $x$. Abusing notation, we will also use the $(\pm)$ symbols to specify signs.\\
    Choose non-zero $v_1\in V_{w_1}^0$, $v_2\in V_{w_2}^0$, and non-trivial $g\in S^1$ so that $$v_1,\ g\cdot v_1, \ v_2,\ g\cdot v_2$$ 
    is an ordered basis for $T_xM$ and defines an orientation which we call the rotation orientation of $T_xM$. This orientation is independent of the choices made. The sign of $x$ is $(+)$ if the rotation orientation agrees with the natural orientation on $T_xM$ and $(-)$ if they disagree.
\end{defn}
Recall that from Proposition \ref{spherearc}, any point $x$ in $M$ with stabilizer $H_w$ lies on an exceptional sphere of weight $w$, which is a submanifold equivarianyly diffeomorphic to $\mathbb{CP}^1$ with the $S^1$--action $$g\cdot [z_0:z_1]=[z_0:g^m\cdot z_1]$$.
These data give us the labels we will use to define the tangential weight graph of $M$ as follows:
 \begin{defn}
     Suppose $M$ is a simply-connected $4$--manifold with an $S^1$--action such that $(E\cup F)^\ast\setminus \del M^\ast$ contains no polyhedral cycles (equivalenly $H_1(E\cup F)=0$). Then, we say $M$ has an $S^1$--action without cycles. In this case, we can consider $(E\cup F)^\ast$ as an abstract labeled graph independent of its embedding in $M^\ast$. The tangential weight graph $\mathcal{G}$ of $M$ is a labeled graph as follows: The vertices of $\mathcal{G}$ are divided into two classes, round vertices which correspond to isolated fixed points and square vertices that correspond to fixed surfaces. An edge between two round vertices corresponds to an exceptional sphere containing the corresponding isolated fixed points. There are no edges connected to square vertices.
     \begin{enumerate}
         \item A round vertex corresponding to an isolated fixed point $x$ is labeled by the sign of $x$.
         \item A square vertex corresponding to a fixed surface component $S$ is labeled with the Euler number $e\in\mathbb{Z}$ of the principal $S^1$--bundle $U^\perp S$ where $S\cong S^\ast$ is oriented as a boundary component of $\del M^\ast$.
         \item An edge correspond to an exceptional sphere $C$ is labeled with the weight $w\geq 2$ of $C$.
     \end{enumerate}
 \end{defn}
\begin{xmpl}\label{ConnSumExample}
    Consider the $S^1$--action on $\mathbb{CP}^1\times\mathbb{CP}^1$ given by
    \[g\cdot ([z_0:z_1],[w_0:w_1])=([z_0:gz_1],[w_0:g^2w_1]).\]
     To see that its tangential weight graph is the one given in Figure \ref{fig:21actionTWG}, note that $\{[1:0]\}\times\mathbb{CP}^1$ and $\{[0:1]\}\times\mathbb{CP}^1$ are two exceptional spheres of weight $2$. 
     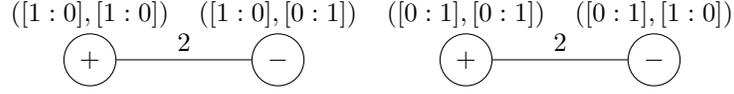
\begin{figure}
         \centering
         \begin{tikzpicture}[node distance={30mm}, main/.style = {draw, circle}]
\node[main] (1) at (0,0) {$+$}; 
\node[main] (2) at (2.5,0)  {$-$}; 
\node[main] (3) at (5,0) {$+$}; 
\node[main] (4) at (7.5,0)  {$-$}; 
\draw[] (1) -- node[midway, above,] {$2$} (2);
\draw[-] (3) -- node[midway, above,] {$2$} (4);

\node[] at (0,0.6) {$([1:0],[1:0])$}; 
\node[] at (2.5,0.6) {$([1:0],[0:1])$}; 
\node[] at (5,0.6) {$([0:1],[0:1])$}; 
\node[] at (7.5,0.6) {$([0:1],[1:0])$};
\end{tikzpicture}
         \caption{Tangential weight graph for $S^1$--action on $\mathbb{CP}^1\times\mathbb{CP}^1$}
         \label{fig:21actionTWG}
     \end{figure}
     To calculate the sign at, e.g. $p=([1:0],[0:1])$, we can decompose the tangent space into each factor
\[T_{p}\mathbb{CP}^1\times\mathbb{CP}^1=T_{[1:0]}\mathbb{CP}^1\times T_{[0:1]}\mathbb{CP}^1\cong V^0_1\oplus V_{-2}^0\]
and thus the sign at $p$ is $(-)$. \par
Let $M^+$ and $M^-$ be two copies of $\mathbb{CP}^1\times\mathbb{CP}^1$ with the above $S^1$--action. Let $x^+$ be the point corresponding to $([1:0],[1:0])$ in $M^+$ and $x^-$ be the point corresponding to $([1:0],[0:1])$ in $M^-$. By Proposition \ref{BredonNeighbor}, there is a ball neighborhood $B^+$ of $x^+$ whose boundary is equivariantly diffeomorphic to the unit sphere in $T_{x^+}M^+\cong V_1^0\oplus V_2^0$. Similarly, there is a ball neighborhood $B^-$ of $x^-$ whose boundary is equivariantly diffeomorphic to the unit sphere in $T_{x^-}M^-\cong V_{1}^0\oplus V_{-2}^0$. Then, the orientation-reversing isomorphism $V_{2}^0\cong V_{-2}^0$ determines an orientation-reversing $S^1$--equivariant diffeomorphism $\varphi:\del B^+\to \del B^-$. Gluing $M^+\setminus \mathring{B}^+$ and $M^-\setminus \mathring{B}^-$ gives an $S^1$--action on $M^+\# M^-\cong (S^2\times S^2)\# (S^2\times S^2)$. In this construction, we take the connected sum of the weight $2$ $\mathbb{CP}^1$ containing $x^+$ and the weight $2$ $\mathbb{CP}^1$ containing $x^-$. Thus, we are left with three weight $2$ $\mathbb{CP}^1$'s. See Figure \ref{fig:equiconnsum}.
\begin{figure}
    \centering
    \begin{center}
    \begin{tikzpicture}[node distance={30mm}, main/.style = {draw, circle}] 
\node[main] (-1) at (-4.5,0) {$+$}; 
\node[main] (1) at (-3,0)  {$-$}; 
\node[main] (2) at (-1.5,0) {$+$};
\node[main] (3) at (0,0)  {$-$};
\draw[] (-1) -- node[midway, above,] {$2$} (1);
\draw[-] (2) -- node[midway, above,] {$2$} (3);
\node[] at (0,0.7) {$x^-$};

\draw[rounded corners] (-5,1.5) rectangle (0.75,-0.75);
\draw[rounded corners] (7,1.5) rectangle (1.25,-0.75);
\draw[rounded corners] (5.5,-1.75) rectangle (-3.5,-4);

\node[main] (4) at (2,0) {$+$}; 
\node[main] (5) at (3.5,0)  {$-$};
\node[main] (6) at (5,0) {$+$};
\node[main] (7) at (6.5,0)  {$-$};
\draw[] (4) -- node[midway, above,] {$2$} (5);
\draw[-] (6) -- node[midway, above,] {$2$} (7);
\node[] at (2,0.6) {$x^+$};
\draw[dashed] (0,0) circle [radius = 0.45cm];
\draw[dashed] (2,0) circle [radius = 0.45cm];
\node[] at (-2.25,1) {$S^2\times S^2$};
\textbf{\node[] at (4.25,1) {$S^2\times S^2$};}

\draw[->] (1,-0.5) -- (1,-1.5);

\node[main] (1) at (-2.75,-2.5) {$+$}; 
\node[main] (2) at (-1.25,-2.5)  {$-$}; 
\node[main] (3) at (0.25,-2.5) {$+$}; 
\node[main] (4) at (1.75,-2.5)  {$-$}; 
\node[main] (5) at (3.25,-2.5) {$+$}; 
\node[main] (6) at (4.75,-2.5)  {$-$}; 
\draw[] (1) -- node[midway, above,] {$2$} (2);
\draw[-] (3) -- node[midway, above,] {$2$} (4);
\draw[-] (5) -- node[midway, above,] {$2$} (6);

\node[] at (1,-3.5) {$(S^2\times S^2)\#(S^2\times S^2)$}; 

\end{tikzpicture}
\end{center}
    \caption{Equivariant connected sum of $S^1$--actions on $S^2\times S^2$}
    \label{fig:equiconnsum}
\end{figure}
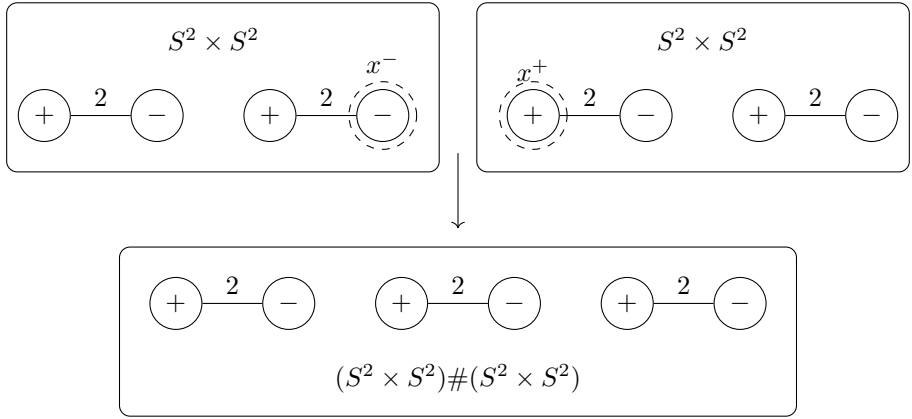
\end{xmpl}
 \subsection{Classification Theorem}
 Our goal is to prove the smooth classification of circle actions on $4$--manifolds in this special case. We will, for the most part, follow the strategy of \cite{FintushelSimply4Circle}, but using smooth techniques.
 \begin{thm}\label{SmoothClassification}
     Let $M^\sigma$ and $M^\tau$ be closed simply-connected $4$--manifolds with effective smooth $S^1$--actions without cycles and respective tangential weight graphs $\mathcal{G}^\sigma$ and $\mathcal{G}^\tau$. Suppose there is a graph isomorphism between $\mathcal{G}^\sigma$ and $\mathcal{G}^\tau$ that preserves the labeling of vertices and edges. Then, there is an orientation-preserving $S^1$--equivariant diffeomorphism $\varphi:M^\sigma\to M^\tau$.
 \end{thm}
 Before we prove Theorem \ref{SmoothClassification}, we will need the following lemmas.
 \begin{lem}\label{NeighborhoodData}
 	Let $M^{\sigma}$ and $M^\tau$ be closed simply-connected $4$--manifolds with effective $S^1$--actions. Let $E^\sigma$, $E^\tau$, $F^\sigma$, $F^\tau$ be the corresponding exceptional and fixed sets. Assume there are connected components $C^\sigma=C_1^\sigma\cup\cdots C_k^\sigma$ and $C^\tau=C_1^\tau\cup\cdots C_k^\tau$ of $E^\sigma\cup F^\sigma$ and $E^\tau\cup F^\tau$ that each correspond to a polyhedral arcs with the labeling
     \begin{center}
    \begin{tikzpicture}[node distance={17mm}, main/.style = {draw, circle}] 
\node[main] (1) {$\epsilon_0$}; 
\node[main] (2) [right of=1] {$\epsilon_1$}; 
\node[] (3) [right of=2] {$\cdots$};
\node[main] (4) [right of=3] {$\epsilon_k$};
\draw[-] (1) -- node[midway, above] {$w_1$} (2);
\draw[-] (2) -- node[midway, above] {$w_2$} (3);
\draw[-] (3) -- node[midway, above] {$w_k$} (4);
\end{tikzpicture} 
\end{center}
so that $C_i^\sigma$ (resp. $C_i^\tau$) is  an exceptional sphere of weight $w_i$ which contains isolated fixed points $x_{i-1}^{\sigma}$, $x_{i}^\sigma$ (resp. $x_{i-1}^\tau$, $x_i^\tau$) with signs $\epsilon_{i-1}$, $\epsilon_i$.  Then, there are $S^1$--invariant open neighborhoods $U^\sigma$ and $U^\tau$ of $C^\sigma$ and $C^\tau$ with a orientation-preserving $S^1$--equivariant diffeomorphism $f:U^\sigma\to U^\tau$. 
 \end{lem}

\begin{lem}\label{BundleIsotopyLemma}
    Let $A$ be a smooth principal $S^1$--bundle over $S^2$. Let $f:A\to A$ be a bundle morphism that preserves each fiber. Then, there is a smooth $S^1$--equivariant isotopy from $f$ to $\id$.
\end{lem}

\begin{proof}[Proof of Theorem \ref{SmoothClassification} (assuming Lemmas \ref{NeighborhoodData}, \ref{BundleIsotopyLemma})]
    Our strategy as in \cite{FintushelSimply4Circle} is to construct our equivariant diffeomorphism $\varphi:M^\sigma\to M^\tau$ around the fixed and exceptional points using Lemma \ref{NeighborhoodData} and extend to all points using Lemma \ref{BundleIsotopyLemma}. \par
    Let $E^\sigma\cup F^\sigma$ and $E^\tau\cup F^\tau$ be the set of exceptional and fixed points in $M^\sigma$ and $M^\tau$ respectively. Let $Q_1^\sigma,\cdots Q_k^\sigma$ be the connected components of $E^\sigma \cup F^\sigma$ corresponding to polyhedral arcs (including isolated points), and with corresponding connected components $Q_1^\tau,\cdots Q_k^\tau$ for $E^\tau\cup F^\tau$. By Lemma \ref{NeighborhoodData}, we know that, for each $i\in\{1,\hdots,k\}$, there are neighborhoods $U_i^\sigma$ of $Q_i^\sigma$ and $U_i^\tau$ of $Q_i^\tau$ with an $S^1$--equivariant diffeomorphism $f_i:U_i^\sigma\to U_i^\tau$. For each $i\in\{1,\hdots,k\}$, choose a closed ball neighborhood $(B_i^\sigma)^\ast$ of $(Q_i^\sigma)^\ast$ contained in $(U_i^\sigma)^\ast$ such that its boundary is a smooth sphere inside $(M^\sigma\setminus E^\sigma\cup F^\sigma)^\ast$. Let $B_i^\sigma$ be the union of the orbits inside $(B_i^\sigma)^\ast$. We define $B_i^\tau=f_i(B_i^\sigma)$. \par
    Let $S_1^\sigma,\hdots, S_\ell^\sigma$ be the connected components of $E^\sigma\cup F^\sigma$ corresponding to fixed surfaces with the corresponding connected components $S_1^\tau,\hdots, S_\ell^\tau$ of $E^\tau\cup F^\tau$. For a sufficiently small $\epsilon'''>0$, the closed $\epsilon'''$--neighborhoods of $S_1^\sigma,\hdots, S_\ell^\sigma$ (resp. $S_1^\tau,\hdots,S_\ell^\tau$) are disjoint from $\bigcup_{i=1}^k B_i^\sigma$ (resp, $\bigcup_{i=1}^k B_i^\sigma$). Let $Z_1^\sigma,\hdots,Z_\ell^\sigma$ and $Z_1^\tau,\hdots,Z_\ell ^\tau$ be these $\epsilon'''$--neighborhoods. Note that each $\del Z_i^\sigma$ and $\del Z_i^\tau$ is a principal $S^1$--bundle over a sphere with the structure group coming from our original $S^1$--action. Then, $Z_i^\sigma$ and $Z_i^\tau$ are produced from $\del Z_i^\sigma$ and $\del Z_i^\tau$ by a change of fiber. Because each pair $Z_i^\sigma$ and $Z_i^\tau$ have the same Euler number, there is an orientation-preserving $S^1$--equivariant bundle equivalence $\alpha_i:Z_i^\sigma\to Z_i^\tau$. This finishes the construction of the equivariant map $\varphi$ on an neighborhood around the fixed and exceptional points.\\
    
    Then,
    \begin{align*}
    	P^\sigma&=M^\sigma\setminus\left(\bigcup_{i=1}^{k} \mathring{B}_i^\sigma\cup \bigcup_{i=1}^\ell \mathring{Z}_i^\sigma\right) \\
	P^\tau&=M^\tau\setminus\left(\bigcup_{i=1}^{k} \mathring{B}_i^\tau\cup \bigcup_{i=1}^\ell \mathring{Z}_i^\tau\right)
    \end{align*}
    are $4$--manifolds with free $S^1$ action and boundaries
     \begin{align*}
    	\del P^\sigma&=\bigcup_{i=1}^{k} \del{B}_i^\sigma\cup \bigcup_{i=1}^\ell \del{Z}_i^\sigma\\
	\del P^\tau&=\bigcup_{i=1}^{k} \del{B}_i^\tau\cup \bigcup_{i=1}^\ell \del{Z}_i^\tau
    \end{align*}
Then, we can think of $P^\sigma,P^\tau$ and their $S^1$--invariant submanifolds can be thought of as smooth principal $S^1$--bundles over their orbit spaces. Let $f^\ast_i:(B_i^\sigma)^\ast\to (B_i^\sigma)^\ast$ and $\alpha_i^\ast:(Z_i^\sigma)^\ast\to (Z_i^\tau)^\ast$ denote the induced maps on the orbit space. Note, as in Corollary \ref{simplyM*}, that $(P^\sigma)^\ast$ and $(P^\tau)^\ast$ are each compact simply-connected $3$--manifolds with diffeomorphic boundaries. By the Poincar\'e conjecture (proven in \cite{perelman2002entropyformularicciflow}), there is some diffeomorphism $\varphi^\ast|_{(P^\sigma)^\ast}:(P^\sigma)^\ast\to (P^\tau)^\ast$ extending the $f_i^\ast$ and $\alpha_i$ from $(\del P^\sigma)^\ast$ to $(P^\sigma)^\ast$. Then, $\varphi^\ast|_{(\del P^\sigma)^\ast}:(\del P^\sigma)^\ast \to (\del P^\tau)^\ast$ is constructed specifically so that the Euler class of the principal bundle $\del P^\tau$ is pulled back to the Euler class of the principal bundle $\del P^\sigma$. If $P$ is a simply-connected $3$--manifold with boundary $\del P$, then from Poincare duality $H^2(P,\del P)\cong H_1(P)=0$ and the exact sequence
\[H^2(P,\del P;\mathbb{Z})\to H^2(P;\mathbb{Z})\to H^2(\del P;\mathbb{Z})\]
we conclude that the map on cohomology map $H^2(P;\mathbb{Z})\to H^2(\del P;\mathbb{Z})$ is injective. Applying this to $(P^\sigma)^\ast$ and $(P^\tau)^\ast$, we conclude that the Euler class of the principal bundle $P^\tau$ is pulled back to $P^\sigma$ through $\varphi^\ast|_{(P^\sigma)^\ast}$. Therefore, there is some smooth ($S^1$--equivariant) bundle equivalence $\psi:P^\sigma\to P^\tau$ that induces the map $\varphi|_{(P^\sigma)^\ast}$ on the orbit spaces. By Lemma \ref{BundleIsotopyLemma}, there is a smooth $S^1$--equivariant isotopy from $\psi|_{\del P^\sigma}:\del P^\sigma\to \del P^\tau$ to $\varphi|_{\del P^\sigma}:\del P^\sigma\to P^\tau$. Using a collared neighborhood, we can extend this to a smooth isotopy from $\psi:P^\sigma\to P^\tau$ to an extension of $\varphi|_{\del P^\sigma}$ to an $S^1$--equivariant diffeomorphism $\varphi|_{P^\sigma}:P^\sigma\to P^\tau$. Then, combining $\varphi|_{P^\sigma}$ with $f_i$ and $\alpha_i$ gives our $S^1$--equivariant diffeomorphism between $M^\sigma$ and $M^\tau$.
\end{proof}

Lemma \ref{BundleIsotopyLemma} is a smooth version of \cite[Lem 6.1]{FintushelSimply4Circle} with the same proof. Thus, all that remains is to prove Lemma \ref{NeighborhoodData}.

\begin{proof}[Proof of Lemma \ref{NeighborhoodData}]
The goal is to find a suitable $S^1$--invariant metric on $M^\sigma$ and $M^\tau$ and an $\epsilon''>0$ so that we can construct $S^1$--equivariant diffeomorphisms 
$$f_i:N_i^\sigma\to N_i^\tau$$
between the $\epsilon''$--neighborhoods $N_i^\sigma$ and $N_i^\tau$ of $C_i^\sigma$ and $C_i^\tau$ for each $i$. Note that if $\epsilon''$ is small enough, then the nearest point projections 
$$\pi_i^\sigma:N_i^\sigma\to C_i^\sigma$$
are smooth fiber bundle projections. In order to ensure that these diffeomorphisms agree on the intersection of the neighborhoods, special care is needed when we choose our metric and $\epsilon''$. We want the maps $\pi_i^\sigma$ and $\pi_{i+1}^\sigma$, when restricted to $X_i^\sigma=N_i^\sigma\cap N_{i+1}^\sigma$, to be isometric to the projection maps $(x,y)\mapsto x$ and $(x,y)\mapsto y$ where $(x,y)\in\mathbb{R}^4$ with $||x||,||y||<\epsilon''$. Now, we will see how $\epsilon''$ can be chosen.

\vspace{.1in}\noindent\textbf{Step 1 -- Choosing $\epsilon''$:}

\noindent By Proposition \ref{BredonNeighbor}, there is an $\epsilon>0$ such that $\exp:TM^{\sigma}\to M^\sigma$ restricted to $B_\epsilon(T^\perp\{x_0^\sigma,\hdots,x_k^\sigma\})$ is an $S^1$--equivariant diffeomorphism onto a compact neighborhood of $\{x_0^\sigma,\hdots, x_k^\sigma\}$. Now, place a new $S^1$--invariant metric on $M^\sigma$ by pushing forward the Euclidean metric of $B_\epsilon(T^\perp\{x_0^\sigma,\hdots,x_k^\sigma\})$ through $\exp$ and extending it to the rest of $M^\sigma$. Call this metric $m$. Then, for tangent vectors $v_1,v_2\in T_xM^\sigma$, we define a new Riemannian metric given by
    \[\frac{1}{2\pi}\int m(e^{i\theta}\cdot v_1,e^{i\theta}\cdot v_2) d\theta\]
    where $S^1$ acts on $TM^\sigma$ by the differential of the action on $M^\sigma$. This will give a new $S^1$--invariant metric on $M^\sigma$ which agrees with the Euclidean metric near $x_0^\sigma,\hdots,x_k^\sigma$.
    \par
    Choose an $\epsilon'<\epsilon/2$ and consider $R_i^\sigma=C_i^\sigma \setminus \mathring{B}_{\epsilon'}(x_i^\sigma)$ for $i\in\{1,\hdots,k-1\}$. Then each $R_i^\sigma$ and $C_{i+1}^\sigma$ are disjoint compact subsets of $M^\sigma$. Similarly, if $L_i^\sigma=C_i^\sigma\setminus \mathring{B}_{\epsilon'}(x_{i-1}^\sigma)$ for $i\in\{1,\hdots,k-1\}$, then $L_i^\sigma$ and $C_{i-1}^\sigma$ are disjoint compact subsets of $M^\sigma$. For a submanifold (possibly with boundary) $A\subseteq M^\sigma$, let $N_{r}(A)=\{x\in M^\sigma:d(x,A)<r\}$ be the $r$--neighborhood of $A$. Choose some $\epsilon''<\epsilon'$ such that for each $i\in\{1,\hdots,k-1\}$, $N_{\epsilon''}(R_i^\sigma)\cap N_{\epsilon''}(C_{i+1}^\sigma)=N_{\epsilon''}(L_i^\sigma)\cap N_{\epsilon''}(C_{i-1}^\sigma)=\varnothing$. \par
    Let $N^\sigma_i=N_{\epsilon''}(C_i^\sigma)$ for $i\in\{1,\hdots,k\}$. Now we consider the intersection $X_i^\sigma=N_i^\sigma\cap N_{i+1}^\sigma$ for $i\in\{1,\hdots,k-1\}$. For $z\in X_i^\sigma$, $z\notin N_{\epsilon''}(R_{i}^\sigma)$ and $z\notin N_{\epsilon''}(L_{i+1}^\sigma)$ and thus there must be some $x\in \mathring{B}_{\epsilon'}(x_i^\sigma)\cap C_i^\sigma$ and $y\in \mathring{B}_{\epsilon'}(x_i^\sigma)\cap C_{i+1}^\sigma$ such that
    \begin{align*}
        d(z,x)&=d(z,C_i^\sigma)<\epsilon'' & d(z,y)=d(z,C_{i+1}^\sigma)<\epsilon''.
    \end{align*}
    Furthermore, the geodesics $[z,x]$ and $[z,y]$ are orthogonal to $C_i^\sigma$ and $C_{i+1}^\sigma$ respectively and contained in $B_\epsilon(x_i^\sigma)$. Therefore, the points $x_i^\sigma$, $x$, $z$, $y$ form a Euclidean rectangle. This means the opposing pairs of sides in the rectangle have the same length and so 
    \[d(x,x_i^\sigma)=d(z,y)<\epsilon'', \ \ \ d(y,x_i^\sigma)=d(x,z)<\epsilon''.\]
    Therefore, each $z\in X_i^\sigma$ is less $\epsilon''$ away from a point on $C_i^\sigma$ (resp. $C_{i+1}^\sigma$) which is less than $\epsilon''$ away from $x_i^\sigma$.\\
    Repeating the same process for $M^\tau$, we can assume there is some $\epsilon,\epsilon''$ with $\epsilon''<\epsilon/2$ and  $S^1$--invariant metrics on $M^\sigma$ and $M^\tau$ so that $N_{\epsilon}(\{x_1^\sigma,\hdots,x_k^\sigma\})$ and $N_{\epsilon}(\{x_1^\tau,\hdots, x_k^\tau\})$ are Euclidean and every $z\in X_i^\sigma$ (resp. $X_i^\tau=N_{\epsilon''}(C_i^\tau)\cap N_{\epsilon''}(C_{i+1}^\tau))$ is less than $\epsilon''$ away from a point in $C_i^\tau$ (resp. $C_{i+1}\sigma$, $C_i^\tau$, $C_{i+1}^\tau$) which is less than $\epsilon''$ away from $x_i^\sigma$ (resp. $x_i^\tau$). \par
\vspace{.1in}\noindent\textbf{Step 2 -- Building the equivariant diffeomorphism:}

\noindent Let $N_i^\tau=N_{\epsilon''}(C_i^\tau)$. The next step is to construct equivariant diffeomorphisms $$f_i:N_i^\sigma\to N_i^\tau$$ such that $f=f_i(x)$ for $x\in N_i^\sigma$ is well-defined. Then, $$f:N_1^\sigma\cup\cdots \cup N_k^\sigma\to N_1^\tau\cup\cdots \cup N_k^\tau$$ will be an $S^1$--equivariant diffeomorphism between open neighborhoods of $C^\sigma$ and $C^\tau$ and we are done. \par
Note that the nearest point projection $\pi_i^\sigma:\bar{N}_i^\sigma\to C_i^\sigma$ is a fiber bundle projection with the fiber above $x\in C_i^\sigma$ being all points $y$ with a geodesic $[x,y]$ in $M^\sigma$ such that $d(x,y)=\epsilon''$ and $[x,y]$ being orthogonal to $C_i^\sigma$. Then, restricting $\pi_i^\sigma$ to $\del N_i^\sigma$, we get a fiber bundle projection $$\del \pi_i^\sigma:\del N_i^\sigma\to C_i^\sigma$$

with which makes $\del N_i^\sigma$ a circle bundle over a sphere. \par
The next part of the proof is to construct a $S^1\times S^1$--action on $\del N_i^\sigma$ that is free except for points in the fibers above $x_{i-1}^\sigma$ and $x_i^\sigma$. This will allow us to use the technology of smooth principal bundles to build our diffeomorphisms.
\par
\vspace{.1in}\noindent\textbf{Step 3 -- Constructing a free $S^1\times S^1$--action:}

\noindent Endow $C_i^\sigma$ with the orientation that induces the rotation orientation on $T_{x_{i-1}^\sigma}C_i^\sigma$. We then orient each fiber of $\del N_i^\sigma$ so that combination of this orientation with the orientation on $C_i^\sigma$ gives the natural orientation on $\del N_i^\sigma$. Because each fiber is an oriented circle, we have an additional $S^1$--action (which denote by $\odot$) given so that $e^{i\theta}$ acts by rotating each circle by $\theta$ according to the induced orientation. This makes $\del N_i^\sigma$ a principal $S^1$--bundle over $C_i^\sigma$. These bundles can also be thought of as the image of subbundles through the exponential map $\exp:T^\perp C_i^\sigma\to M^\sigma$, i.e.
\begin{align*}
    \bar{N}_i^\sigma&=\exp\left(\{(x,v)\in T^\perp C_i^\sigma:||v||\leq \epsilon''\}\right)\\
    N_i^\sigma&=\exp\left(\{(x,v)\in T^\perp C_i^\sigma:||v||< \epsilon''\}\right)\\
    \del{N}_i^\sigma&=\exp\left(\{(x,v)\in T^\perp C_i^\sigma:||v||= \epsilon''\}\right)
\end{align*}
\par
Because our original $S^1$--action (which we will continue to denote by $\cdot$) acts by isometries, it commutes with this new $S^1$--action. Let $T=S^1\times S^1$ and define a $T$--action on $\del N_i^\sigma$ (denoted by $\wasylozenge$) that combines these actions. In particular,
\[(g,h)\wasylozenge x=g\cdot (h\odot x)\text{ for }g,h\in S^1,x\in \del N_i^\sigma\]

Let $\eta_i$ be a generator of $H_{w_i}\cong\mathbb{Z}/w_i\mathbb{Z}$. Then, with our $(\wasylozenge)$--action $(\eta_i,1)$ acts by applying our original $(\cdot)$--action by $\eta_i$. This fixes each point of $C_i^\sigma$ and thus acts by rotating each fiber by a rotation of order $w_i$. Because the set of such rotation angles is discrete, but the action by $\eta_i$ is continuous, there must be some constant rotation angle that is applied to each fiber. In other words, there is some $h_i\in S^1$ such that
\[\eta_i\cdot  x=h_i^{-1}\odot x\text{, so} \ \ (\eta_i,h_i)\wasylozenge\ x=x\text{  for all }x\in \del N_i^\sigma.\]
In fact, we note that
\begin{align*}
    (\del\pi_i^\sigma)^{-1}(x_{\ell}^\sigma)&=\exp(\{(x_{\ell}^\sigma,v)\in T_{x_{\ell}^\sigma}C_i^\sigma:||v||=\epsilon''\})\text{ for }\ell\in\{i-1,i\}\\
\end{align*}
Because we use the rotation orientation on $T_{x_{i-1}^\sigma}C_i^\sigma$, but use the natural orientation on $T_{x_{i-1}}M^\sigma$, we have that for the $(\cdot)$--action on $T_{x_{i-1}^\sigma}M^\sigma$ that
\[T_{x_{i-1}^\sigma}M^\sigma=T_{x_{i-1}}C_{i}^\sigma\oplus T^\perp_{x_{i-1}}C_i^\sigma\cong V^0_{w_{i}}\oplus V^0_{\epsilon_{i-1}w_{i-1}}\]
while for $T_{x_{i}^\sigma}M^\sigma$, the $(\cdot)$--action rotates $T_{x_i^\sigma}C_i^\sigma$ in the opposite orientation and thus
\[T_{x_{i}^\sigma}M^\sigma=T_{x_{i}}C_{i}^\sigma\oplus T^\perp_{x_{i}}C_i^\sigma\cong V^0_{-w_{i}}\oplus V^0_{-\epsilon_{i}w_{i+1}}.\]
Therefore, for $(x_{i}^\sigma,v)\in T_{x_{i}^\sigma}^\perp C_{i}^\sigma$ with $||v||<\epsilon''$ and $x=\exp(x_{i}^\sigma,v)$, then
\begin{align*}
    \eta_i\cdot x&=\eta_i\cdot \exp(x_{i}^\sigma,v)\\
    &=\exp(x_{i}^\sigma,\eta_i^{-\epsilon_{i}w_{i+1}}\cdot v)\\
    &=\eta_i^{-\epsilon_{i}w_{i+1}}\odot\exp(x_i^\sigma,v).
\end{align*}
We conclude that $h_i=\eta_i^{\epsilon_{i}w_{i+1}}$. By an analogous argument, we can conclude that $h_i=\eta_i^{-\epsilon_{i-1}w_{i-1}}$. 
\par
Using this observation, we can define a new $T$--action on $\del N_i^\sigma$ denoted by $(\diamond)$ so that
\[(g,h)\diamond x=(\tilde{g},\tilde{g}^{\epsilon_iw_{i+1}}h)\wasylozenge x.\]
where $\tilde{g}\in S^1$ such that $\tilde{g}^{w_i}=g$. The fact that $(h,h^{\epsilon_iw_{i+1}})\wasylozenge x=x$ for $h\in H_{w_i}$ implies that the $(\diamond)$--action is well-defined and independent of the choice of $\tilde{g}$.
Observe the $(\diamond)$--action is free on $\del N_i^\sigma\setminus (\del \pi_i^{\sigma})^{-1}(\{x_{i-1}^\sigma,x_i^\sigma\})$ as follows:
\par
Take an arbitrary $x\in \del N_i^\sigma\setminus (\del \pi_i^\sigma)^{-1}(\{x_{i-1}^\sigma,x_i^\sigma\})$ and $(g,h)\in T$ such that $(g,h)\ast x=x$. Then, \[\del\pi_i^\sigma((g,h)\diamond x)=\tilde{g}\cdot \del \pi_i^\sigma(x)=\del\pi_i^\sigma(x)\]
So $\tilde{g}$ stabilizes $\del\pi_i^\sigma(x)\in C_i^\sigma\setminus \{x_{i-1}^\sigma,x_i^\sigma\}$. Therefore, $\tilde{g}\in H_{w_i}$ and $g=1$. Then, $(g,h)\diamond x=(1,h)\diamond x=(1,h)\wasylozenge x=h\odot x=x$. But the $(\odot)$--action of $S^1$ on $\del N_i^\sigma$ is free so $h=1$. \par
We also have a bundle structure on $\del N_i^\tau$ with a free $(\diamond)$--action by $T$ on $\del N_i^\tau\setminus (\pi_i^{\tau})^{-1}(\{x_{i-1}^\tau,x_i^\tau\})$. 
\par
Now, we have our $(T,\diamond)$--action we can construct our diffeomorphisms $f_i:N_i^\sigma\to N_i^\tau$. We first define $f_i$ on the intersections $N_i^{\sigma}\cap N_{i+1}^\tau$ and then find ($T,\diamond$)-equivariant diffeomorphisms $\del f_i:\del N_i^\sigma\to \del N_i^\tau$. Because of the careful constructions of the metric and $(T,\diamond)$--action, applying a change of fiber to $\del f_i$ will produce our desired $(S^1,\cdot)$--equivariant diffeomorphisms $f_i: N_i^\sigma\to N_i^\tau$.
\par
Now, we choose $(S^1,\cdot)$--equivariant isometries $\psi_i^\sigma:B_{\epsilon}(V_{w_i}^0\oplus V_{\epsilon_{i}w_{i+1}})^0\to B_{\epsilon}(x_i^\sigma)$ and $\psi_i^\tau:B_{\epsilon}(V_{w_i}^0\oplus V_{\epsilon_{i}w_{i+1}}^0)\to B_{\epsilon}(x_i^\tau)$, then, each $\psi_i^{\tau}\circ(\psi^{\sigma}_i)^{-1}$ restricts to a diffeomorphism between $(\pi_i^{\sigma})^{-1}(C_i^\sigma\cap B_{\epsilon''}(x_{i}^\sigma))$ and $(\pi_i^{\tau})^{-1}(C_i^\tau\cap B_{\epsilon''}(x_{i}^\tau))$ and to a diffeomorphism between $(\pi_{i+1}^{\sigma})^{-1}(C_{i+1}^\sigma\cap B_{\epsilon''}(x_{i}^\sigma))$ and $(\pi_{i+1}^{\tau})^{-1}(C_{i+1}^\tau\cap B_{\epsilon''}(x_{i}^\tau))$. These maps, when restricts to $\del N_i^\sigma$, are $(T,\ast)$--equivariant because they are $(S^1,\cdot)$--equivariant isometries that map points in $C_i^\sigma$ to points in $C_i^\tau$ and thus preserve both of the $S^1$--actions used to define the $(\diamond)$--action. We consider these restrictions as $\del f_i|_{\del \pi_i^{\sigma\ -1}(C_i\cap B_{\epsilon''}(\{x_{i-1}^\sigma,x_{i}^\sigma\}))
}$. We note that the $(T,\diamond)$--orbit space of each $\del N_i^\sigma\setminus (\pi_i^{\sigma})^{-1}(\{x_{i-1}^\sigma,x_i^\sigma\})$ and $\del N_i^\tau\setminus (\pi_i^{\tau})^{-1}(\{x_{i-1}^\tau,x_i^\tau\})$ is diffeomorphic to $\mathbb{R}$. It follows that we can extend $\del f_i|_{(\del \pi_i^\sigma)^{-1}(C_i^\sigma\cap B_{\epsilon''}(x_i^\sigma))}$ to $(T,\diamond)$--equivariant diffeomorphisms $\del f_i:\del N_i^\sigma\to \del N_i^\tau$. It should be noted because the $(T,\diamond)$--action contains both our original $S^1$--action and the free $S^1$--action which makes $\del N_i^\sigma$ and $\del N_i^\tau$ principal $S^1$--bundles, we see that $\del f_i$ is, in fact, an $S^1$--equivariant isomorphism of bundles. Therefore, by a change of fiber, $\del f_i$ induces to an $S^1$--equivariant diffeomorphism $f_i:N_i^\sigma\to N_i^\tau$. Explicitly, this change of fiber goes as follows. Take an arbitrary $z\in N_i^\sigma$ and assume that $z$ lies on the geodesic $[x,y]$ where $x\in C_i^\sigma$ and $y\in(\del\pi_i^\sigma)^{-1}(x)$. Then, then because $\del f_i:\del N_i^\sigma\to \del N_i^\tau$ is an isomorphism of bundles, $(\pi_i^\sigma)^{-1}(x)$ is sent diffeomorphically, by $\del f_i$, onto a fiber $(\del \pi_i^\tau)^{-1}(f_i(x))$ for a unique $f_i(x)\in C_i^\tau$. Then, $\del f_i(y)\in (\del \pi_i^\tau)^{-1}(f_i(x))$ and there is a unique point $f_i(z)$ lying on the geodesic $[f_i(x),\del f_i(y)]$ such that $d(z,y)=d(f_i(z),\del f_i(y))$. Because $\del f_i$ is $(S^1,\cdot)$--equivariant and $(S^1,\cdot)$ acts by isometries, the configuration used to define $f_i(z)$ is preserved by this action. Thus, $f_i:N_i^\sigma\to N_i^\tau$ is $(S^1,\cdot)$--equivariant for each $i\in\{1,\hdots,k\}$.\\ 
All that remains is to verify that $f_i$ and $f_{i+1}$ agree on $X_i^\sigma=N_i^\sigma\cap N_{i+1}^\sigma$. Take an arbitrary $z\in X_i$. Then we can assume $z$ lies on a geodesic $[x,y]$ where $x\in C_i^\sigma\cap B_{\epsilon''}(x_i^\sigma)$ and $y\in (\del \pi_i^\sigma)^{-1}(x)$. Then, $f_i(z)$ is the unique point lying on the geodesic between $f_i(x)$ and $\del f_i(y)$ such that $d(f_i(z),\del f_i(y))=d(z,y)$. However, we know that
\[\del\pi_i(\del f_i(y))=\psi_i^{\tau}(\del\pi_i^\tau((\psi^{\sigma}_i)^{-1}(y))))=\psi^{\tau}_i((\psi^\sigma_i)^{-1}(x))\]
so that $f_i(x)=\psi_i^\tau((\psi_i^\sigma)^{-1}(x))$. Then, because $\psi^\tau\circ(\psi_i^\sigma)^{-1}:B_{\epsilon}(x_i^\sigma)\to B_{\epsilon}(x_i^\tau)$ is also an isometry, the point $z'=\psi^\tau_i((\psi^\sigma_i)^{-1}(z))$ is a point lying on the geodesic $[f_i(x),\del f_i(y)]$ so that $d(\psi_i^\tau((\psi_i^\sigma)^{-1}(z)),\del f_i(y))=d(z,y)$. Therefore, $f_i(z)=\psi_i^\tau((\psi_i^\sigma)^{-1}(z))$
Finally, we note that $f$ is orientation-preserving at each $x_i^\sigma$.
    \end{proof}

 \subsection{Actions on Hirzebruch Surfaces}
 Hirzebruch surfaces are algebraic $\mathbb{CP}^1$ bundles over $\mathbb{CP}^1$. They represent an important class of $4$--manifolds with circle actions. For instance, Jang \cite{Jang_Hirzebruch} showed that any almost complex $4$--manifold $M$ with a circle action with $4$ fixed points has the same tangential weight data as an algebraic action on a Hirzebruch surface. In light of Corollary \ref{SmoothClassification}, this means that if $M$ is simply connected, then $M$ is diffeomorphic to either $S^2\times S^2$ or $\mathbb{CP}{}^2\#\overline{\mathbb{CP}}{}^2$. We will give an explicit of description of Hirzebruch surfaces and their circle actions. 
 
 \begin{defn}
    Let $\mathbb{P}^{q+3}$ be the projectivization of the subspace of $  M_{2, q+2}(\mathbb{C}) \cong \mathbb{C}^{2(q+2)}$ consisting of matrices of the form
    \[\begin{bmatrix}
         x_1 & x_3 & x_4 & \cdots & x_{q+3}\\
         x_2 & x_4 & x_5 & \cdots & x_{q+4}
     \end{bmatrix}\text{ for }x_1,\hdots,x_{q+4}\in\mathbb{C}.\]
     We will denote the equivalence class of the above matrix in $\mathbb{P}^{q+3}$ by
     \[\begin{Bmatrix}
         x_1 & x_3 & x_4 & \cdots & x_{q+3}\\
         x_2 & x_4 & x_5 & \cdots & x_{q+4}
     \end{Bmatrix}.\]
     For $q\geq 0$, the Hirzebruch surface $\Hir(q)$ is the subvariety of $\mathbb{P}^{q+3}$ defined by rank inequality
     \[\operatorname{rank}\begin{bmatrix}
         x_1 & x_3 & x_4 & \cdots & x_{q+3}\\
         x_2 & x_4 & x_5 & \cdots & x_{q+4}
     \end{bmatrix}\leq 1.\]
     Note that $\Hir(q)$ has the defining system of homogeneous equations being the vanishing of the determinants
     \begin{align*}
         x_1x_{k+1}-x_2x_k&=0\text{ for }3\leq k\leq q+3 \\
         x_{j}x_{k+1}-x_{j+1}x_k&=0\text{ for }3\leq j<k\leq q+3
     \end{align*}
 \end{defn}
  Hirzebruch surfaces are examples of rational scrolls. For more discussion of rational scrolls and the above realization of $\Hir(q)$ as a projective variety, see Reid \cite{reid1996chaptersalgebraicsurfaces}.
  \par
  The essential feature of Hirzebruch surface is that they are algebraic $\mathbb{CP}^1$--bundles over $\mathbb{CP}^1$. This means their topology can be determined clutching functions for this bundle.
\begin{prop}\label{HirzeBundle}
    The map
    \[\begin{Bmatrix}
         x_1 & x_3 & x_4 & \cdots & x_{q+3}\\
         x_2 & x_4 & x_5 & \cdots & x_{q+4}
     \end{Bmatrix}\mapsto \operatorname{colsp}\begin{bmatrix}
         x_1 & x_3 & x_4 &\cdots & x_{q+3} \\
         x_2 & x_4 & x_5 & \cdots & x_{q+4}
     \end{bmatrix}\] 
    defines a fiber bundle projection $\pi:\Hir(q)\to \mathbb{CP}^1$ with fiber $\mathbb{CP}^1$. If $q$ is even, this bundle is trivial and $\Hir(q)$ is diffeomorphic to $S^2\times S^2$. If $q$ is odd, this bundle is non-trivial and $\Hir(q)$ is diffeomorphic to $\mathbb{CP}{}^2\#\overline{\mathbb{CP}}{}^2$.
\end{prop}
\begin{proof}
    First, we note that $\pi$ is well-defined because $(x_1,x_2,\hdots,x_{q+4})$ not being the zero vector means that the rank of 
    \[\begin{bmatrix}
         x_1 & x_3 & x_4 &\cdots & x_{q+3} \\
         x_2 & x_4 & x_5 & \cdots & x_{q+4}
     \end{bmatrix}\]
     is non-zero.
     Together with the rank inequality defining $\Hir(q)$ means that the column span
     \[\operatorname{colsp}\begin{bmatrix}
         x_1 & x_3 & x_4 &\cdots & x_{q+3} \\
         x_2 & x_4 & x_5 & \cdots & x_{q+4}
     \end{bmatrix}\]
     is a $1$--dimensional subspace of $\mathbb{C}^2$ and thus defines an element of $\mathbb{CP}^1$.
    Let $U_1=\mathbb{CP}^1\setminus\{[0:1]\}$ and $U_2=\mathbb{CP}^1\setminus\{[1:0]\}$. We will first find trivializations of $\pi$ on the open cover $\mathbb{CP}^1=U_1\cup U_2$.
    For $z\in\mathbb{C}$,
    \begin{align*}
        [\alpha:\beta]\mapsto \begin{Bmatrix}
            \alpha  & \beta  & \beta z & \cdots & \beta z^{q}\\
            \alpha z & \beta z & \beta z^2 & \cdots & \beta z^{q+1}
        \end{Bmatrix}
    \end{align*}
    is a diffeomorphism between $\mathbb{CP}^1$ and $\pi^{-1}([1:z])$. Therefore, $\varphi_1:\pi^{-1}\left(U_1\right)\to U_1\times \mathbb{CP}^1$ defined by 
    \[\varphi_1\left(\begin{Bmatrix}
         x_1 & x_3 & x_4 & \cdots & x_{q+3}\\
         x_2 & x_4 & x_5 & \cdots & x_{q+4}
     \end{Bmatrix}\right)\mapsto \left(\left[x_1:x_2\right],[x_1:x_3]\right)\]
    is a trivialization of $\pi$ on $U_1$.
    \\
    On the other hand, for $z\in\mathbb{C}$
    \begin{align*}
        [\alpha:\beta]\mapsto \begin{Bmatrix}
            \alpha z  & \beta z^{q+1}  & \beta z^{q} & \cdots & \beta z\\
            \alpha  & \beta z^{q} & \beta z^{q-1} & \cdots & \beta
        \end{Bmatrix}
    \end{align*}
    is a diffeomorphism between $\mathbb{CP}^1$ and $\pi^{-1}([z:1])$. Thus, $\varphi_2:\pi^{-1}\left(U_2\right)\to U_2\times\mathbb{CP}^1$ defined by 
    \[\varphi_2\left(\begin{Bmatrix}
         x_1 & x_3 & x_4 & \cdots & x_{q+3}\\
         x_2 & x_4 & x_5 & \cdots & x_{q+4}
     \end{Bmatrix}\right)\mapsto \left(\left[x_1:x_2\right],[x_2:x_{q+4}]\right)\]
    is a trivialization of $\pi$ on $U_2$. Therefore, $\pi:\Hir(q)\to \mathbb{CP}^1$ is a smooth fiber bundle with fiber $\mathbb{CP}^1$. 
    
    We now calculate the transition function $\psi_z:\mathbb{CP}^1\to\mathbb{CP}^1$ above $[1:z]\in U_1\cap U_2$ for $z\in\mathbb{C}^\times$. That is
    \begin{align*}
        ([1:z],\psi_z([\alpha:\beta]))&=\varphi_2(\varphi_1^{-1}([1:z],[\alpha:\beta]))\\
        &=\varphi_2\left(\begin{Bmatrix}
            \alpha  & \beta  & \beta z & \cdots & \beta z^{q}\\
            \alpha z & \beta z & \beta z^2 & \cdots & \beta z^{q+1}
        \end{Bmatrix}\right)\\
        &=([\alpha:\alpha z],[\alpha z:\beta z^{q+1}])\\
        &=([1:z],[\alpha:\beta z^q])
    \end{align*}
    and thus $\psi_z:\mathbb{CP}^1\to\mathbb{CP}^1$ is given by $\psi_z([\alpha:\beta])=[\alpha:\beta z^q]$. Therefore, the bundle has structure group $\PSL(2,\mathbb{C})$ and the isomorphism class of this bundle is determined by the homotopy class of the map $U_1\cap U_2\to\PSL(2,\mathbb{C})$ given by $z\mapsto\psi_z$. Because $\PSL(2,\mathbb{C})$ is connected, we can identify this homotopy class with the element of $\pi_1(\PSL(2,\mathbb{C))}\cong\pi_1(\operatorname{PSU}(2))\cong \mathbb{Z}_2$ represented by the loop $\gamma:[0,1]\to \operatorname{PSU}(2)$ given by $\gamma(\theta)=\psi_{e^{2\pi i\theta}}$.  Then, $\gamma$ lifts to the path $\tilde{\gamma}:[0,1]\to \operatorname{SU}(2)$ in the universal covering space of $\operatorname{PSU}(2)$ given by
    \[\tilde{\gamma}(\theta)=\begin{bmatrix}
        e^{-q\pi i\theta} & \\ & e^{q\pi i \theta}
    \end{bmatrix}\]
    which forms a loop in $\operatorname{SU}(2)$ if and only if $q$ is even. This means there are two homotopy classes and thus two isomorphism classes of bundles. If $q$ is even, $\Hir(q)$ is diffeomorphic to the trivial bundle $\Hir(0)=\mathbb{CP}^1\times\mathbb{CP}^1$ if $q$ is even. If $q$ is odd, then $\Hir(q)$ diffeomorphic to the non-trivial bundle $\Hir(1)\cong\mathbb{CP}{}^2\#\overline{\mathbb{CP}}{}^2$
\end{proof}
Now, we will define a class of $S^1$--actions on each Hirzebruch surface $\Hir(q)$. These will serve as building blocks we need for the $S^1$--actions on the fibers for Anosov representation.
\begin{defn}
    Let $a,b$ be integers.
    Let $\Hir(q;a,b)$ be the Hirzebruch surface $\Hir(q)$ with the $S^1$--action given by
    \[g\cdot\begin{Bmatrix}
         x_1 & x_3 & x_4 & \cdots & x_{q+3}\\
         x_2 & x_4 & x_5 & \cdots & x_{q+4}
     \end{Bmatrix}=\begin{Bmatrix}
         x_1 & g^{a}x_3 & g^{a+b}x_4 & \cdots & g^{a+qb}x_{q+3}\\
         g^{b}x_2 & g^{a+b}x_4 & g^{a+2b}x_5 & \cdots & g^{a+(q+1)b}x_{q+4}
     \end{Bmatrix}.\]
\end{defn}
Since we have this $S^1$--action defined on the simply-connected $4$--manifold $\Hir(q)$, let us now find the tangential weight graph of $\Hir(q;a,b)$. We observe the following:
\begin{prop}\label{HirzeWeights}
    Let $a,b\in\mathbb{Z}$.
    \begin{enumerate}[label=\normalfont{(\roman*)}]
        \item The points
        \begin{align*}
        p_1&=\begin{Bmatrix}
         1 & 0 & 0 & \cdots & 0\\
         0 & 0 & 0 & \cdots & 0
     \end{Bmatrix} & p_2&=\begin{Bmatrix}
         0 & 1 & 0 & \cdots & 0\\
         0 & 0 & 0 &\cdots & 0
     \end{Bmatrix} \\
     p_3&=\begin{Bmatrix}
         0 & 0 & 0 & \cdots & 0\\
         1 & 0 & 0 & \cdots & 0
     \end{Bmatrix} & p_4&=\begin{Bmatrix}
         0 & 0 & \cdots & 0 & 0\\
         0 & 0 &\cdots & 0& 1
     \end{Bmatrix}
    \end{align*}
    are fixed points in $\Hir(q;a,b)$. 
    \item The spheres
    \begin{align*}
        S_{12}&=\left\{\begin{Bmatrix} \alpha & \beta & 0 & \cdots & 0 \\ 0 & 0 & 0 & \cdots & 0
            \end{Bmatrix}:[\alpha:\beta]\in\mathbb{CP}^1\right\}\\
        S_{13}&=\left\{\begin{Bmatrix} \alpha & 0 & 0 & \cdots & 0 \\ \beta & 0 & 0 & \cdots & 0
            \end{Bmatrix}:[\alpha:\beta]\in\mathbb{CP}^1\right\}\\
        S_{24}&=\left\{\begin{Bmatrix} 0 & \alpha^{q+1} & \alpha^{q}\beta & \cdots & \alpha\beta^q \\ 0 & \alpha^{q}\beta & \alpha^{q-1}\beta^2 & \cdots & \beta^{q+1}
            \end{Bmatrix}:[\alpha:\beta]\in\mathbb{CP}^1\right\}\\
        S_{34}&=\left\{\begin{Bmatrix} 0 & 0 & \cdots & 0 & 0\\ \alpha & 0 & \cdots & 0 & \beta
            \end{Bmatrix}:[\alpha:\beta]\in\mathbb{CP}^1\right\}
    \end{align*}
    are $S^1$--invariant submanifolds of $\Hir(q;a,b)$.
    \item As $S^1$--representations,
    \begin{align*}
        T_{p_1}S_{12}&\cong V^0_{a} & T_{p_1}S_{13}&\cong V^0_b\\
        T_{p_2}S_{12}&\cong V^0_{-a} & T_{p_2}S_{24}&\cong V_b^{0}\\
        T_{p_3}S_{13}&\cong V_{-b}^{0} & T_{p_3}S_{34}&\cong V^0_{a+qb}\\
        T_{p_4}S_{24}&\cong V^0_{-b} & T_{p_4}S_{34}&\cong V^0_{-a-qb} 
    \end{align*}
    \end{enumerate}
\end{prop}
\begin{proof}
    Parts (i,ii) follow from direct inspection of the defined action. By restricting the $S^1$--action to each of the spheres in (ii), you can find the weights in (iii).
\end{proof}
This gives us all the information to determine the signs and weights for the tangential weight graph.
\begin{corr}\label{ExceptionalHirTWG}
    For $a,b$ non-zero coprime integers with $a+qb\neq 0$, the tangential weight graph for $\Hir(q;a,b)$ is the one given in Figure \ref{fig:HirzabTWG}
\begin{figure}
    \centering
    \begin{tikzpicture}[node distance=4cm and 7cm, main/.style = {draw, circle}]
\node[main] (1) at (0,0) {$s_1$}; 
\node[main] (2) at (1.85,0) {$s_2$}; 
\node[main] (3) at (0,-1.5) {$s_3$}; 
\node[main] (4) at (1.85,-1.5) {$s_4$}; 
\draw[] (1) -- node[midway, above,] {$|a|$} (2);
\draw[-] (1) -- node[midway, left,] {$|b|$} (3);
\draw[-] (2) -- node[midway, right,] {$|b|$} (4);
\draw[-] (3) -- node[midway, above,] {$|a+qb|$} (4);
\end{tikzpicture}
    \caption{Tangential weight graph for $\Hir(q;a,b)$}
    \label{fig:HirzabTWG}
\end{figure}
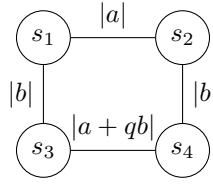
with $s_1=\sign(ab)$, $s_2=-s_1$, $s_3=-s_4$, $s_4=\sign(ab+qb^2)$.
\end{corr}
Compare \cite[Example 2.10]{Jang_Hirzebruch}.\par
\begin{prop}\label{FixedHirTWG}
    For $q\geq 0$, the tangential weight graph of $\Hir(q;1,0)$ is the one given in Figure \ref{fig:HirzqTWG}.
        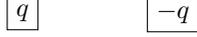
\begin{figure}
            \centering
            \begin{tikzpicture}[node distance={30mm}, main/.style = {draw, circle}]
\node[draw] (1) at (0,-3.6) {$q$}; 
\node[draw] (4) at (2,-3.6)  {$-q$};

\end{tikzpicture}
            \caption{Tangential weight graph for $\Hir(q;1,0)$}
            \label{fig:HirzqTWG}
        \end{figure}
\end{prop}
\begin{proof}
    Inspecting the $S^1$--action on $\Hir(q;1,0)$, $S_{13}$ and $S_{24}$ are fixed surfaces.
    By Proposition \ref{fixedappear}(iv), these fixed surfaces contain the only fixed or exceptional points in $\Hir(q;1,0)$. By \cite[Lemma 5.1]{FintushelSimply4Circle}, the tangential weight graph has two square vertices with labels $e$ and $-e$ for some $e\geq 0$. This means we only need to determine the Euler class of $U^\perp S_{24}$ up to sign. In other words, we can ignore which orientation are used to calculate the Euler class. \par
    Note that the $S^1$--action on $\Hir(q)$ extends to a $\mathbb{C}^\times$--action. Namely,
    \[g\cdot\begin{Bmatrix}
         x_1 & x_3 & x_4 & \cdots & x_{q+3}\\
         x_2 & x_4 & x_5 & \cdots & x_{q+4}
     \end{Bmatrix}=\begin{Bmatrix}
         x_1 & gx_3 & gx_4 & \cdots & gx_{q+3}\\
         x_2 & gx_4 & gx_5 & \cdots & gx_{q+4}
     \end{Bmatrix}\text{ for }g\in\mathbb{C}^\times.\]
    The non-principal orbits of this action are exactly the fixed points in $S_{13}$ and $S_{24}$. Let $E=\Hir(q)\setminus (S_{13}\cup S_{24})$. Because the projection $\Hir(q)\to \mathbb{CP}^1$ is $\mathbb{C}^\times$--invariant under this action, this means $\pi|_{E}:E\to \mathbb{CP}^1$ is a principal $\mathbb{C}^\times$--bundle. Then, the Euler class of $E$ will agree with the Euler class of $U^\perp S_{24}$.\par
    Triangulate $\mathbb{CP}^1$ so that there is a (closed) simplex $\sigma$ such that
    \[[0:1]\in \sigma\subseteq \mathbb{CP}^1\setminus \{[1:0]\}\]
    and $\del \sigma=\{[\alpha:\beta]:|\alpha|=|\beta|\}$
    Let $L$ be the complement of the interior of $\sigma$.
    For $[\alpha:\beta]\in L$,
    \[[\alpha:\beta]\mapsto  \begin{Bmatrix}
         \alpha^{q+1} & \alpha^{q+1} & \alpha^{q}\beta & \cdots & \alpha\beta^q\\
         \alpha^q\beta & \alpha^q\beta & \alpha^{q-1}\beta^2 & \cdots & \beta^{q+1}
     \end{Bmatrix}\]
     defines a section $s:L\to E|_{L}$. Similarly, for $[\alpha:\beta]\in \sigma$,
     \[[\alpha:\beta]\mapsto  \begin{Bmatrix}
         \alpha\beta^q & \alpha^{q+1} & \alpha^{q}\beta & \cdots & \alpha\beta^q\\
         \beta^{q+1} & \alpha^q\beta & \alpha^{q-1}\beta^2 & \cdots & \beta^{q+1}
     \end{Bmatrix}\]
     defines a section $t:\sigma\to E|_{\sigma}$
     Then, the transition function $\phi:\del \sigma\to\mathbb{C}^\times$ is
     given by
     \[[\alpha:\beta]\mapsto\frac{\beta^q}{\alpha^q}.\]
     Thus, the $2$-cochain defining the Euler class assigns $-q=\deg(\phi)$ to the simplex $\sigma$. Any other  $2$-simplex $\sigma'$ in the triangulation is assigned $0$ because it would be contained in $L$. This means that $-q$ is the Euler number for the bundle.
\end{proof}

\section{Tangential weight data for flag varieties}\label{sec:TWD4flags}
Now, we will explain a general method for calculating the tangential weight data of a flag variety of $\SL(n,\mathbb{C})$ as defined in Section \ref{classgrpflg}. In Section \ref{sec:calculations4fibers}, we apply this specifically to flag varieties listed in Proposition \ref{3dflagcases}.
\subsection{Invariant flags and their tangent spaces}

Let $V$ be an $n$--dimensional $\mathbb{C}$--vector space with a $\mathbb{C}$--linear $S^1$--action.
By \cite[Prop.\ 8.1]{Bröcker1985}, there are distinct integers $w_1,\hdots w_k\in\mathbb{Z}$ (called weights) with multiplicities $\delta_1,\hdots, \delta_k\geq 1$ so that $$V\cong (V_{w_1}^0)^{\delta_1}\oplus\cdots\oplus (V_{w_k}^0)^{\delta_k}$$. This is known as the weight decomposition of $V$ and vector in $V$ whose $\mathbb{C}$--span is $S^1$--invariant is called a weight vector. Note in the case $w=0$, then $V_0^0$ should denote the \textit{complex} trivial representation $(V=\mathbb{C},\rho:S^1\to\mathbb{C}^\times)$ with $\rho(g)=1$. 
\begin{defn}
    Let $F^\bullet=(F^{d_1},\hdots,F^{d_\ell})\in\Flag_{d_1,\hdots,d_\ell}(V)$ be a flag in $V$ such that each $F^{d_j}$ is $S^1$--invariant. Then, a weight flag basis is a sequence of vectors $v_1,\hdots,v_n\in V$ that has a sequence of weights $\omega_1,\hdots,\omega_n\in\{ w_1,\hdots, w_k\}$ such that $g\cdot v_i=g^{\omega_i}\cdot v_i$ for $i\in \{1,\hdots,n\}$ and $F^{d_j}=\langle v_1,\hdots, v_{d_j}\rangle$ for each $j\in\{1,\hdots,\ell\}$.
\end{defn}
Any $S^1$--invariant flag admits a weight flag basis, by inductively finding a weight decomposition for each subspace in the flag. \par

Let $\mathcal{F}=\Flag_{d_1,\hdots,d_{\ell}}(\mathbb{C}^n)$ be the flag variety of a fixed signature $(d_1,\hdots,d_\ell)$. Let $P_{d_1,\hdots,d_\ell}$ be the standard parabolic subgroup of $\GL(n,\mathbb{C})$ which stabilizes the standard flag $E^\bullet_{\vartheta}$ of signature $(d_1,\hdots,d_\ell)$ and let $\mathfrak{p}\subseteq\mathbb{C}^{n\times n}$ be the Lie algebra of $P$, which has block triangular form with block sizes $d_1,\hdots,d_\ell$. This means $\mathfrak{p}$ is spanned by $\left\{E_{ij}\right\}_{(i,j)\notin\mathcal{I}}$, where $E_{ij}$ is the $n\times n$ matrix whose only non-zero entry is $1$ in the $i$--th row and $j$--th column and
\[\mathcal{I}= \mathcal{I}_{d_1,\ldots,d_\ell}=\mathbb{N}^2\cap\left(\bigcup_{r=1}^{\ell}(d_r,n]\times [1,d_r]\right).\]
Note that $P=\GL(n,\mathbb{C})\cap\mathfrak{p}$. Let $\mathfrak{u}$ be the span of $\left\{E_{ij}\right\}_{(i,j)\in\mathcal{I}}$ in $\mathbb{C}^{n\times n}$. \par
Consider the flag variety $\mathcal{F}=\Flag_{d_1,\hdots,d_\ell}(V)$ and
    let $F^\bullet\in \mathcal{F}$ be an $S^1$--invariant flag with weight flag basis $v_1,\hdots,v_n$ and corresponding weights $\omega_1,\hdots, \omega_n$.
\begin{defn}\label{affinechart}
    The affine chart of $\mathcal{F}$ relative to the flag basis $v_1,\hdots,v_n$ is the map $\varphi_{v_1,\hdots,v_n}:\mathfrak{u}\to \mathcal{F}$ so that $u\mapsto (I_n+u)\cdot F^\bullet$ where $I_n+u\in\GL(n,\mathbb{C})$ acts on $\mathcal{F}=\Flag_{d_1,\hdots,d_\ell}(V)$ through the isomorphism $\GL(V)\cong \GL(n,\mathbb{C})$ induced by the basis $v_1,\hdots,v_n$. \par
    Explicitly, an element 
    \[u=\sum_{(i,j)\in\mathcal{I}}u_{ij}E_{ij}\] is mapped to the flag $(U^{d_1},\hdots, U^{d_\ell})$ where $U^{d_r}$ is spanned by the vectors
    \[v_j+\sum_{d_r<i\leq n}u_{ij}v_i\text{ for }1\leq j\leq d_r.\]
\end{defn}

An affine chart will give us a way to model the tangent space of a flag variety.

\begin{prop}\label{FlagTangentModel}
    The following hold:
    \begin{enumerate}
        \item The map $\varphi_{v_1,\dots,v_n}:\mathfrak{u}\to\mathcal{F}$ is $S^1$--equivariant if $\mathfrak{u}$ is given the linear $S^1$--action $g\cdot E_{ij}=g^{\omega_i-\omega_j}E_{ij}$ for $(i,j)\in \mathcal{I}$.
        \item The tangent space $T_{F^\bullet}\mathcal{F}$ is isomorphic to $\bigoplus_{(i,j)\in\mathcal{I}} V_{\omega_i-\omega_j}^0$ as an $S^1$--representation.
    \end{enumerate}
\end{prop}
\begin{proof}
    Take an arbitrary $u=\sum_{(i,j)\in\mathcal{I}}u_{i,j}E_{ij}$ in $\mathfrak{u}$. Then, by definition, $\varphi_{v_1,\hdots,v_n}(u)=U^\bullet=(U^{d_1},\hdots,U^{d_{\ell}})$ where $U^{d_r}$ is the subspace of $V$ spanned by 
    \[v_j+\sum_{d_r<i\leq n}u_{ij}v_i\text{ for $1\leq j\leq d_r$}.\]
 Then, $g\cdot U^\bullet$ is spanned by the vectors.
\[v_j'=g\cdot v_j+g\cdot \left(\sum_{d_r<i\leq n}u_{i,j}v_i\right)=g^{\omega_j}v_j+\sum_{d_r<i\leq n}g^{\omega_{i}}u_{i,j}v_i\text{ for $1\leq j\leq d_r$}.\]
In other words, $v_1',\hdots,v_{n}'$ is a flag basis for $g\cdot U^\bullet$. Then, $g^{-\omega_1}v_1',\hdots,g^{-\omega_n}v'_n$ is also a flag basis for $g\cdot U^\bullet$. Then, $g\cdot U^{d_r}$ is spanned by the vectors
\[g^{-\omega_j}v_j'=v_j+\sum_{d_r<i\leq n}g^{\omega_{i}-\omega_j}u_{i,j}v_i\text{ for $1\leq j\leq d_r$},\]
which means 
\[g\cdot U^\bullet=\varphi_{v_1,\hdots,v_n}\left(\sum_{(i,j)\in\mathcal{I}}g^{\omega_i-\omega_j}u_{ij}E_{ij}\right).\]
So if $\mathfrak{u}$ is given the linear $S^1$--action with $g\cdot E_{ij}=g^{\omega_i-\omega_j}E_{ij}$, then $\varphi_{v_1,\hdots,v_n}:\mathfrak{u}\to\mathcal{F}$ is $S^1$--equivariant. 

Then, the differential of $\varphi_{v_1,\hdots, v_n}$ at $0$ is a $S^1$--equivariant linear isomorphism between $\mathfrak{u}$ and $T_{F^\bullet}\mathcal{F}$. The described action on $\mathfrak{u}$ means $E_{ij}$ is a weight vector of weight $\omega_i-\omega_j$ and thus there is also an $S^1$--equivariant linear isomorphism between $\mathfrak{u}$ and $\bigoplus_{(i,j)\in\mathcal{I}}V_{\omega_i-\omega_j}^0$.
\end{proof}
Now, we will use this affine chart to give a way to explicitly identify exceptional spheres inside flag varieties.
\begin{prop}\label{ExceptInFlags}
    For $(i,j)\in\mathcal{I}$, let $q_1,\hdots,q_n$ be the sequence of basis vectors $v_1,\hdots,v_n$ with $v_i$ and $v_j$ swapped, i.e.
    \[(q_1,\hdots,q_n)= (v_1,\hdots,v_{j-1},v_i,v_{j+1},\hdots,v_{i-1},v_j,v_{i+1},\hdots, v_n).\]
    Consider the flag $Q^{\bullet}$ with flag basis $q_1,\hdots,q_n$, and the corresponding affine chart of $\mathcal{F}$ relative to the flag basis $\varphi_{q_1,\hdots,q_n}$. Then, the following hold:
    \begin{enumerate}[label=\normalfont{(\roman*)}]
        \item $\varphi_{v_1,\hdots,v_n}(\lambda E_{ij})=\varphi_{q_1,\hdots,q_n}(\lambda^{-1}E_{ij})$ for $\lambda\in\mathbb{C}^\times$.
        \item Endow $\mathbb{CP}^1$ with the $S^1$--action given by
        \[g\cdot[z_0:z_1]=[g^{\omega_i-\omega_j}z_0:z_1]\text{ for }g\in S^1.\]
        Then, the map $\psi:\mathbb{CP}^1\to \mathcal{F}$ given by
        \[\psi([\alpha:\beta])=\begin{cases}
            \varphi_{v_1,\hdots,v_n}(\lambda E_{ij}) & [\alpha:\beta]=[\lambda:1]\\
            Q^{\bullet} & [\alpha:\beta]=[1:0]
        \end{cases}\]
        is an $S^1$--equivariant embedding.
    \end{enumerate}
\end{prop}
\begin{proof}
    First, we prove (i).
    For $\lambda\in\mathbb{C}$, let $F^\bullet_\lambda=\varphi_{v_1,\hdots,v_n}(\lambda E_{ij})$. In particular, we have $F_0^\bullet=F^\bullet$. For $\lambda\neq 0$, the flag $F^\bullet_\lambda$ is, by definition, the flag $(F^{d_1},\hdots,F^{d_l})$ where 
    \[F^{d_r}=\begin{cases}
        \langle v_1,\hdots,v_{d_r}\rangle & d_r<j\\
        \langle v_1,\hdots, v_{j-1},v_{j}+\lambda v_i,v_{j+1},\hdots,v_{d_r}\rangle & d_r\geq j
    \end{cases}\]
    \[F^{d_r}=\begin{cases}
        \langle v_1,\hdots,v_{d_r}\rangle & d_r<j\\
        \langle v_1,\hdots, v_{j-1},v_{i}+\lambda^{-1} v_j,v_{j+1},\hdots,v_{d_r}\rangle & d_r\geq j
    \end{cases}\]
    On the other hand, let denote $Q_{\lambda^{-1}}^\bullet=\varphi_{q_1,\hdots,q_n}(\lambda^{-1}E_{ij})$ for $\lambda\in\mathbb{C}^\times$. Then,
    \[Q_{\lambda^{-1}}^{d_r}=\begin{cases}
        \langle q_1,\hdots,q_{d_r}\rangle & d_r<j\\
        \langle q_1,\hdots, q_{j-1},q_{j}+\lambda^{-1} q_i,q_{j+1},\hdots,q_{d_r}\rangle & d_r\geq j
    \end{cases}\]
    If we rewrite this in terms of $v_1,\hdots,v_n$, we see that 
    \[Q_{\lambda^{-1}}^{d_r}=\begin{cases}
        \langle v_1,\hdots,v_{d_r}\rangle & d_r<j\\
        \langle v_1,\hdots, v_{j-1},v_{i}+\lambda^{-1} v_j,v_{j+1},\hdots,v_{i-1},v_j,v_{i+1}\hdots,v_{d_r}\rangle & d_r\geq j
    \end{cases}\]
    We note that
    \[\langle v_{i}+\lambda^{-1}v_j,v_j\rangle=\langle v_i,v_j\rangle=\langle v_i+\lambda^{-1}v_j,v_{i}\rangle\]
    and thus for $j\geq d_r$, $Q^{d_r}_{\lambda^{-1}}$ can also be thought of as the subspace spanned by
    \[v_1,\hdots,v_{j-1},v_{i}+\lambda^{-1}v_j,v_{j+1},\hdots,v_{d_r}.\]
    In other words, $F^{d_r}_{\lambda}$ and $Q^{d_r}_{\lambda^{-1}}$ are spanned by the same vectors. We conclude that
    \[\varphi_{v_1,\hdots,v_n}(\lambda E_{ij})=F^\bullet_{\lambda}=Q^{\bullet}_{\lambda^{-1}}=\varphi_{q_1,\hdots,q_n}(\lambda^{-1}E_{ij}).\]
    This completes the proof of (i).
    
    To show (ii), it is sufficient to show that $\psi:\mathbb{CP}^1\to\mathcal{F}$ is an $S^1$--equivariant, injective immersion. Then, compactness of $\mathbb{CP}^1$ will imply that $\psi$ is an embedding. We will show this by restricting to the two affine charts $\chi_0,\chi_{\infty}:\mathbb{C}\to\mathbb{CP}^1$ given by $\chi_0(\lambda)=[\lambda:1]$ and $\chi_{\infty}(\lambda)=[1:\lambda]$ for $\lambda\in\mathbb{C}$.
    
    First, consider the chart $\chi_0$. Then,
    \begin{align*}
        \psi(\chi_0(\lambda))&=\psi([\lambda:1])=\varphi_{v_1,\hdots,v_n}(\lambda E_{ij}).
    \end{align*}
    Note that $\varphi_{v_1,\hdots,v_n}$ is a parameterization of an open Schubert cell and hence, in particular, an embedding. This means that $\psi\circ\chi_0$ is an injective immersion. Furthermore, for $g\in S^1$ and $\lambda\in \mathbb{C}$, we can use Proposition \ref{FlagTangentModel} to conclude that
    \begin{align*}
        g\cdot \psi(\chi_0(\lambda))&=g\cdot \varphi_{v_1,\hdots,v_n}(\lambda E_{ij})\\
        &=\varphi_{v_1,\hdots,v_n}(g^{\omega_i-\omega_j}\lambda E_{ij})\\
        &=\psi(\chi_0(g^{\omega_i-\omega_j}\lambda))\\
        &=\psi([g^{\omega_i-\omega_j}\lambda:1]) \\
        &=\psi(g\cdot [\lambda:1])\\
        &=\psi(g\cdot\chi_0(\lambda))
    \end{align*}
    and thus the map $\psi$ is $S^1$--equivariant on the chart $\chi_0$. 

    Note that part (i) implies that
    \[\psi(\chi_\infty(\lambda))=\psi([1:\lambda])=\varphi_{q_1,\hdots,q_n}(\lambda E_{ij}).\]
    Thus, we can use the same reasoning as before to conclude that $\psi$ restricted to the chart $\psi_\infty$ is a $S^1$--equivariant, injective, immersion. 

    Being $S^1$--equivariant or an immersion is a local property so this implies that $\psi$ is $S^1$--equivariant immersion. Injectivity is not a local condition but because each chart excludes a single point ($[1:0]$ for $\chi_0$ and $[0:1]$ for $\chi_\infty$), we only need to note that
    \[\psi([0:1])=F^\bullet\neq Q^\bullet=\psi([1:0])\]
    to conclude that $\psi$ is, in fact, injective. This completes the proof of (ii).
\end{proof}
The information that we glean from Proposition \ref{FlagTangentModel} and Proposition \ref{ExceptInFlags} can be represented visually in a matrix. 
\begin{defn}\label{DiffMatrix}
    Let $V$ be an $n$--dimensional complex vector space with a linear $\SOrth(2)$--action and $F^\bullet\in\Flag_{d_1,\hdots,d_\ell}$ be an $\SOrth(2)$--invariant flag. We define a $\SOrth(2)$--difference matrix for $F^\bullet$ as follows:\\
    Choose a weight flag basis $v_1,\hdots,v_n$ for $F^\bullet$. The columns and rows of the difference matrix are labeled by $v_1,\hdots,v_n$. If $(i,j)\notin\mathcal{I}$, then the entry of the $i$--th row and $j$--th column is $\ast$. If $(i,j)\in\mathcal{I}$, then the entry of the $i$--th row and $j$--th column is $\omega_i-\omega_j$. \par
    Now assume that $-I_2\in\SOrth(2)$ acts trivially on $\mathbb{P}(V)$. Then, a $\PSOrth(2)$--difference matrix for $F^\bullet$ is as above except the entry in $i$--th row and $j$--column is labeled by $\frac{\omega_i-\omega_j}{2}$ if $(i,j)\in\mathcal{I}$.
\end{defn}
How to interpret a difference matrix with respect to the above propositions will be demonstrated in the next example and utilized throughout Section \ref{sec:calculations4fibers}.
\begin{xmpl}
    Endow $V=\mathbb{C}^3$ with the linear $\SOrth(2)$--action so that
    \[R_\theta\cdot v=\begin{bmatrix}
        e^{2\theta} & & \\ & 1 & \\ & & e^{-2\theta}
    \end{bmatrix}\]
    Then, if $v_{1},v_2,v_{3}$ denotes the standard ordered basis for $V=\mathbb{C}^3$, then they form a weight basis with weights $\omega_1=2,\omega_2=0,\omega_3=-2$. Let $F^\bullet$ be the flag of signature $(1,2)$ with flag basis $v_{2},v_{0},v_{-2}$ so that $F^\bullet$ is an $\SOrth(2)$--invariant flag, i.e. \\
    \[F^1=\langle f_2\rangle,\ \ F^2=\langle f_2,f_0\rangle\]
    Then, the $\SOrth(2)$--difference matrix for $F^\bullet$ is
    \begin{center}
        \begin{tabular}{c|ccc}
             $F^\bullet$ & $v_{2}$ & $v_{0}$ & $v_{-2}$  \\ \hline
             $v_{2}$ & $\ast$ & $\ast$  & $\ast$\\
             $v_{0}$ & $-2$ & $\ast$  & $\ast$\\
             $v_{-2}$ & $-4$ & $-2$  & $\ast$
        \end{tabular}
    \end{center}
    This means that through the identification $S^1\cong \SOrth(2)$, $T_{F^\bullet}\Flag_{1,2}(V)$ is isomorphic to $V^0_{-2}\oplus V_{-2}^0\oplus V_{-4}^0$. The entry in row $v_{0}$ and column $v_2$ corresponds to the tangent direction of flags, for $\lambda\in\mathbb{C}$, of the form
    \[F^\bullet_{\lambda}=\varphi_{v_2,v_0,v_{-2}}\left(\begin{bmatrix}
       0 & 0 & 0\\ \lambda & 0 & 0\\ 0 & 0 & 0
    \end{bmatrix}\right)=\begin{bmatrix}
        1 & 0 & 0 \\ \lambda & 1 & 0\\ 0 & 0 & 1
    \end{bmatrix}\cdot F^\bullet=\left(\langle v_2+\lambda v_0\rangle,\langle v_2,v_0\rangle\right).\]
    The $-2$ entry records that the tangent space of $\{F^\bullet_\lambda:\lambda\in\mathbb{C}\}$ at $F^\bullet$ is isomorphic as an $\SOrth(2)$--module to $V_{-2}^0$. As $\lambda\to\infty$, 
    \[F_\lambda^\bullet\to Q^\bullet=(\langle v_0\rangle,\langle v_0,v_2\rangle),\]
    which is exactly the flag with flag basis $v_0,v_2,v_{-2}$, i.e.\ where $v_2$ and $v_0$ have been swapped.\\
    Now, the action of $\SOrth(2)$ on $\Flag_{1,2}(V)$ is not effective; the weights $\omega_1=2,\omega_2=0,\omega_3=-2$ share a common factor of $2$. Thus, the matrix $\big[\begin{smallmatrix}
        -1 & 0 \\ 0 & -1
    \end{smallmatrix}\big]$ acts trivially and there is an induced $\PSOrth(2)$--action on $\Flag_{1,2}(V)$. This has the effect of halving the order of each stabilizer and thus halving the weights for each tangent space. This is why the $\PSOrth(2)$--difference matrix for $F^\bullet$ is
    \begin{center}
        \begin{tabular}{c|ccc}
             $F^\bullet$ & $v_{2}$ & $v_{0}$ & $v_{-2}$  \\ \hline
             $v_{2}$ & $\ast$ & $\ast$  & $\ast$\\
             $v_{0}$ & $-1$ & $\ast$  & $\ast$\\
             $v_{-2}$ & $-2$ & $-1$  & $\ast$
        \end{tabular}
    \end{center}
    Thus, the difference matrices provide a concise summary of the weights of a $\SOrth(2)$ (or $\PSOrth(2)$) action on the flag variety.
\end{xmpl}

\subsection{Fixed and exceptional points in the fiber}

Now, we will describe the fundamental observation that allows us to relate the the fixed and exceptional points of the fiber to the fixed and exceptional points of the flag variety.
\begin{lem}\label{CheapTrick}
    Let $\Omega_\rho^I\subsetneq\mathcal{F}_D$ be a Kapovich-Leeb-Porti domain for an $\iota$--Fuchsian representation $\rho=\iota\circ\phi:\pi_1(S_g)\to G$ and let $\bar{p}:\mathcal{F}_\eta\to\overline{\mathbb{H}^2}$ be the compactified $\SL(2,\mathbb{R})$--equivariant projection map described in Remark \ref{compactproject}. Let $M$ be the fiber of $\bar{p}$ above the origin $O$ in $\mathbb{H}^2$.
    \begin{enumerate}[label=\normalfont{(\roman*)}]
        \item If the stabilizer of $x\in\mathcal{F}_\eta$ in $\SOrth(2)$ has order greater than $2$, then $x\in M$. Thus, $\Fix(M,m,\SOrth(2))=\Fix(\mathcal{F}_\eta,m,\SOrth(2))$ for $m\neq 1,2$.
        \item Assume $\iota:\SL(2,\mathbb{R})\to G$ induces a representation $\bar{\iota}:\PSL(2,\mathbb{R})\to\text{P}G$ so that the equivariant $\bar{p}:\mathcal{F}_\eta\to\overline{\mathbb{H}^2}$ is $\bar{\iota}$--equivariant. Then, if the stabilizer of $x\in\mathcal{F}_\eta$ in $\PSOrth(2)$ is non-trivial, then $x\in M$. Thus, $\Fix(M,m,\PSOrth(2))=\Fix(\mathcal{F}_\eta,m,\PSOrth(2))$ for $m\neq 1$.
        \item If $x\in M$ is an isolated fixed point for the action by $\SOrth(2)$, then the sign of $T_xM$ equals the sign of $T_x\mathcal{F}_\eta$.
    \end{enumerate}
\end{lem}
\begin{proof}
For (i), let $g\in\Stab_{\SOrth(2)}(x)$ be an element of order greater than $2$ and consider the projection $\bar{p}(x)\in\overline{\mathbb{H}}{}^2$. 
    Then,
    \[g\cdot x=x\Rightarrow g\cdot \bar{p}(x)=\bar{p}(x).\]
    However, $g$, as an element in $\SOrth(2)$ whose order is greater than $2$, has only one fixed point in $\overline{\mathbb{H}^2}$, viz. $O$. Therefore, $p(x)=O$ and $x\in M=p^{-1}(O)$. \par
    Assume we have an induced representation $\bar{\iota}:\PSL(2,\mathbb{C})\to \text{P}G$. Then, if $x\in\mathcal{F}_D$ has non-trivial stabilizer in $\PSOrth(2)$, then its stabilizer in $\SOrth(2)$ must have order greater than $2$ and thus (ii) follows. \par
    For (iii), we note that the equivariance of $p$ gives the exact sequence of $S^1$--representation
    \[0\to T_xM\to T_x\mathcal{F}_\eta\xrightarrow{dp}T_O\mathbb{H}^2\to O.\]
    Therefore, we get the isomorphism of oriented $S^1$--representations
    \[T_x\mathcal{F}_\eta\cong T_xM\oplus T_O\mathbb{H}^2\cong T_xM\oplus V_2^0.\]
    Because the weight for $T_O\mathbb{H}^2$ is positive, the sign of $T_xM$ and $T_x\mathcal{F}_\eta$ must agree.
\end{proof}
Note that if the fixed and exceptional points of $\mathcal{F}_\eta$ only lies on exceptional spheres, then the definition of the tangential weight graph still makes sense for $\mathcal{F}_\eta$. This gives us the following consequence:
\begin{prop}\label{SameTWG}
Continuing the notation from Lemma \ref{CheapTrick}
    \begin{enumerate}[label=\normalfont{(\roman*)}]
        \item If the $\iota(\SOrth(2))$--action on $\mathcal{F}_\eta$ is effective with tangential weight graph $\mathcal{G}$ and $\mathcal{G}$ does not have any square vertices or edges labeled with weight $2$, then $\mathcal{G}$ is also the tangential weight graph for $M$.
        \item Assume that $\iota:\SL(2,\mathbb{R})\to G$ induces a representation $\bar{\iota}:\PSL(2,\mathbb{R})\to PG$. If the $\bar{\iota}(\PSOrth(2))$--action on $\mathcal{F}_D$ is effective with tangential weight graph $\mathcal{G}$ and $\mathcal{G}$ have any square vertices, then $\mathcal{G}$ is also the tangential weight graph for $M$.
    \end{enumerate}
\end{prop}
\begin{proof}
    For (i), assume that the $\iota(\SOrth(2))$--action on $\mathcal{F}_\eta$ is effective with tangential weight graph $\mathcal{G}$ and no square vertices or edges label with weight $2$. This means, for any fixed or exceptional point in $\mathcal{F}_\eta$, Proposition \ref{CheapTrick}(i) applies. Therefore, the tangential weight graph of $M$ has the same vertices, edges, and weights labeling edges. Furthermore, Proposition (iii) implies the signs labeling the round vertices are the same. \par
    The proof of (ii) is similar.
\end{proof}
The propositions \ref{ExceptInFlags} and  \ref{SameTWG} give a uniform approach to determine the fixed and exceptional points for the fibers of domains of discontinuities in flag varieties. Namely, difference matrices provide a compact framework to calculate the tangential weight graph for $\mathcal{F}_D$ and then Proposition \ref{SameTWG} relates this to the tangential weight graph for $M$. We will implement this strategy explicitly in the following section.
\section{Topology of fibers for $\dim_{\mathbb{C}}\Omega=3$}\label{sec:calculations4fibers}
Now that we have developed the necessary tools, we are ready to examine each of the domains of discontinuity listed in Proposition \ref{fullcases} and determine the topology and structure group action in each case. For each choice of $G$ and $\mathcal{F}_\eta$, we will consider the appropriate irreducible and reducible representation $\iota:\SL(2,\mathbb{R})\to G$ and describe the induced $\SOrth(2)$--action on $\mathcal{F}_\eta$. Then, by applying the techniques described in Section \ref{sec:TWD4flags}, we can determine the tangential weight graph of the fiber $M_\rho^I$ in each case. Finally, by comparing these tangential weight graphs to the tangential weight graphs for Hirzebruch surfaces, we can apply Theorem \ref{SmoothClassification} to determine $M_\rho^I$, up to equivariant diffeomorphism.
\subsection{For $G=\SL(3,\mathbb{C})$, $\Omega\subseteq\Flag(\mathbb{C}^3)$}
Let $G=\SL(3,\mathbb{C})$ and let $\mathcal{F}_\eta=\Flag(\mathbb{C}^3)$. Recall from Section \ref{handledomains} that if $\iota_{\bldmth{f},\bldmth{i}}=\rho_3$, then any $\iota_{\bldmth{f},\bldmth{i}}$--Fuchsian representation $\rho_{\bldmth{f},\bldmth{i}}:\Gamma\to\SL(3,\mathbb{C})$ is Borel Anosov and admits a domain of discontinuity
\begin{align*}
    \Omega_{\bldmth{f},\bldmth{i}}=\{F^\bullet\in\Flag(\mathbb{C}^3):F^1\neq \xi_{\bldmth{f},\bldmth{i}}^1(t)\text{ and }F^2\neq \xi^2_{\bldmth{f},\bldmth{i}}(t)\text{ for all }t\in\mathbb{RP}^1\},
\end{align*}
where $\xi_{\bldmth{f},\bldmth{i}}:\mathbb{RP}^1\to\Flag(\mathbb{C}^3)$ is the limit curve. We denote the fiber above the origin $O$ of the $\SL(2)$--equivariant projection $p_{\bldmth{f},\bldmth{i}}:\Omega_{\bldmth{f},\bldmth{i}}\to\mathbb{H}^2$ by $M_{\bldmth{f},\bldmth{i}}$, that is
$$ M_{\bldmth{f},\bldmth{i}} := p_{\bldmth{f},\bldmth{i}}^{-1}(O).$$

Similarly if $\iota_{\bldmth{f},\bldmth{r}}=\rho_2\oplus\rho_1$, then any $\iota_{\bldmth{f},\bldmth{r}}$--Fuchsian representation $\rho_{\bldmth{f},\bldmth{r}}:\Gamma\to\SL(3,\mathbb{C})$ is Borel Anosov and admits a domain of discontinuity
\begin{align*}
\Omega_{\bldmth{f},\bldmth{r}}=\{F^\bullet\in\Flag(\mathbb{C}^3):F^1\neq \xi_{\bldmth{f},\bldmth{r}}^1(t)\text{ and }F^2\neq \xi^2_{\bldmth{f},\bldmth{r}}(t)\text{ for all }t\in\mathbb{RP}^1\}
\end{align*}
where $\xi_{\bldmth{f},\bldmth{r}}:\mathbb{RP}^1\to\Flag(\mathbb{C}^3)$ is the limit curve. We denote the fiber above the origin of the $\SL(2)$--equivariant projection $p_{\bldmth{f},\bldmth{r}}:\Omega_{\bldmth{f},\bldmth{r}}\to\mathbb{H}^2$ by $M_{\bldmth{f},\bldmth{r}}$, that is
$$ M_{\bldmth{f},\bldmth{r}} := p_{\bldmth{f},\bldmth{r}}^{-1}(O).$$

\subsubsection{The irreducible case}

Now, we will determine the topology and structure group action for $M_{\bldmth{f},\bldmth{i}}$. \par
Let $V=\mathbb{C}^{(2)}[X,Y]$ be $\mathbb{C}$--span of $$X^2,XY,Y^2$$ in $\mathbb{C}[X,Y]$. Then, we can represent $\iota_{\bldmth{f},\bldmth{i}}:\SL(2,\mathbb{C})\to\SL(3,\mathbb{C})$ by
\[\iota\left(\begin{bmatrix}
    a & b \\ c & d
\end{bmatrix}\right)\cdot f(X,Y)=f(dX-bY,-cX+aY).\]
By restricting $\iota_{\bldmth{f},\bldmth{i}}$ to $\SOrth(2)$, we get a linear $\SOrth(2)$--action on $V$. $V$ has weight vectors $f_2=(X-iY)^2, f_0=X^2+Y^2,f_{-2}=(X+iY)^2$ with weights 
\begin{align*}
    \iota_{\bldmth{f},\bldmth{i}}(R_\theta)\cdot f_2&=e^{2i\theta}f_2 \\
    \iota_{\bldmth{f},\bldmth{i}}(R_\theta)\cdot f_0&=f_0 \\
    \iota_{\bldmth{f},\bldmth{i}}(R_\theta)\cdot f_{-2}&=e^{-2i\theta}f_{-2}
\end{align*}
Note that because $\iota_{\bldmth{f},\bldmth{i}}(\{\pm I_{2}\})\subseteq\{\pm I_{3}\}$, there is an induced $\PSL(2,\mathbb{C})$--action on $\Flag(\mathbb{C}^3)$ with $g\cdot x=\iota_{\bldmth{f},\bldmth{i}}(\tilde{g})\cdot x$ for $g\in\PSL(2,\mathbb{C})$, $g=\pm\tilde{g}$ and $x\in\Flag(\mathbb{C}^3)$. Furthermore, $\Omega_{\bldmth{f},\bldmth{i}}$ will be $\PSL(2,\mathbb{R})$--invariant and $p_{\bldmth{f},\bldmth{i}}:\Omega_{\bldmth{f},\bldmth{i}}\to\mathbb{H}^2$ will be $\PSL(2,\mathbb{R})$--equivariant.  By Lemma \ref{CheapTrick}, $$\Fix(M_{\bldmth{f},\bldmth{i}},m,\PSOrth(2))=\Fix(\Flag(\mathbb{C}^3),m,\PSOrth(2))$$ for all $m\neq 1$.  Therefore,
\begin{align*}
\Fix(M_{\bldmth{f},\bldmth{i}},0,\PSOrth(2))=\{F_{2,0}^\bullet,F_{2,-2}^\bullet,F_{0,2}^\bullet,F_{0,-2}^\bullet,F_{-2,2}^\bullet,F_{-2,0}^\bullet\},
\end{align*}
where $F_{i,j}^\bullet$ is the flag with basis $f_i,f_j$, i.e.\
\[F^\bullet_{i,j}=(\langle f_i\rangle), \langle f_i,f_j\rangle)\text{ for $i,j\in\{-2,0,2\}$}.\]

We are now ready to see the following result.
\begin{prop}\label{Case1TWG}
    The tangential weight graph for the $\PSOrth(2)$--action on $\Flag(\mathbb{C}^3)$ with $\iota_{\bldmth{f},\bldmth{i}}=\rho_3$ looks as in Figure \ref{fig:case1TWG}.
    The tangential weight graph for the $\PSOrth(2)$--action on $\Flag(\mathbb{C}^3)$ with $\iota_{\bldmth{f},\bldmth{i}}=\rho_3$ looks as in Figure \ref{fig:case1TWG}.
\begin{figure}
    \centering
    \begin{tikzpicture}[node distance={30mm}, main/.style = {draw, circle}]
\node[main] (1) at (0,0) {$+$}; 
\node[main] (2) at (1.5,0)  {$-$}; 
\node[main] (3) at (3,0) {$+$}; 
\node[main] (4) at (4.5,0)  {$-$}; 
\node[main] (5) at (6,0) {$+$}; 
\node[main] (6) at (7.5,0)  {$-$}; 
\draw[] (1) -- node[midway, above,] {$2$} (2);
\draw[-] (3) -- node[midway, above,] {$2$} (4);
\draw[-] (5) -- node[midway, above,] {$2$} (6);

\node[] at (0,0.6) {$F_{2,-2}^\bullet$}; 
\node[] at (1.5,0.6) {$F_{-2,2}^\bullet$}; 
\node[] at (3,0.6) {$F_{-2,0}^\bullet$}; 
\node[] at (4.5,0.6) {$F_{2,0}^\bullet$}; 
\node[] at (6,0.6) {$F_{0,2}^\bullet$}; 
\node[] at (7.5,0.6) {$F_{0,-2}^\bullet$}; 
\end{tikzpicture}
    \caption{Caption}
    \label{fig:case1TWG}
\end{figure}
\end{prop}

\begin{proof}
    To determine tangential weight graph, we determine the difference matrix for the action on $\PSOrth(2)$ at each $F_{i,j}^\bullet\in F$, see Definition \ref{DiffMatrix}.
    
    For $F_{2,-2}^\bullet$ and $F_{-2,2}^\bullet$, the $\PSOrth(2)$--difference matrices are
    \begin{align*}
        \begin{tabular}{c|ccc}
             $F_{2,-2}^\bullet$ & $f_2$ & $f_{-2}$ & $f_{0}$  \\ \hline
             $f_2$ & $\ast$ & $\ast$  & $\ast$\\
             $f_{-2}$ & $-2$ & $\ast$  & $\ast$\\
             $f_{0}$ & $-1$ & $1$  & $\ast$
        \end{tabular} & & \begin{tabular}{c|ccc}
             $F_{-2,2}^\bullet$ & $f_{-2}$ & $f_2$ & $f_{0}$  \\ \hline
             $f_{-2}$ & $\ast$ & $\ast$  & $\ast$\\
             $f_2$ & $2$ & $\ast$  & $\ast$\\
             $f_{0}$ & $1$ & $-1$  & $\ast$
        \end{tabular}
    \end{align*}
    By Proposition \ref{FlagTangentModel}, the sign of $F_{2,-2}^\bullet$ is $(+)$ and the sign of $F_{-2,2}^\bullet$ is $(-)$ and there is a weight $2$ $\mathbb{CP}^1$ connecting $F_{2,-2}^\bullet$ and $F_{-2,2}^\bullet$. \\
    For $F_{-2,0}^\bullet$ and $F_{2,0}^\bullet$, the $\PSOrth(2)$--difference matrices are
    \begin{align*}
        \begin{tabular}{c|ccc}
             $F_{-2,0}^\bullet$ & $f_{-2}$ & $f_{0}$ & $f_{2}$  \\ \hline
             $f_{-2}$ & $\ast$ & $\ast$  & $\ast$\\
             $f_{0}$ & $1$ & $\ast$  & $\ast$\\
             $f_{2}$ & $2$ & $1$  & $\ast$
        \end{tabular} & & \begin{tabular}{c|ccc}
             $F_{2,0}^\bullet$ & $f_{2}$ & $f_{0}$ & $f_{-0}$  \\ \hline
             $f_{2}$ & $\ast$ & $\ast$  & $\ast$\\
             $f_{0}$ & $-1$ & $\ast$  & $\ast$\\
             $f_{-2}$ & $-2$ & $-1$  & $\ast$
        \end{tabular}
    \end{align*}
    Thus, the sign of $F_{-2,0}^\bullet$ is $(+)$ and the sign of $F_{2,0}^\bullet$ is $(-)$ and there is a weight $2$ $\mathbb{CP}^1$ connecting $F_{-2,0}^\bullet$ and $F_{2,0}^\bullet$. \\
    For $F_{0,2}^\bullet$ and $F_{0,-2}^\bullet$, the $\PSOrth(2)$--difference matrices are
    \begin{align*}
        \begin{tabular}{c|ccc}
             $F_{0,2}^\bullet$ & $f_{0}$ & $f_{2}$ & $f_{-2}$  \\ \hline
             $f_{0}$ & $\ast$ & $\ast$  & $\ast$\\
             $f_{2}$ & $1$ & $\ast$  & $\ast$\\
             $f_{-2}$ & $-1$ & $-2$  & $\ast$
        \end{tabular} & & \begin{tabular}{c|ccc}
             $F_{0,-2}^\bullet$ & $f_{0}$ & $f_{-2}$ & $f_{2}$  \\ \hline
             $f_{0}$ & $\ast$ & $\ast$  & $\ast$\\
             $f_{-2}$ & $-1$ & $\ast$  & $\ast$\\
             $f_{2}$ & $1$ & $2$  & $\ast$
        \end{tabular}
    \end{align*}
    Thus, the sign of $F_{0,2}^\bullet$ is $(+)$ and the sign of $F_{0,-2}^\bullet$ is $(-)$ and there is a weight $2$ $\mathbb{CP}^1$ connecting $F_{0,2}^\bullet$ and $F_{0,-2}^\bullet$. \\
    This determines the stated tangential weight graph. By Proposition \ref{SameTWG}, this is also the tangential weight graph for $M_{\bldmth{f},\bldmth{i}}$.
\end{proof}
Recall from Example \ref{ConnSumExample} that there is a smooth $S^1$--action on $(S^2\times S^2)\#(S^2\times S^2)$ with this tangential weight graph.
By Theorem \ref{SmoothClassification}, we get the following result.
\begin{corr}\label{fiberflgirr}
    For $\iota_{\bldmth{f},\bldmth{i}}=\rho_3$ and $\mathcal{F}_\eta=\Flag(\mathbb{C}^3)$,
    \begin{enumerate}
        \item[(i)] The $\PSL(2,\mathbb{R})$--equivariant fiber bundle projection $p_{\bldmth{f},\bldmth{i}}:\Omega_{\bldmth{f},\bldmth{i}}\to\mathbb{H}^2$ has fiber $M_{\bldmth{f},\bldmth{i}}$ which is diffeomorphic to $(S^2\times S^2)\#(S^2\times S^2)$. The structure group is $\PSOrth(2)$ and the action is as described in Figure \ref{fig:case1TWG}.
        \item[(ii)] The manifold $\mathcal{W}_{\bldmth{f},\bldmth{i}}$ is diffeomorphic to the bundle associated to the $\PSOrth(2)$--principal bundle over $S_g$ of Euler class $g-1$ and fiber $(S^2\times S^2)\#(S^2\times S^2)$ with the described $\PSOrth(2)$--action.
    \end{enumerate}
\end{corr}

\subsubsection{The reducible case}
Now, we will determine the topology and structure group action for $M_{\bldmth{f},\bldmth{r}}$. This case will require a slightly different argument compared to the other cases. \par
Let $V$ be the $\mathbb{C}$--span of $X,Y,Z$ in $\mathbb{C}[X,Y,Z]$. Then, we represent $\iota_{\bldmth{f},\bldmth{r}}=\rho_2\oplus \rho_1:\SL(2,\mathbb{C})\to\SL(3,\mathbb{C})$ by
\[\iota\left(\begin{bmatrix}
    a & b \\ c & d
\end{bmatrix}\right)\cdot f(X,Y,Z)=f(dX-cY,aY-bX,Z)\]
By restricting $\iota$ to $\SOrth(2)$, we get a linear $\SOrth(2)$--action on $V$. Then, $V$ has weight vectors $f_1=X-iY$, $f_0= Z$, $f_{-1}=X+iY$ with weights
\begin{align*}
    \iota_{\bldmth{f},\bldmth{r}}(R_\theta)\cdot f_1&=e^{i\theta}f_1\\
    \iota_{\bldmth{f},\bldmth{r}}(R_\theta)\cdot f_0&=f_0\\
    \iota_{\bldmth{f},\bldmth{r}}(R_\theta)\cdot f_{-1}&=e^{-i\theta}f_{-1}
\end{align*}
By Lemma \ref{CheapTrick},
\begin{align*}
    \Fix(M,0,\SOrth(2))=\{F_{1,0}^\bullet,F_{1,-1}^\bullet,F_{0,1}^\bullet,F_{0,-1}^\bullet,F_{-1,1}^\bullet,F_{-1,0}^\bullet\},
\end{align*}
where
\[F^\bullet_{i,j}=(\langle f_i\rangle), \langle f_i,f_j\rangle)\text{ for $i,j\in\{1,0,-1\}$}.\]
Because the weights for $V$ with $\iota_{\bldmth{f},\bldmth{r}}=\rho_2\oplus \rho_1$ are half the weights for $V$ with $\iota_{\bldmth{f},\bldmth{i}}=\rho_3$, the $\SOrth(2)$--difference matrix for $\Flag(\mathbb{C}^3)$ with $\iota_{\bldmth{f},\bldmth{r}}=$ is the same as the $\PSOrth(2)$--difference matrix for $\Flag(\mathbb{C}^3)$ with $\iota_{\bldmth{f},\bldmth{r}}=\rho_3$. Therefore, the tangential weight graph of the $\SOrth(2)$--action on $\Flag(\mathbb{C}^3)$ for $\iota_{\bldmth{f},\bldmth{r}}$ is as in Figure \ref{fig:fakecase2TWG}.
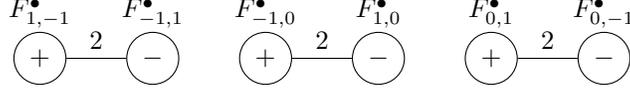
\begin{figure}
    \centering
    \begin{center}
        \begin{tikzpicture}[node distance={30mm}, main/.style = {draw, circle}]
\node[main] (1) at (0,0) {$+$}; 
\node[main] (2) at (1.5,0)  {$-$}; 
\node[main] (3) at (3,0) {$+$}; 
\node[main] (4) at (4.5,0)  {$-$}; 
\node[main] (5) at (6,0) {$+$}; 
\node[main] (6) at (7.5,0)  {$-$}; 
\draw[] (1) -- node[midway, above,] {$2$} (2);
\draw[-] (3) -- node[midway, above,] {$2$} (4);
\draw[-] (5) -- node[midway, above,] {$2$} (6);

\node[] at (0,0.6) {$F_{1,-1}^\bullet$}; 
\node[] at (1.5,0.6) {$F_{-1,1}^\bullet$}; 
\node[] at (3,0.6) {$F_{-1,0}^\bullet$}; 
\node[] at (4.5,0.6) {$F_{1,0}^\bullet$}; 
\node[] at (6,0.6) {$F_{0,1}^\bullet$}; 
\node[] at (7.5,0.6) {$F_{0,-1}^\bullet$}; 
\end{tikzpicture}
    \end{center}
    \caption{Tangential weight graph for $\Flag(\mathbb{C}^3)$ for $\iota_{\bldmth{f},\bldmth{r}}=\rho_2\oplus\rho_1$}
    \label{fig:fakecase2TWG}
\end{figure}
Note however that this tangential weight graph specifically has edges with weight $2$. This means we cannot directly apply Proposition \ref{SameTWG}. In fact, the tangential weight graph for $M_{\bldmth{f},\bldmth{r}}$ is not the same as we will now prove.
\begin{prop}\label{Case2TWG}
    The following hold:
    \begin{enumerate}[label=\normalfont{(\roman*)}]
        \item The intersection of the fiber $M_{\bldmth{f},\bldmth{r}}$ and $\Fix(\Flag(\mathbb{C}^3),2,\SOrth(2))$ is transverse. The points of intersection are precisely $\{F_{1,0}^\bullet,F_{1,-1}^\bullet,F_{0,1}^\bullet,F_{0,-1}^\bullet,F_{-1,1}^\bullet,F_{-1,0}^\bullet\}$.
        \item The tangential weight graph for $M_{\bldmth{f},\bldmth{r}}$ is as in Figure \ref{Case2TWG}.
        \begin{figure}
            \centering
            \begin{tikzpicture}[node distance={30mm}, main/.style = {draw, circle}]
\node[main] (1) at (0,0) {$+$}; 
\node[main] (2) at (1.5,0)  {$-$}; 
\node[main] (3) at (3,0) {$+$}; 
\node[main] (4) at (4.5,0)  {$-$}; 
\node[main] (5) at (6,0) {$+$}; 
\node[main] (6) at (7.5,0)  {$-$}; 

\node[] at (0,0.6) {$F_{1,-1}^\bullet$}; 
\node[] at (1.5,0.6) {$F_{-1,1}^\bullet$}; 
\node[] at (3,0.6) {$F_{-1,0}^\bullet$}; 
\node[] at (4.5,0.6) {$F_{1,0}^\bullet$}; 
\node[] at (6,0.6) {$F_{0,1}^\bullet$}; 
\node[] at (7.5,0.6) {$F_{0,-1}^\bullet$}; 
\end{tikzpicture}
            \caption{Caption}
            \label{fig:case2TWG}
        \end{figure}
    \end{enumerate}
\end{prop}
\begin{proof}
    Take an arbitrary $x\in M_{\bldmth{f},\bldmth{r}}\cap\Fix(\Flag(\mathbb{C}^3),2,\SOrth(2))$. By Proposition \ref{FlagTangentModel}, the tangent space $T_{x}\Flag(\mathbb{C}^3)$ under the action of $-I\in\SOrth(2)$ has eigenspaces $E_{1}$, $E_{-1}$ with $\dim_{\mathbb{C}}E_{1}=1$ and $\dim_{\mathbb{C}} E_{-1}=2$. Note that $T_x\Fix(\Flag(\mathbb{C}^3),2,\SOrth(2))$ is fixed by $-I$. Furthermore, if we fix an $\SOrth(2)$--invariant Hermitian metric on $\Flag(\mathbb{C}^3)$, then the equivariance of $\rho$ implies that $T^\perp_xM_{\bldmth{f},\bldmth{r}}$ is also fixed by $-I$. Thus, $T_x\Fix(\Flag(\mathbb{C}^3),2,\SOrth(2))=T_x^\perp M_{\bldmth{f},\bldmth{r}}=E_{1}$. In particular, $M_{\bldmth{f},\bldmth{r}}$ and $\Fix(\Flag(\mathbb{C}^3),2,\SOrth(2))$ intersect transversely at $x$.\\
    Therefore, the fiber $M_{\bldmth{f},\bldmth{r}}$ and $\Fix(\Flag(\mathbb{C}^3),2,\SOrth(2))$ intersect transversely and thus at finitely many points including $F_{1,0}^\bullet,F_{1,-1}^\bullet,F_{0,1}^\bullet,F_{0,-1}^\bullet,F_{-1,1}^\bullet,F_{-1,0}^\bullet$. Because the intersection is also $\SOrth(2)$--invariant, it cannot contain any other points; otherwise, it would contain a $1$--dimensional orbit.\\
    The sign calculations still apply by Lemma \ref{CheapTrick}. Thus, the tangential weight graph for $M_{\bldmth{f},\bldmth{r}}$ must be as described.
\end{proof}
\begin{prop}\label{Case2Action}
    There is a smooth $S^1$--action on $(S^2\times S^2)\# (S^2\times S^2)$ which has the same tangential weight graph as in  Figure \ref{fig:case2TWG}
\end{prop}
\begin{proof}
    We resume our identification of $S^1=\{g\in\mathbb{C}:|g|=1\}$. Consider the $S^1$--action on $\mathbb{CP}^1\times\mathbb{CP}^1$ given by
    \[g\cdot ([z_0:z_1],[w_0:w_1])=([z_0:gz_1],[w_0:gw_1]).\]
    Following the same reasoning as in Example \ref{ConnSumExample}, the tangential weight graph is as in Figure \ref{fig:11actionTWG}
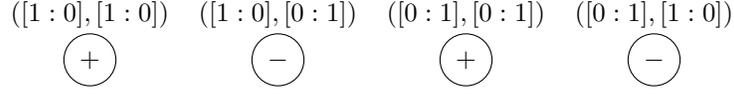
\begin{figure}
    \centering
    \begin{tikzpicture}[node distance={30mm}, main/.style = {draw, circle}]
\node[main] (1) at (0,0) {$+$}; 
\node[main] (2) at (2.5,0)  {$-$}; 
\node[main] (3) at (5,0) {$+$}; 
\node[main] (4) at (7.5,0)  {$-$}; 

\node[] at (0,0.6) {$([1:0],[1:0])$}; 
\node[] at (2.5,0.6) {$([1:0],[0:1])$}; 
\node[] at (5,0.6) {$([0:1],[0:1])$}; 
\node[] at (7.5,0.6) {$([0:1],[1:0])$};  
\end{tikzpicture}
    \caption{Tangential weight graph for an $S^1$--action on $\mathbb{CP}^1\times\mathbb{CP}^1$}
    \label{fig:11actionTWG}
\end{figure}
Applying the same connected sum construction as in Example \ref{ConnSumExample}, we get an action on $(S^2\times S^2)\#(S^2\times S^2)$ with the described tangential weight graph. 
\end{proof}
Then, again by Theorem \ref{SmoothClassification}, we are able to conclude the topology and structure group action of $M_{\bldmth{f},\bldmth{r}}$ as follows.
\begin{corr}\label{fiberflgred}
    For $\iota_{\bldmth{f},\bldmth{r}}=\rho_2\oplus\rho_1$ and $\mathcal{F}_\eta=\Flag(\mathbb{C}^3)$,
    \begin{enumerate}[label=\normalfont{(\roman*)}]
        \item The $\SL(2,\mathbb{R})$--equivariant fiber bundle projection $p_{\bldmth{f},\bldmth{r}}:\Omega_{\bldmth{f},\bldmth{r}}\to\mathbb{H}^2$ has fiber $M_{\bldmth{f},\bldmth{r}}$ which is diffeomorphic to $(S^2\times S^2)\#(S^2\times S^2)$. The structure group is $\SOrth(2)$ and the action is as described in Figure \ref{fig:case2TWG}.
        \item The manifold $\mathcal{W}_{\bldmth{f},\bldmth{r}}$ is diffeomorphic to the bundle associated to the $\SOrth(2)$--principal bundle over $S_g$ of Euler class $2g-2$ and fiber $(S^2\times S^2)\#(S^2\times S^2)$ with the described $\SOrth(2)$--action.
    \end{enumerate}
\end{corr}
\subsubsection{Obstruction to complex and symplectic structures}
Before we move onto the domains in other flag varieties, we remark on a consequence of the fibers $M_{\bldmth{f},\bldmth{i}}$ and $M_{\bldmth{f},r}$ being diffeomorphic to $(S^2\times S^2)\# (S^2\times S^2)$. 
\begin{prop}
    The $4$--manifold $(S^2\times S^2)\# (S^2\times S^2)$ does not admit an almost complex structure, and hence does not admit a complex or symplectic structure.
\end{prop}
\begin{proof}
    A classic result of Hirzebruch \cite[p.\ 777]{HirzebruchSignEulCong} says that for any almost complex manifold $M$ of dimension $4k$, we have that
    \[\sigma(M)\equiv (-1)^{2k}\chi(M)\pmod{4}\]
    where $\sigma(M)$ is the signature of the intersection form on $H_{2k}(M)$ and $\chi(M)$ is the Euler characteristic of $M$. \par
    On the other hand for $M=(S^2\times S^2)\# (S^2\times S^2)$, we have
    \[\sigma(M)=\sigma\left(S^2\times S^2\right)+\sigma\left(S^2\times S^2\right)=0+0=0\]
    and
    \[\chi(M)=\chi\left(S^2\times S^2\right)+\chi\left(S^2\times S^2\right)-2=4+4-2=6.\]
    Therefore, $M=(S^2\times S^2)\# (S^2\times S^2)$ does not admit an almost complex structure.
\end{proof}
This leads to the following corollary.
\begin{corr}
    For the domain of discontinuity $\Omega=\Omega_{\bldmth{f},\bldmth{i}}$ (resp. $\Omega_{\bldmth{f},\bldmth{r}}$),
    \begin{enumerate}[label=\normalfont{(\roman*)}]
        \item There is no holomorphic fiber bundle projection $\pi:\Omega\to\mathbb{H}^2$.
        \item There is no symplectic fiber bundle projection $\pi:\Omega\to\mathbb{H}^2$, in the sense of \cite{lalonde2006symplecticstructuresfiberbundles}.
    \end{enumerate}
\end{corr}
Part (i) was shown, in greater generality, with complex-analytic methods, by Dumas and Sanders \cite[Thm.\ D]{dumas-sanders}.
\subsection{For $G=\SL(4,\mathbb{C})$, $\Omega\subseteq\mathbb{CP}^3$}
Let $G=\SL(4,\mathbb{C})$ and let $\mathcal{F}_\eta=\mathbb{CP}^3$. Recall from Section \ref{handledomains} if $\iota_{\bldmth{p},\bldmth{i}}=\rho_4$, then any $\iota_{\bldmth{p},\bldmth{i}}$--Fuchsian representation $\rho_{\bldmth{p},\bldmth{i}}:\Gamma\to\SL(4,\mathbb{C})$ would be $2$--Anosov and admits a domain of discontinuity
\begin{align*}
    \Omega_{\bldmth{p},\bldmth{i}}=\{F^1\in\mathbb{CP}^3:F^1\nsubseteq \xi^2_{\bldmth{p},\bldmth{i}}(t)\text{ for all }t\in\mathbb{RP}^1\},
\end{align*}
where $\xi_{\bldmth{p},\bldmth{i}}:\mathbb{RP}^1\to\Gr_2(\mathbb{C}^4)$ is the limit curve. We denote the fiber above the origin of the $\SL(2)$--equivariant projection $p_{\bldmth{p},\bldmth{i}}:\Omega_{\bldmth{p},\bldmth{i}}\to\mathbb{H}^2$ by $M_{\bldmth{p},\bldmth{i}}$. \par
Similarly if $\iota_{\bldmth{p},\bldmth{r}}=\rho_2\oplus\rho_2$, then any $\iota_{\bldmth{p},\bldmth{r}}$--Fuchsian representation $\rho_{\bldmth{p},\bldmth{r}}:\Gamma\to\SL(4,\mathbb{C})$ is $2$--Anosov and admits a domain of discontinuity
\begin{align*}
    \Omega_{\bldmth{p},\bldmth{r}}=\{F^1\in\mathbb{CP}^3:F^1\nsubseteq \xi^2_{\bldmth{p},\bldmth{r}}(t)\text{ for all }t\in\mathbb{RP}^1\},
\end{align*}
where $\xi_{\bldmth{p},\bldmth{r}}:\mathbb{RP}^1\to\Gr_2(\mathbb{C}^4)$ is the limit curve. We denote the fiber above the origin of the $\SL(2)$--equivariant projection $p_{\bldmth{p},\bldmth{i}}:\Omega_{\bldmth{p},\bldmth{i}}\to\mathbb{H}^2$ by $M_{\bldmth{p},\bldmth{i}}$. 
\subsubsection{The irreducible case}\label{irrSL4}
In this section we will determine the topology and the structure group action for $M_{\bldmth{p},\bldmth{i}}$.

Let $V=\mathbb{C}^{(3)}[X,Y]$ be the $\mathbb{C}$--span of $X^3,X^2Y,XY^2,Y^3$ in $\mathbb{C}[X,Y]$. Then, we can represent $\iota_{\bldmth{p},\bldmth{i}}=\rho_4:\SL(2,\mathbb{R})\to\SL(4,\mathbb{C})$ by
\[\iota\left(\begin{bmatrix}
    a & b \\ c & d
\end{bmatrix}\right)\cdot f(X,Y)=f(dX-bY,aY-cX).\]
By restricting $\iota_{\bldmth{p},\bldmth{r}}$ to $\SOrth(2)$, we get a linear $\SOrth(2)$=-action on $V$. The $\SOrth(2)$--module $V$ has weight vectors $f_3=(X-iY)^3$, $f_1=(X-iY^2)(X+iY)$, $f_{-1}=(X+iY)(X-iY)^2$, $f_{-3}=(X+iY)^3$
with weights
\begin{align*}
    \iota_{\bldmth{p},\bldmth{i}}(R_{\theta})\cdot f_w&=e^{wi\theta}f_w.
\end{align*}
Note that because $\iota(-I_2)=-I_4$, $\iota_{\bldmth{p},\bldmth{i}}$ gives a $\PSL(2,\mathbb{C})$--action on $\mathbb{CP}^3$ with $g\cdot x=\iota_{\bldmth{p},\bldmth{i}}(\tilde{g})\cdot x$ for $g\in\PSL(2,\mathbb{C})$, $g=\pm\tilde{g}$ and $x\in\mathbb{CP}^3$. By Lemma \ref{CheapTrick} we have that,
\[F=\Fix(M_{\bldmth{p},\bldmth{i}},0,\PSOrth(2))=\{\ell_{3},\ell_{1},\ell_{-1},\ell_{-3}\},\]
where
\[\ell_w=\langle f_w\rangle\text{ for }w\in\{3,1,-1,-3\}.\]
\begin{prop}\label{Case3TWG}
    The tangential weight graph for the $\PSOrth(2)$--action on $M_{\bldmth{p},\bldmth{i}}$ looks as in Figure \ref{fig:case3TWG}.
    \begin{figure}
        \centering
        \begin{tikzpicture}[node distance={30mm}, main/.style = {draw, circle}]
\node[main] (1) at (0,0) {$+$}; 
\node[main] (2) at (1.5,0) {$+$}; 
\node[main] (3) at (3,0)  {$-$}; 
\node[main] (4) at (4.5,0)  {$-$}; 
\draw[] (1) -- node[midway, above,] {$2$} (2);
\draw[-] (2) -- node[midway, above,] {$3$} (3);
\draw[-] (3) -- node[midway, above,] {$2$} (4);

\node[] at (0,0.6) {$\ell_1$}; 
\node[] at (1.5,0.6) {$\ell_{-3}$}; 
\node[] at (3,0.6) {$\ell_3$}; 
\node[] at (4.5,0.6) {$\ell_{-1}$}; 
\end{tikzpicture}
        \caption{Tangential weight graph for $\PSOrth(2)$--action on $M_{\bldmth{p},\bldmth{i}}$}
        \label{fig:case3TWG}
    \end{figure}
\end{prop}
\begin{proof}
    To determine tangential weight graph, we determine the difference matrix for the action on $\PSOrth(2)$ at each $\ell_w\in F$.
    
    The $\PSOrth(2)$--difference matrices are
    \begin{align*}
        \begin{tabular}{c|cccc}
             $\ell_1$ & $f_1$ & $f_{3}$ & $f_{-1}$ & $f_{-3}$ \\ \hline
             $f_1$ & $\ast$ & $\ast$  & $\ast$ & $\ast$\\
             $f_{3}$ & $1$ & $\ast$  & $\ast$ & $\ast$\\
             $f_{-1}$ & $-1$ & $\ast$  & $\ast$  & $\ast$\\
             $f_{-3}$ & $-2$ & $\ast$ & $\ast$ & $\ast$
        \end{tabular} & & \begin{tabular}{c|cccc}
             $\ell_{-3}$ & $f_{-3}$ & $f_{3}$ & $f_{1}$ & $f_{-1}$ \\ \hline
             $f_{-3}$ & $\ast$ & $\ast$  & $\ast$ & $\ast$\\
             $f_{3}$ & $3$ & $\ast$  & $\ast$ & $\ast$\\
             $f_{1}$ & $2$ & $\ast$  & $\ast$  & $\ast$\\
             $f_{-1}$ & $1$ & $\ast$ & $\ast$ & $\ast$
        \end{tabular} \\ \\
        \begin{tabular}{c|cccc}
             $\ell_3$ & $f_3$ & $f_{1}$ & $f_{-1}$ & $f_{-3}$ \\ \hline
             $f_3$ & $\ast$ & $\ast$  & $\ast$ & $\ast$\\
             $f_{1}$ & $-1$ & $\ast$  & $\ast$ & $\ast$\\
             $f_{-1}$ & $-2$ & $\ast$  & $\ast$  & $\ast$\\
             $f_{-3}$ & $-3$ & $\ast$ & $\ast$ & $\ast$
        \end{tabular} & & \begin{tabular}{c|cccc}
             $\ell_{-1}$ & $f_{-1}$ & $f_{3}$ & $f_{1}$ & $f_{-3}$ \\ \hline
             $f_{-1}$ & $\ast$ & $\ast$  & $\ast$ & $\ast$\\
             $f_{3}$ & $2$ & $\ast$  & $\ast$ & $\ast$\\
             $f_{1}$ & $1$ & $\ast$  & $\ast$  & $\ast$\\
             $f_{-3}$ & $-1$ & $\ast$ & $\ast$ & $\ast$
        \end{tabular}
    \end{align*}
    Thus, the sign of $\ell_1$ is ($+$) and there is a weight $2$ $\mathbb{CP}^1$ connecting $\ell_1$ and $\ell_{-3}$. The sign of $\ell_{-3}$ is ($+$) and there is also a weight $3$ $\mathbb{CP}^1$ connecting $\ell_{-3}$ and $\ell_{3}$. The sign of $\ell_{3}$ is ($-$) and there is also a weight $2$ $\mathbb{CP}^1$ connecting $\ell_{3}$ and $\ell_{-1}$.
    This determines the stated tangential weight graph for $\mathbb{CP}^3$. By Proposition \ref{SameTWG}, this also tells us the tangential weight graph for $M_{\bldmth{p},\bldmth{i}}$.
\end{proof}
We note that the tangential weight graph in Figure \ref{fig:case3TWG} is the same as the tangential weight graph determined for $\Hir(2;-1,2)$ by Corollary \ref{ExceptionalHirTWG}.
Thus, by Theorem \ref{SmoothClassification}, we get the following result.
\begin{corr}\label{fiberprjirr}
    For $\iota_{\bldmth{p},\bldmth{i}}=\rho_4$ and $\mathcal{F}_\eta=\mathbb{CP}^3$,
    \begin{enumerate}
        \item[(i)] The $\PSL(2,\mathbb{R})$--equivariant fiber bundle projection $p_{\bldmth{p},\bldmth{i}}:\Omega_{\bldmth{p},\bldmth{i}}\to\mathbb{H}^2$ has fiber $M_{\bldmth{p},\bldmth{i}}$ is equivariantly diffeomorphic to $\Hir(2;-1,2)$, identifying $S^1\cong \PSOrth(2)$. In particular, $M_{\bldmth{p},\bldmth{i}}$ is diffeomorphic to $S^2\times S^2$.
        \item[(ii)] The manifold $\mathcal{W}_{\bldmth{p},\bldmth{i}}$ is diffeomorphic to the bundle associated to the $\PSOrth(2)$--principal bundle over $S_g$ of Euler class $g-1$ and fiber $S^2\times S^2$ with the described $\PSOrth(2)$--action.
    \end{enumerate}
\end{corr}

\subsubsection{The reducible case}

In this section, we will determine the topology and the structure group action for $M_{\bldmth{p},\bldmth{r}}$. This is the first case where we will see fixed surfaces in the fiber. 

Let $V$ be the $\mathbb{C}$--span of $X_1,Y_1,X_2,Y_2$ inside $\mathbb{C}[X_1,Y_1,X_2,Y_2]$. We represent $\iota_{\bldmth{p},\bldmth{r}}=\rho_2\oplus\rho_2:\SL(2,\mathbb{C})\to\SL(4,\mathbb{C})$ by the linear action on $V$ given by
\[\iota_{\bldmth{p},\bldmth{r}}\left(\begin{bmatrix}
    a & b \\ c & d
\end{bmatrix}\right)\cdot f(X_1,Y_1,X_2,Y_2)=f(dX_1-cY_1,aY_1-bX_1,dX_2-cY_2,aY_2-bX_2).\]
By restricting $\iota_{\bldmth{p},\bldmth{r}}$ to $\SOrth(2)$, we get a linear $\SOrth(2)$--action on $V$. Then, $V$ has weight vectors $e_1=X_1-iY_1,e_{-1}=X_1+iY_1,f_{1}=X_2-iY_2,f_{-1}=X_2+iY_2$ with weights
\begin{align*}
    \iota_{\bldmth{p},\bldmth{r}}(R_\theta)\cdot e_1&= e^{i\theta}e_1 & \iota_{\bldmth{p},\bldmth{r}}(R_\theta)\cdot e_{-1}&=e^{-i\theta}e_{-1}\\
    \iota_{\bldmth{p},\bldmth{r}}(R_\theta)\cdot f_1&= e^{i\theta}f_1 & \iota_{\bldmth{p},\bldmth{r}}(R_\theta)\cdot f_{-1}&=e^{-i\theta}f_{-1}
\end{align*}
Then, any $\SOrth(2)$--weight vector is either a $\mathbb{C}$--linear combination of $e_1,f_1$ or a $\mathbb{C}$--linear combination of $e_{-1},f_{-1}$. Because the spans of weight vectors are exactly the fixed points of the $\SOrth(2)$--action on $\mathbb{CP}^3$, the projective lines $\langle e_{1},f_{1}\rangle$ and $\langle e_{-1},f_{-1}\rangle$ are fixed surfaces in $\mathbb{CP}^3$ and thus also in $M_{\bldmth{p},\bldmth{r}}$, by Proposition \ref{CheapTrick}. Because there are no other fixed points, there are also no exceptional spheres in $M_{\bldmth{p},\bldmth{r}}$. Thus, the only tangential weight data to calculate are the Euler classes for these fixed surfaces. \par
In order to calculate this Euler class, we recall the relationship between Chern classes and Euler classes. For any rank $1$ complex vector bundle $E$ over a manifold $B$ has an associated principal $\mathbb{C}^\times$--bundle $P$ so that $E$ is isomorphic to $P\underset{\mathbb{C}^\times}{\times}\mathbb{C}$. Then, the Euler class of $P$ agrees with the first Chern class of $E$ and is denoted by $c_1(E)$. For more thorough discussion of Chern classes, see \cite{milnorstasheff}.
\begin{prop}\label{Case4TWG}
    The tangential weight graph for the $\PSOrth(2)$--action on $M_{\bldmth{p},\bldmth{r}}$ looks as in Figure \ref{fig:case4TWG}.
    \begin{figure}
        \centering
        \begin{tikzpicture}[node distance={30mm}, main/.style = {draw, circle}]
\node[draw] (1) at (0,-3.6) {$2$}; 
\node[draw] (4) at (2,-3.6)  {$-2$};

\end{tikzpicture}
        \caption{Tangential weight graph for $M_{\bldmth{p},\bldmth{r}}$}
        \label{fig:case4TWG}
    \end{figure}
\end{prop}
\begin{proof}
    Note that by \cite[Lemma 5.1]{FintushelSimply4Circle}, because there are only two fixed surfaces, the tangential weight graph has two square vertices with labels $e$ and $-e$ for some $e\in\mathbb{Z}$. Therefore, we only need to determine $|e|$ to determine the tangential weight graph.

    Fix an $\PSOrth(2)$--invariant Hermitian metric on $\mathcal{F}_\eta=\mathbb{CP}^3$.
    Consider the fixed surfaces $C_1=\langle e_1,f_1\rangle$ and $C_{-1}=\langle e_{-1},f_{-1}\rangle$. Choose a weight basis $v_1,v_2,v_3,v_4$ of $V$ such that $v_1,v_2\in C_1$ and $v_3,v_4\in C_{-1}$. Then, the $\PSOrth(2)$--weight matrix for $\ell=\langle v_1\rangle$ with this basis is
    \begin{align*}
        \begin{tabular}{c|cccc}
             $\ell$ & $v_1$ & $v_2$ & $v_3$  & $v_4$\\ \hline
             $v_1$ & $\ast$ & $\ast$  & $\ast$ & $\ast$\\
             $v_2$ & $0$ & $\ast$  & $\ast$ & $\ast$\\
             $v_3$ & $-1$ & $\ast$  & $\ast$ & $\ast$\\
             $v_4$ & $-1$ & $\ast$ & $\ast$ & $\ast$
        \end{tabular}
    \end{align*}
    Using the Hermitian metric, the tangent space $T_\ell\mathbb{CP}^3$ decomposes as $T_\ell C_1\oplus T_\ell^\perp C_1$. The $\PSOrth(2)$--actions on $T_\ell C_1$ has weight $0$ because $C_1$ is a fixed surface. $T_\ell C_1^\perp$ has only weight $-1$ and $\PSOrth(2)$ acts by complex multiplication. Then, $T_\ell^\perp C_1\cap T_\ell M_{\bldmth{p},\bldmth{r}}$ is real subspace of $T_\ell^\perp M_{\bldmth{p},\bldmth{r}}$ which is $\PSOrth(2)$--invariant. Because $\PSOrth(2)$ acts on $T_\ell ^\perp C_1$ by complex multiplication, $T_\ell^\perp C_1\cap T_\ell M_{\bldmth{p},\bldmth{r}}$ is actually a complex subspace of $T_{\ell}C_1^\perp$. By varying $\ell$ over $C_1$, we conclude that $T\mathbb{CP}^3|_{C}$ decomposes as
    \[T\mathbb{CP}^3|_{C_1}=TC_1\oplus(T^\perp C_1\cap TM_{\bldmth{p},\bldmth{r}})\oplus T^\perp M_{\bldmth{p},\bldmth{r}}|_{C_1}\]
    Let $E=T^\perp C_1\cap TM_{\bldmth{p},\bldmth{r}}$. 
    Note that $p_{\bldmth{p},\bldmth{r}}:\Omega_{\bldmth{p},\bldmth{r}}\to \mathbb{H}^2$ induces the trivialization $dp_{\bldmth{p},\bldmth{r}}:T^\perp M_{\bldmth{p},\bldmth{r}}\to T_O\mathbb{H}^2$ so $T^\perp M_{\bldmth{p},\bldmth{r}}|_{C_1}$ is a trivial bundle. Therefore, we have the relationship between first Chern classes.
    \begin{align*}
        c_1(T\mathbb{CP}^3|_{C_1})&=c_1(TC_1)+ c_1(E)\\
    \end{align*}
    By capping with the fundamental class of $C_1$ with the complex orientation, we conclude that the Chern number of $E$ as a rank $1$ complex vector bundle is $2$. This Chern number will agree in absolute value with $e$. Thus, $|e|=2$ and the tangential weight graph is as described.
\end{proof}
We note that the tangential weight graph in Figure \ref{fig:case3TWG} is the same as the tangential weight graph determined for $\Hir(2;1,0)$ by Proposition \ref{FixedHirTWG}.
Thus, by Theorem \ref{SmoothClassification}, we get the following result.
\begin{corr}\label{fiberprjred}
    For $\iota_{\bldmth{p},\bldmth{r}}=\rho_2\oplus \rho_2$ and $\mathcal{F}_\eta=\mathbb{CP}^3$,
    \begin{enumerate}
        \item[(i)] The $\PSL(2,\mathbb{R})$--equivariant fiber bundle projection $p_{\bldmth{p},\bldmth{r}}:\Omega_{\bldmth{p},\bldmth{r}}\to\mathbb{H}^2$ has fiber $M_{\bldmth{p},\bldmth{r}}$ which is equivariantly diffeomorphic to $\Hir(2;1,0)$, identifying $S^1\cong \PSOrth(2)$. In particular, $M_{\bldmth{p},\bldmth{r}}$ is diffeomorphic to $S^2\times S^2$.
        \item[(ii)] The manifold $\mathcal{W}_{\bldmth{p},\bldmth{r}}$ is diffeomorphic to the bundle associated to the $\PSOrth(2)$--principal bundle over $S_g$ of Euler class $g-1$ and fiber $S^2\times S^2$ with the described $\PSOrth(2)$--action.
    \end{enumerate}
\end{corr}
\subsection{For $G=\Sp(4,\mathbb{C})$, $\Omega\subseteq\Lag(\mathbb{C}^4)$}
Let $G=\Sp(4,\mathbb{C})$ and let $\mathcal{F}_\eta=\Lag(\mathbb{C}^4)$. Recall from Section \ref{handledomains} if $\iota_{\bldmth{l},\bldmth{i}}=\rho_4$, then any $\iota_{\bldmth{l},\bldmth{i}}$--Fuchsian representation $\rho_{\bldmth{l},\bldmth{i}}:\Gamma\to\Sp(4,\mathbb{C})$ would be $\{\alpha_1\}$--Anosov (where $\{\alpha_1\}$ is the $\Sp(4, \mathbb{C})$--stabilizer of a line in $\mathbb{CP}^3$) and admits a domain of discontinuity
\begin{align*}
    \Omega_{\bldmth{l},\bldmth{i}}=\{L\in\Lag(\mathbb{C}^4):\xi^1_{\bldmth{l},\bldmth{i}}(t)\nsubseteq L\text{ for all }t\in\mathbb{RP}^1\},
\end{align*}
where $\xi_{\bldmth{l},\bldmth{i}}:\mathbb{RP}^1\to\mathbb{CP}^3$ is the limit curve. We denote the fiber above the origin of the $\SL(2)$--equivariant projection $p_{\bldmth{l},\bldmth{i}}:\Omega_{\bldmth{l},\bldmth{i}}\to\mathbb{H}^2$ by $M_{\bldmth{l},\bldmth{i}}$. \par
Similarly if $\iota_{\bldmth{l},\bldmth{r}}=\rho_2\oplus\rho_1\oplus\rho_1$, then any $\iota_{\bldmth{l},\bldmth{r}}$--Fuchsian representation $\rho_{\bldmth{l},\bldmth{r}}:\Gamma\to\Sp(4,\mathbb{C})$ is $\{\alpha_1\}$--Anosov and admits a domain of discontinuity
\begin{align*}
    \Omega_{\bldmth{l},\bldmth{r}}=\{L\in\Lag(\mathbb{C}^4):\xi^1_{\bldmth{l},\bldmth{r}}(t)\nsubseteq L\text{ for all }t\in\mathbb{RP}^1\},
\end{align*}
where $\xi_{\bldmth{l},\bldmth{r}}:\mathbb{RP}^1\to\mathbb{CP}^3$ is the limit curve. We denote the fiber above the origin of the $\SL(2)$--equivariant projection $p_{\bldmth{l},\bldmth{r}}:\Omega_{\bldmth{l},\bldmth{r}}\to\mathbb{H}^2$ by $M_{\bldmth{l},\bldmth{r}}$. 
\subsubsection{The irreducible case}
Now, we will determine the topology and structure group action for $M_{lag,irr}$. The strategy is similar to before except we will need to pay special attention to the symplectic form we use.\par 
Let $V$ be $\mathbb{C}^{(2)}[X,Y]$. We represent $\iota_{\bldmth{l},\bldmth{i}}=\rho_4$ acting on $V$ as in Section \ref{irrSL4}. \\

Using the basis $f_3=(X-iY)^3$, $f_1=(X-iY)^2(X+iY)$, $f_{-1}=(X-iY)(X+iY)^2$, $f_{-3}=(X+iY)^3$, endow $V$ with the symplectic form $\omega$
\[\omega=\begin{bmatrix}
    & & & -4 \\ & & 1 & \\ & -1 & & \\ 4 & & &
\end{bmatrix}.\]
Then, the image of $\iota_{\bldmth{l},\bldmth{i}}=\rho_4$ is contained in $\Sp(V,\omega)$. Thus, we can think of $\iota_{\bldmth{l},\bldmth{i}}$ as the principal representation into $\Sp(4,\mathbb{C})$.\\
Our goal is to find the fixed and exceptional points in $\Lag(V,\omega)$. First, let us apply our usual methods to find the fixed and exceptional points in $\Gr_2(V)$. Recall that $f_3,f_1,f_{-1},f_{-3}$ is a $\SOrth(2)$--weight basis for $V$ with weights
\[\iota_{\bldmth{l},\bldmth{i}}(R_{\theta})\cdot f_w=e^{wi\theta}f_w.\]
Because the weights are distinct, a fixed point in $\Gr_2(V)$ corresponds to a choice of $2$ basis vector from $f_3,f_1,f_{-1},f_{-3}$. Thus, we have the fixed points
\begin{align*}
    L_{w_1,w_2}&=\langle f_{w_1},f_{w_2}\rangle \text{ for }w_1,w_2\in\{3,1,-1,-3\}\text{ with }w_1>w_2
\end{align*}
By inspection, $L_{3,1}$, $L_{3,-1}$, $L_{1,-3}$, $L_{-1,-3}$ are Lagrangian but $L_{3,-3}$ and $L_{1,-1}$ are not.
\begin{prop}\label{Case5TWG}
    The tangential weight graph for the $\PSOrth(2)$--action on $M_{\bldmth{l},\bldmth{i}}$ looks as in Figure \ref{fig:case5TWG}.
    \begin{figure}
        \centering
        \begin{tikzpicture}[node distance={30mm}, main/.style = {draw, circle}]
\node[main] (1) at (0,0) {$+$}; 
\node[main] (2) at (1.5,0) {$+$}; 
\node[main] (3) at (3,0)  {$-$}; 
\node[main] (4) at (4.5,0)  {$-$}; 
\draw[] (1) -- node[midway, above,] {$3$} (2);
\draw[-] (2) -- node[midway, above,] {$2$} (3);
\draw[-] (3) -- node[midway, above,] {$3$} (4);

\node[] at (0,0.6) {$L_{3,{-1}}$}; 
\node[] at (1.5,0.6) {$L_{-1,-3}$}; 
\node[] at (3,0.6) {$L_{3,1}$}; 
\node[] at (4.5,0.6) {$L_{3,-1}$}; 
\end{tikzpicture}
        \caption{Tangential weight graph for $\PSOrth(2)$--action on $M_{\bldmth{l},\bldmth{i}}$}
        \label{fig:case5TWG}
    \end{figure}
\end{prop}
\begin{proof}
    Now for each fixed Lagrangian $L$, we want a model for $T_{L}\Lag(V,\omega)$. As defined in Definition \ref{affinechart}, let $\mathfrak{u}=\mathbb{C}E_{31}\oplus\mathbb{C}E_{32}\oplus\mathbb{C}E_{41}\oplus\mathbb{C}E_{42}$ and $\varphi_{i,j}:\mathfrak{u}\to \Gr_2(V)$ be the affine charts relative to the flag bases
\begin{align*}
    f_3,f_1,f_{-1},f_{-3} & \text{ if $i=3$, $j=1$}\\
    f_3,f_{-1},f_{1},f_{-3} & \text{ if $i=3$, $j=-1$}\\
    f_1,f_{-3},f_{3},f_{-1} & \text{ if $i=1$, $j=-3$}\\
    f_{-1},f_{-3},f_{3},f_{1} & \text{ if $i=-1$, $j=-3$}.
\end{align*}
Then, for $L_{3,1}$, we have:
\begin{align*}
    \varphi_{3,1}^{-1}(\Lag(V,\omega))&=\left\{\begin{bmatrix}
        0 & 0 & 0 &0  \\ 0 & 0 & 0 & 0\\ a & b & 0 & 0\\ c & d & 0 & 0
    \end{bmatrix}:\begin{bmatrix}
        1 & 0 & 0 &0  \\ 0 & 1 & 0 & 0\\ a & b & 1 & 0\\ c & d & 0 & 1
    \end{bmatrix}L_{3,1}\text{ is Lagrangian}\right\}\\
    &=\left\{\begin{bmatrix}
        0 & 0 & 0 &0  \\ 0 & 0 & 0 & 0\\ a & b & 0 & 0\\ c & d & 0 & 0
    \end{bmatrix}:\omega(f_3+af_{-1}+cf_{-3},f_1+bf_{-1}+df_{-3})=0\right\}\\
    &=\left\{\begin{bmatrix}
        0 & 0 & 0 &0  \\ 0 & 0 & 0 & 0\\ a & b & 0 & 0\\ c & d & 0 & 0
    \end{bmatrix}:4d+a=0\right\}.
\end{align*}
If we compare it to the $\PSOrth(2)$--difference matrix for $L_{3,1}$
\[\begin{tabular}{c|cccc}
             $L_{3,1}$ & $f_3$ & $f_{1}$ & $f_{-1}$ & $f_{-3}$ \\ \hline
             $f_3$ & $\ast$ & $\ast$  & $\ast$ & $\ast$\\
             $f_{1}$ & $\ast$ & $\ast$  & $\ast$ & $\ast$\\
             $f_{-1}$ & $-2$ & $-1$  & $\ast$  & $\ast$\\
             $f_{-3}$ & $-3$ & $-2$ & $\ast$ & $\ast$
        \end{tabular},\]
we see that $T_{L_{3,1}}\Lag(V,\omega)$ is isomorphic to $V^0_{-1}\oplus V^0_{-2}\oplus V_{-3}^0$. Therefore, the sign of $L_{3,1}$ is $(-)$. Inside of $\Gr_2(V)$, by Proposition \ref{ExceptInFlags}, there is a $\mathbb{CP}^1$ of weight $3$ connecting $L_{3,1}$ and $L_{1,-3}$ and a $\mathbb{CP}^1\times\mathbb{CP}^1$ of weight $2$ connecting $L_{3,1}$, $L_{3,-3}$, $L_{3,-3}$, and $L_{-1,-3}$. Thus, inside of $\Lag(V,\omega)$, there is a $\mathbb{CP}^1$ of weight $3$ connecting $L_{3,1}$ to $L_{-3,1}$ and a $\mathbb{CP}^1$ of weight $2$ connecting $L_{3,1}$ to $L_{-1,-3}$. For the other fixed Lagrangian subspaces, we have preimage and $\PSOrth(2)$--difference matrix
\begin{align*}
    \varphi_{-1,-3}^{-1}(\Lag(V,\omega))&=\left\{\begin{bmatrix}
        0 & 0 & 0 &0  \\ 0 & 0 & 0 & 0\\ a & b & 0 & 0\\ c & d & 0 & 0
    \end{bmatrix}:4a+d=0\right\}, &
    \begin{tabular}{c|cccc}
             $L_{-1,-3}$ & $f_{-1}$ & $f_{-3}$ & $f_{3}$ & $f_{1}$ \\ \hline
             $f_{-1}$ & $\ast$ & $\ast$  & $\ast$ & $\ast$\\
             $f_{-3}$ & $\ast$ & $\ast$  & $\ast$ & $\ast$\\
             $f_{3}$ & $2$ & $3$  & $\ast$  & $\ast$\\
             $f_{1}$ & $1$ & $2$ & $\ast$ & $\ast$
        \end{tabular}\\
    \varphi_{1,-3}^{-1}(\Lag(V,\omega))&=\left\{\begin{bmatrix}
        0 & 0 & 0 &0  \\ 0 & 0 & 0 & 0\\ a & b & 0 & 0\\ c & d & 0 & 0
    \end{bmatrix}:4a-d=0\right\}, &
    \begin{tabular}{c|cccc}
             $L_{1,-3}$ & $f_{1}$ & $f_{-3}$ & $f_{3}$ & $f_{-1}$ \\ \hline
             $f_{1}$ & $\ast$ & $\ast$  & $\ast$ & $\ast$\\
             $f_{-3}$ & $\ast$ & $\ast$  & $\ast$ & $\ast$\\
             $f_{3}$ & $1$ & $3$  & $\ast$  & $\ast$\\
             $f_{-1}$ & $-1$ & $1$ & $\ast$ & $\ast$
        \end{tabular} \\
    \varphi_{3,-1}^{-1}(\Lag(V,\omega))&=\left\{\begin{bmatrix}
        0 & 0 & 0 &0  \\ 0 & 0 & 0 & 0\\ a & b & 0 & 0\\ c & d & 0 & 0
    \end{bmatrix}:a+4d=0\right\}, &
    \begin{tabular}{c|cccc}
             $L_{3,-1}$ & $f_{3}$ & $f_{-1}$ & $f_{1}$ & $f_{-3}$ \\ \hline
             $f_{3}$ & $\ast$ & $\ast$  & $\ast$ & $\ast$\\
             $f_{-1}$ & $\ast$ & $\ast$  & $\ast$ & $\ast$\\
             $f_{1}$ & $-1$ & $1$  & $\ast$  & $\ast$\\
             $f_{-3}$ & $-3$ & $-1$ & $\ast$ & $\ast$
        \end{tabular}
\end{align*}
Thus, the signs of $L_{-1,-3}$ and $L_{3,-1}$ is $(+)$ and the sign of $L_{1,-3}$ is $(-)$. There is also a $\mathbb{CP}^1$ of weight $3$ connecting $L_{-1,-3}$ and $L_{3,-1}$. This gives the described tangential weight graph. 
\end{proof}
From Corollary \ref{ExceptionalHirTWG} and Theorem \ref{SmoothClassification}, we conclude
\begin{corr}\label{fiberlagirr}
    For $\iota_{\bldmth{l},\bldmth{i}}=\rho_4$ and $\mathcal{F}_\eta=\Lag(\mathbb{C}^4)$,
    \begin{enumerate}
        \item[(i)] The $\PSL(2,\mathbb{R})$--equivariant fiber bundle projection $p_{\bldmth{l},\bldmth{i}}:\Omega_{\bldmth{l},\bldmth{i}}\to\mathbb{H}^2$ has fiber $M_{\bldmth{l},\bldmth{i}}$ which is equivariantly diffeomorphic to $\Hir(1;-1,3)$, identifying $S^1\cong \PSOrth(2)$. In particular, $M_{\bldmth{l},\bldmth{i}}$ is diffeomorphic to $\mathbb{CP}{}^2\#\overline{\mathbb{CP}}{}^{2}$.
        \item[(ii)] The manifold $\mathcal{W}_{\bldmth{l},\bldmth{i}}$ is diffeomorphic to the bundle associated to the $\PSOrth(2)$--principal bundle over $S_g$ of Euler class $g-1$ and fiber $\mathbb{CP}{}^2\#\overline{\mathbb{CP}}{}^2$ with the described $\PSOrth(2)$--action.
    \end{enumerate}
\end{corr}
\subsubsection{The reducible case}
Now, we will determine the topology and structure group action for $M_{\bldmth{l},\bldmth{r}}$. As for $M_{\bldmth{l},\bldmth{r}}$, there will be fixed surfaces that we handle similarly. \par
Let $V$ be the $\mathbb{C}$--span of $X_1$, $Y_1$, $X_2$, $Y_2$ inside $\mathbb{C}[X_1,Y_1,X_2,Y_2]$. We represent $\iota_{\bldmth{l},\bldmth{r}}=\rho_2\oplus\rho_1\oplus\rho_1$ by the linear action on $V$ given as
\[\iota_{\bldmth{l},\bldmth{r}}\left(\begin{bmatrix}
    a & b \\ c & d
\end{bmatrix}\right)\cdot f(X_1,Y_1,X_2,Y_2)=f(dX_1-cY_1,aY_1-bX_1,X_2,Y_2)\]
By restricting $\iota_{\bldmth{l},\bldmth{r}}$ to $\SOrth(2)$, we get a linear $\SOrth(2)$--action on $V$. The weight vectors of $V$ are $f_1,=X_1-iY_1$, $f_{-1}=X_{1}+iY_1$, $X_2$ and $Y_2$ with weights
\begin{align*}
    \iota_{\bldmth{l},\bldmth{r}}(R_{\theta})\cdot f_1&= e^{i\theta}f_1 & \iota_{\bldmth{l},\bldmth{r}}(R_{\theta})\cdot f_{-1}&=e^{-i\theta}f_{-1}\\
    \iota_{\bldmth{l},\bldmth{r}}(R_\theta)\cdot X_2&=X_2 & \iota_{\bldmth{l},\bldmth{r}}(R_{\theta})\cdot Y_2&=Y_2
\end{align*}
Using the basis $f_{1},f_{-1},X_2,Y_2$, equip $V$ with the symplectic form
\[\omega=\begin{bmatrix}
    & -1 & & \\ 1 & & &\\ & &  & -1\\ & & 1 &
\end{bmatrix}.\]
In other words, the symplectic forms on $\langle f_1,f_{-1}\rangle$ and $\langle X_2,Y_2\rangle$ are both identified with the standard symplectic form on $\mathbb{C}^2$. Because $\Sp(\mathbb{C},2)=\SL(2,\mathbb{C})$, we see that $\iota_{\bldmth{l},\bldmth{r}}(\SL(2,\mathbb{R}))$ preserves $\omega$. Therefore, we have $\iota_{\bldmth{l},\bldmth{r}}:\SL(2,\mathbb{R})\to\Sp(V,\omega)\cong \Sp(4,\mathbb{C})$. \par
Now, let us determine the Lagrangian subsapces of $V$ which are fixed by the action of $\SOrth(2)$.
\begin{prop}
    Let $L\in\Fix(\Lag(V,\omega),0,\SOrth(2))$ Then, 
    \begin{itemize}
        \item either $L=\langle f_1,\alpha X_2+\beta Y_2\rangle$ for some $[\alpha:\beta]\in\mathbb{CP}^1$,
        \item or $L=\langle f_{-1},\alpha X_2+\beta Y_2\rangle$ for some $[\alpha:\beta]\in\mathbb{CP}^1$.
    \end{itemize}
\end{prop}
\begin{proof}
    Let $P_1=\langle X_1,Y_1\rangle$ and $P_{2}=\langle X_2,Y_2\rangle$. Note that the weight spaces for $V$ are $\langle f_1\rangle$, $\langle f_{-1}\rangle$ and $P_2$. Because $L$ is $\SOrth(2)$--invariant it is a weight submodule. Therefore,
    \[L=(L\cap \langle f_1\rangle)\oplus (L\cap \langle f_{-1}\rangle)\oplus (L\cap P_2)\]
    Because neither $P_1=\langle f_1,f_{-1}\rangle$ nor $P_2$ are isotropic, it must be that either $L=(L\cap P_1)\oplus (L\cap P_2)$. Note that $L\cap P_1$ is an $\SOrth(2)$--invariant line in $P_1$. Because the weights for the $\SOrth(2)$--action on $P_1$ are distinct, the only $\SOrth(2)$--invariant lines are the weight spaces $\langle f_1\rangle$ and $\langle f_{-1}\rangle$. Therefore, either $L=\langle f_1\rangle\oplus L\cap P_2$ or $L=\langle f_{-1}\rangle\oplus L\cap P_2$. Finally, we note that any line in $P_2$ is $\SOrth(2)$--invariant and $\omega$--orthogonal to $P_1$. Therefore, any choice of line $\langle\alpha X_2+\beta Y_2\rangle$ works as $L\cap P_2$.
\end{proof}
Now, to determine the tangential weight graph, we just need to calculate the Euler classes of these fixed surfaces. We will need the following useful fact about Chern classes for projective varieties.
\begin{prop}[\text{\cite[Ex.\ 5.19]{Eisenbud_Harris_2016}}]
    Let $\mathbb{P}^n$ be a complex projective space and Let $\alpha$ be the Poincar\'e dual of a hyperplane in $\mathbb{P}^n$. Consider $Z_1,\hdots,Z_k$ algebraic hypersurfaces in $\mathbb{P}^n$ with respective degrees $d_1,\hdots,d_k$. Assume that the intersection
    \[X=Z_1\cap\cdots \cap Z_k\]
    is a complete intersection.
    Then, the first Chern class of the tangent bundle of $X$ is
$$\left(n+1-\sum_{i=1}^kd_i\right)\alpha|_{X}.$$
\end{prop}
\begin{prop}\label{Case6TWG}
    The $\SOrth(2)$--action on $M_{\bldmth{l},\bldmth{r}}$ is effective and the tangential weight graph looks as in Figure \ref{fig:case6TWG}
    \begin{figure}
        \centering
        \begin{tikzpicture}[node distance={30mm}, main/.style = {draw, circle}]
\node[draw] (1) at (0,-3.6) {$1$}; 
\node[draw] (4) at (2,-3.6)  {$-1$};

\end{tikzpicture}
        \caption{Tangential weight graph for $\SOrth(2)$--action on $M_{\bldmth{l},\bldmth{r}}$}
        \label{fig:case6TWG}
    \end{figure}
\end{prop}
\begin{proof}
    Fix a Hermitian metric on $\Lag(V,\omega)$.
    Let $L\in C=\{\langle f_1,\alpha X_2+\beta Y_2\rangle:[\alpha:\beta]\in\mathbb{CP}^1\}$. Choose some $v_2$ such that $v_1,v_2$ is a basis for $\langle X_2,Y_2\rangle$. Then, using the basis $f_1,v_1,f_{-1},v_2$, $L$ has the $\SOrth(2)$--difference matrix
    \[\begin{tabular}{c|cccc}
             $L$ & $f_{1}$ & $v_1$ & $f_{-1}$ & $v_2$ \\ \hline
             $f_{1}$ & $\ast$ & $\ast$  & $\ast$ & $\ast$\\
             $v_1$ & $\ast$ & $\ast$  & $\ast$ & $\ast$\\
             $f_{-1}$ & $-2$ & $-1$  & $\ast$  & $\ast$\\
             $v_2$ & $-1$ & $0$ & $\ast$ & $\ast$
        \end{tabular}\]
        The $-2$ entry corresponds to $T_L^\perp M_{\bldmth{l},\bldmth{r}}\cong T_O\mathbb{H}^2$. Thus, the $\SOrth(2)$--action on $T_{L}M_{\bldmth{l},\bldmth{r}}$ has weights $-1$ and $0$ for each $L\in C$. Thus, we conclude that the $\SOrth(2)$--action on $M_{\bldmth{l},\bldmth{r}}$ is effective. Because $M_{\bldmth{l},\bldmth{r}}$ only contains two fixed surfaces, its tangential weight graph must have two square vertices with labels $e$ and $-e$ for some $e\in\mathbb{Z}$.
        As in Proposition \ref{Case4TWG}, the fact that the weights are only either $-1$ or $0$ means $TM_{\bldmth{l},\bldmth{r}}$ is a complex subbundle of $T\Lag(V,\omega)|_{M_{\bldmth{l},\bldmth{r}}}$.
        The absolute value of $e$ will agree with the absolute value of the Chern number of $T^\perp C\cap TM_{\bldmth{l},\bldmth{r}}$.\\
        First, we consider $\Gr_2(V)$ as a variety in $\mathbb{P}(\wedge^2V)$ with the Pl\"ucker relation
        \[x_{12}x_{34}-x_{13}x_{24}+x_{14}x_{23}=0\]
        where $x_{ij}$ is the homogeneous coordinate corresponding to $x_i\wedge x_j$ where $x_1,x_2,x_3,x_4$ is our original basis $f_1$, $f_{-1}$, $X_2$, $Y_2$. Then, $\Lag(V,\omega)$ is a subvariety of $\Gr_2(V)$ produced by adding the relation
        \[x_{12}+x_{34}=0\]
        Furthermore, $C$ is a projective line in $\mathbb{P}(\wedge^2V)$ spanned by $f_1\wedge X_2$ and $f_1\wedge Y_2$. Therefore, we have an inclusion of subvarieties
        \[C\subseteq \Lag(V,\omega)\subseteq \Gr_2(V)\subseteq \mathbb{P}(\wedge^2 V).\]
        Note that since $\Lag(V,\omega)$ is a subvariety defined by system of equations with a degree $2$ relation and a degree $1$ relation, its tangent bundle has first Chern class 
        \[(6-2-1)\alpha|_{\Lag(V,\omega)}=3\alpha|_{\Lag(V,\omega)}\]
        where $\alpha$ is the Poincare dual of a hyperplane, and the generator of $H^2(\mathbb{P}(\wedge^2V),\mathbb{Z})$. Also, the isomorphism $H^2(\mathbb{P}(\wedge^2V),\mathbb{Z})\cong H^2(C,\mathbb{Z})$ means $\alpha|_C$ generates $H^2(C,\mathbb{Z})$. Thus,
        \begin{align*}
            c_1(T^\perp C)&=c_1(T\Lag(V,\omega)|_C)-c_1(TC)=3\alpha|_C-2\alpha|_C=\alpha|_C
        \end{align*}
        Because $T^\perp M_{\bldmth{l},\bldmth{r}}$ is a trivial subbundle of $T\Lag(V,\omega)|_{M_{\bldmth{l},\bldmth{r}}}$, we conclude that
        \[c_1(T^\perp C\cap TM_{\bldmth{l},\bldmth{r}})=3\alpha|_C.\]
        Therefore, the absolute value of the Euler class of $T^\perp C\cap TM_{\bldmth{l},\bldmth{r}}$ is $1$. This gives the described tangential weight graph.
\end{proof}
Note that this tangential weight graph agrees with the one described for $\Hir(1;1,0)$ in Proposition \ref{FixedHirTWG}. Thus, by Theorem \ref{SmoothClassification}, we conclude the following.
\begin{corr}\label{fiberlagred}
    For $\iota_{\bldmth{l},\bldmth{r}}=\rho_2\oplus \rho_1\oplus \rho_1$ and $\mathcal{F}_\eta=\Lag(\mathbb{C}^4)$,
    \begin{enumerate}
        \item[(i)] The $\SL(2,\mathbb{R})$--equivariant fiber bundle projection $p_{\bldmth{l},\bldmth{r}}:\Omega_{\bldmth{l},\bldmth{r}}\to\mathbb{H}^2$ has fiber $M_{\bldmth{l},\bldmth{r}}$ is equivariantly diffeomorphic to $\Hir(1;1,0)$, identifying $S^1\cong \SOrth(2)$. In particular, $M_{\bldmth{l},\bldmth{r}}$ is diffeomorphic to $\mathbb{CP}{}^2\#\overline{\mathbb{CP}}{}^2$.
        \item[ii] The manifold $\mathcal{W}_{\bldmth{l},\bldmth{r}}$ is diffeomorphic to the bundle associated to the $\SOrth(2)$--principal bundle over $S_g$ of Euler class $2g-2$ and fiber $\mathbb{CP}{}^2\#\overline{\mathbb{CP}}{}^2$ with the described $\SOrth(2)$--action.
    \end{enumerate}
\end{corr}
\subsection{Summary}
We conclude by summarizing the results of Corollaries \ref{fiberflgirr}, \ref{fiberflgred}, \ref{fiberprjirr}, \ref{fiberprjred}, \ref{fiberlagirr}, \ref{fiberlagred} in this final theorem.
\begin{thm}
    For $(G,\mathcal{F}_\eta,\iota)$ listed as entries in the table from Proposition \ref{fullcases}, the corresponding domain of disconinuity $\Omega_\rho ^I$ has a fiber $M_\rho ^I$ with an effective structure group $Q$ and tangetial weight graph for the $Q$--action as follows:
    \begin{enumerate}
        \item For $G=\SL(3,\mathbb{C})$, $\mathcal{F}_\eta=\Flag(\mathbb{C}^3)$, then $M_\rho^I$ is diffeomorphic to $(S^2\times S^2)\#(S^2\times S^2)$. Furthermore,
        \vspace{-0.3pt}
        \begin{itemize}
            \item if $\iota=\rho_3$, then $Q=\PSOrth(2)$ and the tangential weigh graph is
            \vspace{3pt}
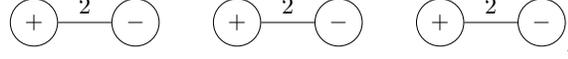
            \begin{center}
                \scalebox{0.9}{\begin{tikzpicture}[node distance={30mm}, main/.style = {draw, circle}] 

\node[main] (1) at (0,0) {$+$}; 
\node[main] (2) at (1.5,0) {$-$}; 
\node[main] (3) at (3,0) {$+$}; 
\node[main] (4) at (4.5,0) {$-$}; 
\node[main] (5) at (6,0) {$+$}; 
\node[main] (6) at (7.5,0) {$-$}; 
\draw[] (1) -- node[midway, above,] {$2$} (2);
\draw[] (3) -- node[midway, above,] {$2$} (4);
\draw[] (5) -- node[midway, above,] {$2$} (6);
\end{tikzpicture}},
            \end{center}
            \vspace{3pt}
        \item if $\iota=\rho_2\oplus\rho_1$, then $Q=\SOrth(2)$ and the tangential weigh graph is 
        \vspace{3pt}
        \begin{center}
        {\scalebox{0.9}{\begin{tikzpicture}[node distance={30mm}, main/.style = {draw, circle}] 

\node[main] (1) at (0,0) {$+$}; 
\node[main] (2) at (1.5,0) {$-$}; 
\node[main] (3) at (3,0) {$+$}; 
\node[main] (4) at (4.5,0) {$-$}; 
\node[main] (5) at (6,0) {$+$}; 
\node[main] (6) at (7.5,0) {$-$}; 
\end{tikzpicture}}.}
\end{center}
\vspace{3pt}
        \end{itemize}
        \item For $G=\SL(4,\mathbb{C})$, $\mathcal{F}_\eta=\mathbb{CP}^3$, then $M_\rho^I$ is diffeomorphic to $S^2\times S^2$. Furthermore,
        \vspace{-0.3pt}
        \begin{itemize}
            \item if $\iota=\rho_3$, then $Q=\PSOrth(2)$ and the tangential weigh graph is 
            \vspace{3pt}
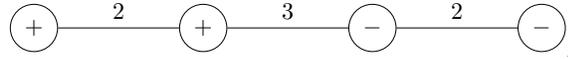
            \begin{center}
                
            \scalebox{0.9}{\begin{tikzpicture}[node distance={30mm}, main/.style = {draw, circle}] 

\node[main] (1) at (0,-2.4) {$+$}; 
\node[main] (2) at (2.5,-2.4) {$+$};
\node[main] (3) at (5,-2.4) {$-$};
\node[main] (4) at (7.5,-2.4) {$-$};
\draw[] (1) -- node[midway, above,] {$2$} (2);
\draw[] (2) -- node[midway, above,] {$3$} (3);
\draw[] (3) -- node[midway, above,] {$2$} (4);
\end{tikzpicture}},
\end{center}
\vspace{3pt}
        \item if $\iota=\rho_2\oplus\rho_2$, then $Q=\PSOrth(2)$ and the tangential weigh graph is 
        \vspace{3pt}
        \begin{center}
            
        \scalebox{0.9}{\begin{tikzpicture}[node distance={30mm}, main/.style = {draw, circle}] 

\node[draw] (1) at (2.5,-3.6) {$2$}; 
\node[draw] (4) at (5,-3.6) {$-2$};
\end{tikzpicture}}.
\end{center}
\vspace{3pt}
        \end{itemize}
        \item For $G=\Sp(4,\mathbb{C})$, $\mathcal{F}_\eta=\Lag(\mathbb{C}^4)$, then $M_\rho^I$ is diffeomorphic to $\mathbb{CP}{}^2\times \overline{\mathbb{CP}}{}^2$. Furthermore,
        \vspace{-0.3pt}
        \begin{itemize}
            \item if $\iota=\rho_3$, then $Q=\PSOrth(2)$ and the tangential weigh graph is 
            \vspace{3pt}
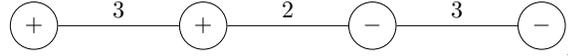
            \begin{center}
                
            \scalebox{0.9}{\begin{tikzpicture}[node distance={30mm}, main/.style = {draw, circle}] 

\node[main] (1) at (0,-2.4) {$+$}; 
\node[main] (2) at (2.5,-2.4) {$+$};
\node[main] (3) at (5,-2.4) {$-$};
\node[main] (4) at (7.5,-2.4) {$-$};
\draw[] (1) -- node[midway, above,] {$3$} (2);
\draw[] (2) -- node[midway, above,] {$2$} (3);
\draw[] (3) -- node[midway, above,] {$3$} (4);
\end{tikzpicture}},
\end{center}
\vspace{3pt}
        \item if $\iota=\rho_2\oplus\rho_1\oplus\rho_1$, then $Q=\SOrth(2)$ and the tangential weigh graph is 
        \vspace{3pt}
        \begin{center}
            
        \scalebox{0.9}{\begin{tikzpicture}[node distance={30mm}, main/.style = {draw, circle}] 

\node[draw] (1) at (2.5,-3.6) {$1$}; 
\node[draw] (4) at (5,-3.6) {$-1$};
\end{tikzpicture}}.
\end{center}
\vspace{3pt}
        \end{itemize}
    \end{enumerate}
\end{thm}

\printbibliography
\end{document}